\documentclass[11pt]{amsart}
\usepackage{amsfonts,latexsym,amsthm,amssymb,amsmath,amscd,euscript,tikz, tikz-cd}
\usepackage[alphabetic, msc-links, bibtex-style, nobysame]{amsrefs}
\usepackage{stackengine}
\usepackage{framed}
\usepackage{xfrac}
\usepackage[makeroom]{cancel}
\usepackage{faktor}
\usepackage{braket}
\usepackage{pgf,tikz,pgfplots}
\pgfplotsset{compat=1.17}
\usepgfplotslibrary{fillbetween} 
\usepackage{mathrsfs}
\usetikzlibrary{arrows}
\definecolor{wrwrwr}{rgb}{0.3803921568627451,0.3803921568627451,0.3803921568627451}
\definecolor{rvwvcq}{rgb}{0.08235294117647059,0.396078431372549,0.7529411764705882}
\definecolor{mblue}{rgb}{0.2, 0.3, 0.8}
\definecolor{morange}{rgb}{1, 0.5, 0}
\definecolor{mgreen}{rgb}{0.1, 0.4, 0.2}
\definecolor{mred}{rgb}{0.5, 0, 0}
\definecolor{ForestGreen}{RGB}{34,139,34}
\usepackage{hyperref}
\hypersetup{colorlinks=true,citecolor=ForestGreen,linkcolor = blue,urlcolor =black,linkbordercolor={1 0 0}}
\usepackage{float}

\numberwithin{equation}{section}

\usepackage{lmodern}
\usepackage{supertabular}
\usepackage{verbatim}
\usepackage{amssymb}
\usepackage{enumerate}
\usepackage{stmaryrd}
\usepackage{bbm}
\usepackage{mathtools}
\usepackage[nopatch=footnote]{microtype}
\usepackage{setspace}
\usepackage{tikz,tikz-cd}
\usetikzlibrary{matrix,calc,positioning,arrows,decorations.pathreplacing,patterns,knots}

\newcommand{\la}{\langle}
\newcommand{\rg}{\rangle}

\newtheorem{theorem}{{Theorem}}[section]
\newtheorem*{theorem*}{Theorem}
\newtheorem{lemma}[theorem]{Lemma}
\newtheorem{proposition}[theorem]{Proposition}

\newtheorem*{corollary*}{Corollary}

\theoremstyle{definition}
\newtheorem{definition}{Definition}
\newtheorem{remark}{Remark}

\usepackage{lmodern,url,enumerate,mathtools,microtype}
\usepackage[hmargin = 1in,vmargin=1in]{geometry}
\usepackage{graphicx}
\usepackage{subcaption}

\newcommand{\ve}{\varepsilon}

\newcommand{\mr}[1]{{\rm #1}}
\newcommand{\cA}{\mathcal{A}}

\newcommand{\cH}{\mathcal{H}}
\newcommand{\cJ}{\mathcal{J}}
\newcommand{\cL}{\mathcal{L}}
\newcommand{\cM}{\mathcal{M}}
\newcommand{\cO}{\mathcal{O}}\newcommand{\cP}{\mathcal{P}}

\newcommand{\cW}{\mathcal{W}}

\newcommand{\bC}{\mathbb{C}}\newcommand{\bD}{\mathbb{D}}
\newcommand{\bF}{\mathbb{F}}
\newcommand{\bG}{\mathbb{G}}\newcommand{\bH}{\mathbb{H}}

\newcommand{\bN}{\mathbb{N}}
\newcommand{\bO}{\mathbb{O}}\newcommand{\bP}{\mathbb{P}}
\newcommand{\bQ}{\mathbb{Q}}\newcommand{\bR}{\mathbb{R}}
\newcommand{\bS}{\mathbb{S}}\newcommand{\bT}{\mathbb{T}}
\newcommand{\bV}{\mathbb{V}}

\newcommand{\bZ}{\mathbb{Z}}

\newcommand{\fg}{\mathfrak{g}}

\newcommand{\fk}{\mathfrak{k}}

\newcommand{\fp}{\mathfrak{p}}

\newcommand{\nc}{\newcommand}

\nc{\on}{\operatorname}
\nc{\p}{\partial}
\nc{\ol}{\overline}
\nc{\ul}{\underline}
\nc{\pa}{\partial}

\nc{\pb}{\partial_b}
\nc{\pc}{\partial_c}
\nc{\pd}{\partial_d}
\nc{\pe}{\partial_e}
\nc{\pf}{\partial_f}
\nc{\pg}{\partial_g}
\nc{\ph}{\partial_h}
\nc{\pari}{\partial_i}
\nc{\pj}{\partial_j}
\nc{\pk}{\partial_k}
\nc{\pl}{\partial_l}
\nc{\pell}{\partial_\ell}
\nc{\parm}{\partial_m}
\nc{\pn}{\partial_n}
\nc{\po}{\partial_o}
\nc{\pp}{\partial_p}
\nc{\pq}{\partial_q}
\nc{\pr}{\partial_r}
\nc{\ps}{\partial_s}
\nc{\pt}{\partial_t}
\nc{\pu}{\partial_u}
\nc{\pv}{\partial_v}
\nc{\pw}{\partial_w}
\nc{\px}{\partial_x}
\nc{\py}{\partial_y}
\nc{\pz}{\partial_z}

\nc{\Spec}{\on{Spec}}

\nc{\sn}{\mr{sn}}
\nc{\cn}{\mr{cn}}
\nc{\dn}{\mr{dn}}

\nc{\thru}{~\!\!--~\!\!}

\numberwithin{equation}{section}
\makeatletter
\@addtoreset{equation}{section}
\makeatother

\allowdisplaybreaks

\title[Topology of minimal surfaces in the sphere from capillarity]{Topology of minimal surfaces in the sphere from capillarity}
\date{\today}

\author[Benjy Firester]{Benjy Firester}
\address{Department of Mathematics, MIT \newline 
{\href{mailto:benjyfir@mit.edu}{benjyfir@mit.edu}}}
\author[Raphael Tsiamis]{Raphael Tsiamis}
\address{Department of Mathematics, Columbia University \newline
{\href{mailto:r.tsiamis@columbia.edu}{r.tsiamis@columbia.edu}}}

\begin{document}

\begin{abstract}
We present a general construction of embedded minimal and constant mean curvature surfaces in $\mathbb{S}^n$ and one-phase free boundaries joined by a smooth interpolation by capillary hypersurfaces. This framework recovers all known families and produces new minimal surfaces in the sphere with rich topological structures as sphere bundles over base spaces which include space-form products, projective planes over division algebras, Stiefel manifolds, complex quadrics, and twisted products and quotients of Lie subgroups of $\textup{SO}(n)$. We show these bundles are non-trivial and study their homotopy types using topological obstructions, including characteristic classes and tools from $K$-theory and stable homotopy theory. Finally, we prove uniqueness results for the rotationally invariant capillary CMC problem. 
\end{abstract}

\maketitle

\vspace{-0.5cm}

\section{Introduction} 

We present a general framework to construct minimal and constant mean curvature (CMC) surfaces in the sphere that belong to smooth families of capillary surfaces and exhibit rich topologies.
For small capillary angles, these interpolating families converge to the novel one-phase singularity models constructed in the companion paper~\cite{one-phase-isoparametric}.
The resulting minimal hypersurfaces are realized as sphere bundles over distinguished high codimension submanifolds in the sphere, which include projective planes over the division algebras, Stiefel manifolds, complex quadrics, and homogeneous symmetric spaces. 
Our construction employs a free boundary ODE, producing effectively computable solutions with precise quantitative descriptions while displaying significant topological complexity.
Finally, we prove uniqueness properties for the capillary CMC problem.

The new surfaces we construct address key questions in topology and minimal surface theory.
Distinguishing fiber bundles is a fundamental and difficult problem in topology that has influenced the creation of several disciplines and techniques, including $K$-theory, homotopy theory, and spectral sequences. 
Major accomplishments, from Milnor’s discovery of exotic spheres as $\bS^3$-bundles over $\bS^4$~\cites{james-whitehead, milnor} to the James–Whitehead classification of sphere bundles by homotopy-theoretic methods, have far-reaching applications outside topology.
The geometry of sphere bundles has received contributions from Chern, do Carmo, Kobayashi, Hsiang, Wallach, and other mathematicians, with recent results advancing this topic further~\cites{ chern-spanier , kobayashi , wallach-docarmo , kassel}.
A central question is the embeddability of sphere bundles into homogeneous spaces, notably the round sphere $\bS^n$, as well as the realization of Lie group and homogeneous space products as bundles; notably, Hsiang-Szczarba studied the embedding problem for sphere bundles over spheres~\cite{hsiang-embeddability}.
In our setting, the base manifolds exhibit a wealth of topological behaviors, and the question of realizing sphere bundles as minimal or CMC hypersurfaces becomes significantly more challenging.
We construct such embeddings for general families of sphere bundles using foliations of the sphere and identify their diffeomorphism through topological ideas such as characteristic classes, cohomology operations, $K$-theory, and methods from stable homotopy theory, including the Adams $J$-homomorphism.

The study of minimal surfaces in the sphere, equivalently minimal cones, is a classical topic in geometric analysis and the regularity theory of area-minimizing submanifolds.
A fundamental problem, going back to Hsiang-Lawson~\cite{hsiang-lawson-jdg}*{Ch.~III \S~3} and Hsiang-Hsiang~\cite{hsiang-hsiang}*{Problem~3}, is to understand the landscape of $G$-invariant minimal hypersurfaces in the sphere.
Our framework advances this program beyond the low cohomogeneity context by producing a rich package of minimal, capillary, and one-phase geometries associated with foliations of the sphere.

Our construction utilizes the theory of isoparametric hypersurfaces in $\bS^n$, namely hypersurfaces with constant principal curvatures.
Isoparametric foliations greatly generalize the equivariant framework proposed by Hsiang and Lawson~\cite{hsiang-lawson-jdg}, and their classification has been an important problem in differential geometry with contributions of many authors~\cites{Cartan1939 , dorfmeister-neher , grove-halperin , stolz, thorbergsson, cecil-chi-jensen , chi-jdg-III, chi-jdg-IV , miyaoka , siffert-AGAG }.
Every isoparametric foliation of the sphere comes with two distinguished focal submanifolds $M_1, M_2$ of codimensions $m_1+1$ and $m_2+1$, together with a family of regular leaves $M_s$ between them having $g$ distinct principal curvatures, where $g \in \{1,2,3,4, 6\}$.
Our theorems construct minimal and CMC surfaces as sphere bundles over the focal as well as the regular leaves, leading to a wide range of topological types.
In Section~\ref{subsec:isoparametric}, we provide further background on isoparametric foliations and their topology.

\begin{theorem}\label{thm:new-minimal-surfaces}
    For any isoparametric hypersurface $M \subset \bS^{n-1}$ with at least two distinct principal curvatures, let $M_1, M_2$ be the associated focal manifolds.
    There exist three closed embedded minimal surfaces $\mathbf{S}_{M_1}, \mathbf{S}_{M_2}, \bar{\mathbf{S}}_M \subset \bS^n$ of two classes with the following properties.
    \begin{enumerate}
        \item[\textup{I}.] The surfaces $\mathbf{S}_{M_i}$ are minimal embeddings of the sphere bundle $S(\nu_{M_i} \oplus \mathbf{1})$ over $M_i$, where $\nu_{M_i} \to M_i$ denotes the rank $m_i+1$ normal bundle of the embedding $M_i \hookrightarrow \bS^{n-1}$.
        \item[\textup{II}.] The surface $\bar{\mathbf{S}}_M$ is a minimal embedding of $M \times \bS^1$ in the sphere $\bS^n$.
    \end{enumerate}
\end{theorem}
In Section~\ref{section:topology}, we characterize the topology of the surfaces $\mathbf{S}_{M_i}$ and $\bar{\mathbf{S}}_M \cong M \times \bS^1$.
For the sphere bundles $\mathbf{S}_{M_i}$, we obtain the following complete classification result.

\begin{theorem}\label{thm:non-triviality-theorem}
    The hypersurfaces $\mathbf{S}_{M_i} \subset \bS^n$ exhibit the following product and non-product behaviors:
    \begin{enumerate}[(i)]
        \item If $g \in \{1,2\}$, then $\mathbf{S}_{M_i}$ is a Clifford torus $\bS^{n-k-1} \times \bS^k$ for $1 \leq k \leq n-2$.
        \item If $g \geq 3$ and $M$ is not an OT-FKM isoparametric hypersurface, then $\mathbf{S}_{M_i}$ is not stably homotopy equivalent to a product $M_i \times \bS^r$. 
        In particular, it is not homotopy equivalent to any $M_i \times \bS^r$.
        \item Finally, if $M$ is an OT-FKM hypersurface with multiplicities $(m_1, m_2)$, then $\mathbf{S}_{M_1} \cong M_1 \times \bS^{m_1+1}$ always.
        For $M_2$, either $\mathbf{S}_{M_2}$ is not homotopy equivalent to any product $M_2 \times \bS^r$, or $M_2 \cong \bS^{m_1} \times \bS^{m_1 + m_2}$ and $\mathbf{S}_{M_2} \cong \bS^{m_1} \times \bS^{m_1+m_2} \times \bS^{m_2+1}$ as smooth sphere bundles, when
\[
(m_1,m_2)\in\{ (1, 2d) , (2d,1) , (2,5),(5,2),(3,4) , (4,3) \}, \qquad d \in \bN^*.
\]
        \end{enumerate}
    Furthermore, the homotopy types of the minimal surfaces $\mathbf{S}_{M_i}$ are characterized as follows:
    \begin{enumerate}[$(a)$]
        \item  For $g = 3$, the sphere bundles $\mathbf{S}_{\bF} := \mathbf{S}_{M_1} \cong \mathbf{S}_{M_2}$ are twisted products of Lie groups.
    In particular, $\mathbf{S}_{\bR} \cong \bS^2 \times \bS^2 / \sigma$ is the quotient space of an involution $\sigma$.
    \item 
    For isoparametric foliations with $g \in \{ 4, 6 \}$ except for the one coming from the isotropy representation of $G_2 / \textup{SO}(4)$, the minimal surfaces $\mathbf{S}_{M_1}$ and $\mathbf{S}_{M_2}$ are not homotopy equivalent.
    In the latter case, the surfaces $\mathbf{S}_{M_1}$ and $\mathbf{S}_{M_2}$ are diffeomorphic but not isometric.
    \end{enumerate}
\end{theorem}
The results $(ii)$ and $(iii)$ imply that $\mathbf{S}_{M_i}$ are non-trivial as sphere bundles, and are not even homeomorphic to products over $M_i$.
A more precise version of this result, together with the topological methods used to obtain it, is presented in Section~\ref{section:topology}.
We refer the reader to Propositions~\ref{prop:all-nontrivial-isoparametric} and~\ref{prop:trivial-product-bundles-ot-fkm} for the non-trivial homotopy type of the surfaces $\mathbf{S}_{M_i}$, Proposition~\ref{prop:g=3Topology} for the characterization of $\mathbf{S}_{\bF} := \mathbf{S}_{M_i}$ when $g=3$, and Proposition~\ref{prop:surfaces-are-distinct} for the distinct homotopy types of the $\mathbf{S}_{M_i}$.

As exemplified in Table~\ref{table:diffeomorphism-types}, the diffeomorphism types of the surfaces $\mathbf{S}_{M_i},\bar{\mathbf{S}}_{M}$ are very diverse.
Among the Type I examples are twisted products of exceptional Lie groups realized as nontrivial sphere bundles over $\bF\bP^2$ for $\bF \in \{\bR,\bC,\bH,\bO\}$, as well as bundles over Stiefel manifolds, complex quadrics, and the twistor space of $G_2/\textup{SO}(4)$. 
Moreover, the Type II examples are products $M \times \bS^1$, so our framework encompasses multiple constructions which result in genuinely twisted bundle topologies as well as products. 
In some low-dimensional examples, corresponding to~\hyperref[thm:non-triviality-theorem]{Theorem \ref{thm:non-triviality-theorem}$(iii)$}, these bundles decompose into explicit products such as $\bS^3 \times \bS^2 \times \bS^2$ and $\bS^7 \times \bS^6 \times \bS^2$, while in the Veronese and exceptional homogeneous cases the twisting survives and can be detected by characteristic classes and stable homotopy obstructions. 

\begin{table}
\centering
\makebox[\textwidth][c]{
\begin{tabular}{|c|c|c|l|l|}
\hline
$g$ & $(m_1,m_2)$ & $\bS^n$ & Type I: $\mathbf{S}_{M_1}$, $\mathbf{S}_{M_2}$ & Type II: $\bar{\mathbf{S}}_M$ \\
\hline \hline
$1$ & $(n-2,n-2)$ & $\bS^n$ & $\bS^{n-1}$ & $\bS^{n-2} \times \bS^1$ \\
\hline
$2$ & $(p,q)$ & $\bS^{p+q+2}$ & $\bS^p \times \bS^{q+1}$, $\bS^{p+1} \times \bS^q$ & $\bS^p \times \bS^q \times \bS^1$ \\
\hline
$3$ & $(1,1)$ & $\bS^5$ & $\textup{SO}(3) \times_{\textup{O}(2)} \bS^2\,{}^{(\textcolor{blue}{\dagger})}$ & $\textup{SO}(3)/ \bZ_2^{\oplus 2} \times \bS^1$ \\
    & $(2,2)$ & $\bS^8$ & $\textup{SU}(3) \times_{\textup{U}(2)} \bS^3\,{}^{(\textcolor{blue}{\dagger})}$ & $\textup{SU}(3)/\bT^2 \times \bS^1$ \\
    & $(4,4)$ & $\bS^{14}$ & $\textup{Sp}(3) \times_{\textup{Sp}(2)\textup{Sp}(1)} \bS^5\,{}^{(\textcolor{blue}{\dagger})}$ & $\textup{Sp}(3)/\textup{Sp}(1)^3 \times \bS^1$ \\
    & $(8,8)$ & $\bS^{26}$ & $F_4 \times_{\textup{Spin}(9)} \bS^9\,{}^{(\textcolor{blue}{\dagger})}$ & $F_4/\textup{Spin}(8) \times \bS^1$ \\
\hline
$4$ & $(1,2)$ & $\bS^{8}$ & $\bS^3 \times \bS^2 \times \bS^2$, $\bS^3 \times \bS^1 \times \bS^3$ & $\bS^3 \times \bS^2 \times \bS^1 \times \bS^1$ \\
& $(1,6)$ & $\bS^{16}$ & 
$\bS^7 \times \bS^6 \times \bS^2$, $\bS^7 \times \bS^1 \times \bS^7$ & $\bS^7 \times \bS^6 \times \bS^1 \times \bS^1$ \\
 & $(1,2d-2)$ & $\bS^{4d}$ & $\bV(2, 2d) \times \bS^2, \bS^{2d-1} \times \bS^1 \times \bS^{2d-1}$ & $\bV(2,2d)\times \bS^1\times \bS^1$ \\
    & $(2,5)$ & $\bS^{16}$ & $\bS^5 \times \bS^7 \times \bS^3$, $\bS^7 \times \bS^2 \times \bS^6$ & $\bS^5 \times \bS^7 \times \bS^2 \times \bS^1$ \\
    & $(3,4)$ & $\bS^{16}$ & $\bS^4 \times \bS^7 \times \bS^4$, $\bS^7 \times \bS^3 \times \bS^5$ & $\bS^4 \times \bS^7 \times \bS^3 \times \bS^1$ \\
    & $(m, \ell\!-\!m\!-\!1)$ & $\bS^{2\ell}$ & $\bV(2,C_{m-1}) \times \bS^{m+1}$, $\bS^{\ell-m}$-bdl/$(\bS^{k\delta(m)}$-bdl/$\bS^m) \,{}^{(\textcolor{blue}{\ddagger})}$ & $\bV(2,C_{m-1}) \times \bS^m \times \bS^1$ \\
\cline{2-5}
    & $(2,2)^{\textcolor{red}{*}}$ & $\bS^{10}$ & $\bS^3$-bdl/$\bC\bP^3\,{}^{(\textcolor{blue}{\dagger})}$, $\bS^3$-bdl/$Q^3\,{}^{(\textcolor{blue}{\dagger})}$ & $ \textup{SO}(5) / \bT^2 \times \bS^1$ \\
    & $(4,5)^{\textcolor{red}{*}}$ & $\bS^{20}$ & $\bS^5$-bdl/$M_1^{14}\,{}^{(\textcolor{blue}{\dagger})}$, $\bS^6$-bdl/$M_2^{13}\,{}^{(\textcolor{blue}{\dagger})}$ & $M_{(4,4,5)} \times \bS^1$ \\
\hline
$6$ & $(1,1)$ & $\bS^8$ & $\bS^2$-bdl/$(\bS^3 \!\times\! \bR\bP^2)\,{}^{(\textcolor{blue}{\dagger})}$ & $\textup{SO}(4)/\bZ_2^{\oplus 2} \times \bS^1$ \\
    & $(2,2)$ & $\bS^{14}$ & $\bS^3$-bdl/$Q^5\,{}^{(\textcolor{blue}{\dagger})}$, $\bS^3$-bdl/$Z_{G_2}\,{}^{(\textcolor{blue}{\dagger})}$ & $G_2/ (\textup{U}(1) \times \textup{U}(1)) \times \bS^1$ \\
\hline
\end{tabular}
}
\caption{Some diffeomorphism types of the minimal surfaces $\mathbf{S}_{M_i},\bar{\mathbf{S}}_M$ from Theorem~\ref{thm:new-minimal-surfaces}.\\ 
${}^{(\textcolor{blue}{\dagger})}$: Non-trivial bundle. ${}^{(\textcolor{blue}{\ddagger})}$: Iterated non-trivial bundles depending on parameters. For $g=3$, $\mathbf{S}_{M_1} \cong \mathbf{S}_{M_2}$. $^{(\textcolor{red}{*})}$: Exceptional homogeneous cases. 
For $(g,m_1,m_2)=(4,4,5)$, $M_1^{14} \cong \textup{U}(5)/(\textup{Sp}(2) \times \textup{U}(1))$ and $M_2^{13} \cong \textup{U}(5) / ( \textup{SU}(2) \times \textup{U}(3))$, with regular leaf $M_{(4,4,5)} \cong \textup{U}(5)/ ( \textup{SU}(2) \times \textup{SU}(2) \times \textup{U}(1))$.
$Q^n$: complex quadric. $\bV(2,k)$: Stiefel manifold~\eqref{eqn:stiefel-manifold}. $\bV(2,C_{m-1})$: Clifford-Stiefel manifold~\eqref{eqn:CliffordStiefel}. 
$Z_{G_2} = G_2/\textup{U}(2)^+$: twistor space of $G_2/\textup{SO}(4)$. }
\label{table:diffeomorphism-types}
\end{table}

Another key feature of the minimal surfaces $\mathbf{S}_{M_i}$ and $\bar{\mathbf{S}}_M$ is that they belong to smooth interpolating families of capillary hypersurfaces.
\begin{theorem}\label{thm:capillary-interpolation}
    Consider an isoparametric hypersurface $M \subset \bS^{n-1}$ with at least two distinct principal curvatures and associated focal manifolds $M_1, M_2$. 
    For every $\theta \in (0,\tfrac{\pi}{2}]$, there exist capillary minimal surfaces $\Sigma_{M_1,\theta},\Sigma_{M_2,\theta}, \bar{\Sigma}_{M,\theta} \subset \bS^{n}_+$ of two classes, with contact angle $\theta$ and the following properties.
\begin{enumerate}[(i)]
    \item The surfaces $\Sigma_{M_i, \theta} \subset \bS^n_+$ are diffeomorphic to the normal disk bundles $\Sigma_{M_i,\theta} \cong D(\nu_{M_i})$.
    \item The surface $\bar{\Sigma}_{M , \theta} \subset \bS^n_+$ is diffeomorphic to $M \times [0,1]$.
\end{enumerate}
Moreover, the members of each family $\Sigma_{M_i,\theta}$ and $\bar{\Sigma}_{M,\theta}$ are ambiently isotopic to each other.
\end{theorem}
The minimal surfaces $\mathbf{S}_{M_1},\mathbf{S}_{M_2},\bar{\mathbf{S}}_M \subset \bS^{n}$ of Theorem~\ref{thm:new-minimal-surfaces} are formed by doubling the free-boundary surfaces $\Sigma_{M_1,\frac{\pi}{2}},\Sigma_{M_2,\frac{\pi}{2}},\bar{\Sigma}_{M,\frac{\pi}{2}}$ along their boundaries on the equator.
    Moreover, the cones $\mathbf{C}_{M_i, \theta} := C(\Sigma_{M_i,\theta})$ and $\bar{\mathbf{C}}_{M,\theta} := C(\bar{\Sigma}_{M,\theta})$ are capillary minimal cones in $\bR^{n+1}_+$, which admit a more precise description: they belong to smooth families that converge, after rescaling by the contact angle, to solutions $U_{M_1}$ and $U_{M_2}$ of the one-phase free boundary problem~\eqref{eqn:one-phase-problem}.
    These new singularity models of the Bernoulli problem are constructed and studied in~\cite{one-phase-isoparametric}*{Theorem 1.1}, together with one-phase solutions with disconnected spherical traces, denoted $\bar{U}_M$.
\begin{theorem}\label{thm:smooth-interpolation}
Consider an isoparametric foliation of $\bS^{n-1}$ with $g \geq 2$ principal curvatures and focal submanifolds $\{ M_1, M_2 \}$.
    There exist smooth one-parameter families of capillary cones $\{ \tilde{\mathbf{C}}_{M_i, a}\}_{a \in (0,a^*_{M_i}]}$ that interpolate between the free-boundary minimal cone over $\mathbf{S}_{M_i} \cap \bS^n_+$ and the homogeneous solution $U_{M_i}$ of the one-phase problem through cones of every angle.
    More precisely,
    \begin{enumerate}[(i)]
        \item As $a \downarrow 0$, the rescaled cones $\frac{1}{a} \tilde{\mathbf{C}}_{M_i, a}$ converge in $C^{\infty}_{\text{loc}}$ to the graph of the function $U_{M_i}$ solving the one-phase free boundary problem~\eqref{eqn:one-phase-problem}.
        \item As $a \in (0,a^*_{M_i}]$, the cones $\tilde{\mathbf{C}}_{M_i,a}$ attain every capillary angle $\theta \in (0,\frac{\pi}{2}]$.
        \item The cone $\tilde{\mathbf{C}}_{M_i, a^*_{M_i}}$ is the free-boundary cone $\mathbf{C}_{M_i, \frac{\pi}{2}}$, whose double has link given by the minimal surface $\mathbf{S}_{M_i} \subset \bS^{n}$.
    \end{enumerate}
    When the hypersurface $M \subset \bS^{n-1}$ satisfies $m_1 = m_2$, the halved cone $\bar{\mathbf{C}}_{M, \frac{\pi}{2}} := C( \bar{\mathbf{S}}_M) \cap \bR^{n+1}_+$ also admits an interpolation to the solution $\bar{U}_M$ of the one-phase problem~\eqref{eqn:one-phase-problem} by a smooth family $\{ \check{\mathbf{C}}_{M,a} \}_{a \in (0, a^*_M]}$ of capillary cones attaining every angle, which satisfies the properties $(i)$\thru$(iii)$.
\end{theorem}
The one-phase solution $\bar{U}_M$ of Type II exists for arbitrary isoparametric hypersurfaces, and is constructed in our companion paper~\cite{one-phase-isoparametric}*{Theorem 1.1}.
The resulting solution $\bar{U}_M$ then satisfies the subsequential convergence 
\[
\tfrac{1}{\tan \theta_i} \bar{\mathbf{C}}_{M,\theta_i} \to \textup{graph} (\bar{U}_M), \qquad \tfrac{1}{a_i} \bar{\mathbf{C}}_{M,a_i} \to \textup{graph} (\bar{U}_M)
\]
for some $\theta_i , a_i \downarrow 0$.
The property $m_1 = m_2$ applies to all isoparametric hypersurfaces with $g=3$ (called the \textit{Cartan-M\"unzner family}) or $g=6$ distinct principal curvatures.
Other examples are the Clifford tori $\bS^p \times \bS^p \subset \bS^{2p+1}$ and the homogeneous hypersurfaces with $(g,m_1,m_2) = (4,2,2)$.

When $g = 1$ and $M = \bS^{n-2}$, the isoparametric leaves are the distance spheres from the poles, which is a distinguished case.
The one-phase problem reduces to an $\textup{O}(n-1)$-invariant ansatz shown to be minimizing in $n \geq 7$ by De Silva-Jerison~\cite{desilva-jerison-cones}, which we prove to be unique in \cite{one-phase-isoparametric}.
\begin{theorem}\label{thm:uniqueness-of-axisymmetric}
For any angle $\theta \in (0,\frac{\pi}{2}]$, the axisymmetric cone $\mathbf{C}_{n,1,\theta}$ of~\cite{FTW-1} is the unique $\textup{O}(n-1)$-invariant non-flat capillary minimal cone in $\bR^{n+1}_+$.
In particular, the Clifford torus $\bS^{n-2} ( \sqrt{\frac{n-2}{n-1}}) \times \bS^1 ( \sqrt{\frac{1}{n-1}})$ and the sphere $\bS^{n-1}$ are the only rotationally invariant embedded minimal hypersurfaces of $\bS^n$.
\end{theorem}
The uniqueness of the Clifford torus $\bS^1(\frac{1}{\sqrt{2}}) \times \bS^1(\frac{1}{\sqrt{2}})$ in $\bS^3$ was proved in celebrated work of Brendle~\cite{clifford-tori}, later extended to embedded Weingarten tori~\cite{weingarten-tori}.
In Section~\ref{section:axisymmetric}, we prove a more general rigidity Theorem~\ref{thm:even-function-axisymmetric} for axisymmetric solutions of the capillary CMC problem.
Our general result is connected to the classification problem for embedded CMC tori in $\bS^3$ studied by Andrews-Li and Perdomo~\cites{andrews-li , perdomo }.

Constant mean curvature hypersurfaces arise as the natural constrained-area analogues of minimal hypersurfaces, notably as the boundaries of isoperimetric sets.
Their construction, global geometry, and topology have received great attention through various techniques, including gluing and min-max methods~\cites{korevaar-kusner, korevaar-kusner-solomon, meeks, kapouleas-cmc, compact-cmc-kapouleas , cmc-min-max , pmc-min-max , mazurowski , cmc-delaunay-ends , mazzeo-gafa , meeks-mira-perez-ros }.
Furthermore, capillary CMC surfaces capture the proper notion of variational solutions when the enclosed volume is prescribed, providing the correct generalization of minimal hypersurfaces with boundary and the necessary additional flexibility for performing slicing and $\mu$-bubble constructions in problems coming from scalar curvature geometry and general relativity~\cites{huisken-yau , qing-tian, brendle-eichmair , eichmair-metzger}.

Our framework from Theorems~\ref{thm:new-minimal-surfaces}\thru\ref{thm:smooth-interpolation} extends to the construction of minimal and CMC surfaces in the sphere with any prescribed non-negative mean curvature, along with capillary deformations in $\bS^n_+$ through every contact angle. 
The CMC deformations are ambiently isotopic to the minimal examples, so they maintain the same topological complexities and assemble into smooth families.
\begin{theorem}\label{thm:cmc-construction}
    For any isoparametric hypersurface $M \subset \bS^{n-1}$ with at least two distinct principal curvatures and any $H \geq 0$, there exist three closed embedded surfaces $\mathbf{S}^{H}_{M_1}, \mathbf{S}^{H}_{M_2}, \bar{\mathbf{S}}^{H}_M \subset \bS^n$ of constant mean curvature $H$.
    For every $H$, we have smooth ambient isotopies
    \[
    \mathbf{S}^H_{M_1} \simeq_{\textup{isot}} \mathbf{S}^0_{M_1} =: \mathbf{S}_{M_1}, \qquad \mathbf{S}^{H}_{M_2} \simeq_{\textup{isot}} \mathbf{S}^0_{M_2} =: \mathbf{S}_{M_2}, \qquad \bar{\mathbf{S}}^H_M \simeq_{\textup{isot}} \bar{\mathbf{S}}^0_M =: \bar{\mathbf{S}}_M.
    \]
    The surfaces $\Sigma^H_{M_i , \frac{\pi}{2}} := \mathbf{S}^H_{M_i} \cap \bS^n_+$ and $\bar{\Sigma}^H_{M,\frac{\pi}{2}} := \bar{\mathbf{S}}^H_M \cap \bS^n_+$ arise as endpoints of capillary families:
    \begin{enumerate}[$(i)$]
        \item The surfaces $\Sigma^H_{M_i, \frac{\pi}{2}} \subset \bS^n_+$ belong to a smooth family of capillary hypersurfaces $a \mapsto \Sigma^H_{M_i,a} \subset \bS^n_+$ of constant mean curvature $H$, attaining every contact angle $\theta \in (0,\frac{\pi}{2}]$, with each $\Sigma^H_{M_i,a}$ diffeomorphic to the open normal neighborhood of the focal submanifold $M_i$.
        \item For every $\theta \in (0, \frac{\pi}{2}]$, there exists a capillary hypersurface $\bar{\Sigma}^H_{M,\theta} \subset \bS^n_+$ of constant mean curvature $H$ diffeomorphic to $M \times [0,1]$.
        If the multiplicities of $M$ satisfy $m_1 = m_2$, then these surfaces belong to a smooth, surjective capillary CMC family $a \mapsto \bar{\Sigma}^H_{M,a} \subset \bS^n_+$.
    \end{enumerate}
\end{theorem}
In fact, our framework can be applied, after small modifications, to the construction of CMC surfaces $\mathbf{S}^H_{M_i}, \bar{\mathbf{S}}^H_M \subset \bS^n$ with mean curvature $H \in (-\ve_n, \infty)$, for some $\ve_n > 0$.
However, one cannot generally produce capillary surfaces of every angle when $H$ is negative; see Section~\ref{section:CMC-in-the-sphere} for more details.

\begin{remark}
The capillary constructions in~\cites{FTW-1,new-minimal-surfaces} correspond to the specific case of Theorem~\ref{thm:new-minimal-surfaces} where $M$ is the Clifford hypertorus, with $\mathbf{S}_{M_1}, \mathbf{S}_{M_2}$ being Clifford tori $\bS^{n-k} \times \bS^{k-1}$ or $\bS^{n-k-1} \times \bS^k$.
These results produced the first instances of complete capillary interpolations through every contact angle converging to a homogeneous one-phase solution in the rescaled small-angle limit, notably exhibiting the first examples of minimizing cones for the capillary problem.
The surfaces $\mathbf{S}_{M_i}$ of Theorem~\ref{thm:new-minimal-surfaces} are minimal embeddings of sphere bundles in $\bS^n$ and recover constructions of Gorodski~\cite{gorodski} as special examples, coming from isotropy representations of rank-$3$ compact symmetric spaces.
The construction of the minimal surfaces in the most symmetric Type II case for $M = \bS^p \times \bS^p \subset \bS^{2p + 1}$ was obtained by~\cite{carlotto-schulz} and~\cite{spherical-bernstein-xin-zhou} for $p=1$, then generalized to arbitrary $M = \bS^p \times \bS^q$ in~\cite{new-minimal-surfaces}.
Moreover,~\cite{lai-wei-2026} constructed minimal embeddings $M \times \bS^1 \hookrightarrow \bS^n$ using the approach of~\cite{riedler-shrinkers}, and~\cites{huang-wei-2022 , lai-wei-tori } adapted the construction of~\cite{carlotto-schulz} to CMC surfaces.
\end{remark}

\smallskip
\noindent\textbf{Acknowledgments:} We are grateful to Zihui Zhao for helpful conversations and comments on a preliminary version of this manuscript. 

\section{Preliminaries}\label{section:preliminaries}

We first recall some properties of the main objects of study in Theorems~\ref{thm:new-minimal-surfaces} through~\ref{thm:smooth-interpolation}, namely one-phase free boundaries and isoparametric hypersurfaces.
The \textit{one-phase Bernoulli} problem seeks pairs $(u,\Omega)$ of a function $u$ and a domain $\Omega$ with the property that $\Omega = \{ u > 0 \}$ and
\begin{equation}\label{eqn:one-phase-problem}\tag{OP}
        \Delta u = 0 \; \text{ in } \; \Omega, \qquad u = 0 \quad \text{in } \; \Omega^{\text{\sffamily{C}}}, \qquad |\nabla u| = 1 \; \text{ on } \; \partial \Omega.
\end{equation}
Solutions of the problem~\eqref{eqn:one-phase-problem} are critical points of the \textit{Alt-Caffarelli functional} $\cJ(u) := \int_B |\nabla u|^2 + \chi_{ \{ u > 0 \} }$.
The one-phase Bernoulli problem is closely connected to the theory of minimal surfaces, with many constructions and theorems having direct counterparts.
In low dimensions, classical methods such as the Weierstrass representation and gluing techniques have led to classification results and new examples~\cites{ entire-hairpins, traizet , jerison-kamburov, n-dim-catenoid ,  hines-kolesar-mcgrath}.
We refer the reader to~\cites{ one-phase-simon-solomon , FTW-1 , FTW-stability-one-phase } for some recent results on the one-phase problem and its emergent connections to the theory of minimal surfaces via capillary interpolations.

An important phenomenon in the one-phase Bernoulli problem is that the energy $\cJ(u)$ can be formally recovered as the limit under rescaling of an appropriate capillary energy $\cA^{\theta}$ with $\theta \downarrow 0$.
Moreover, given a family of minimizers $u_{\theta}$ of $\cA^{\theta}$, a subsequence of $\frac{u_{\theta}}{\tan \theta}$ converges, as $\theta \downarrow 0$, to a minimizer $u_0$ of $\cJ$.
This idea was utilized in~\cite{improved-regularity} and further developed in~\cites{FTW-1 , new-minimal-surfaces}, notably producing the first examples of non-trivial families exhibiting this limit, with both minimizing and non-minimizing behaviors occurring.
Crucially, these constructions lay the groundwork for a unifying emerging picture: not only do minimal surfaces have analogues in one-phase theory, but minimal submanifolds and free boundaries in fact form the two endpoints of an interpolating family of capillary hypersurfaces.
Theorem~\ref{thm:capillary-interpolation} exhibits this phenomenon in a fairly general setting.

\subsection{Isoparametric hypersurfaces of the sphere}\label{subsec:isoparametric}

A function $F: \bS^{n-1} \to \bR$ is called \textit{isoparametric} if there exist smooth functions $a, b : \bR \to \bR$ such that
\begin{equation}\label{eqn:distance-properties}
\bigl|\nabla^{\bS^{n-1}} F \bigr|^2 = a(F), \qquad \Delta_{\bS^{n-1}} F = b(F).
\end{equation}
Geometrically, this property ensures that the level sets $M_s := F^{-1}(s)$ of the function $F$ produce a foliation (possibly singular) of $\bS^{n-1}$ with constant mean curvature and a local transverse coordinate.
Conversely, any CMC foliation by equidistant hypersurfaces is locally isoparametric.

The leaves of such a foliation have constant principal curvatures and are called \textit{isoparametric hypersurfaces} $M \subset \bS^{n-1}$.
By work of Cartan and M\"unzner~\cites{Cartan1939, cheese1, cheese2}, these are characterized by parameters $(g, m_1, m_2)$ where $g \in \{1,2,3,4,6\}$ is the number of distinct principal curvatures and $(m_1, m_2)$ are the multiplicities of the first two principal curvatures, while $m_i = m_{i+2}$ modulo $g$.
The parameters $(g, m_1, m_2)$ satisfy
\begin{equation}\label{eqn:counting-multiplicities}
    g( m_1 + m_2) = 2 (n-2).
\end{equation}
Let $M_s$ denote the parallel surface to $M$ at distance $s$, which is also isoparametric except for two endpoints, called the \textit{focal submanifolds}, of codimensions $m_1 + 1$ and $m_2+ 1$.
The surfaces $M_s$ foliate $\bS^{n-1}$ and precisely one $M_s$ is minimal, which we will distinguish as $M$.
By rescaling, every point in $\bR^{n}$ lies on a dilate of exactly one $M_s$.
A function on $\bS^{n-1}$ that is constant on each leaf of a foliation (resp. one-homogeneous on $\bR^n$ and constant along dilated leaves) is called \textit{$F$-invariant}.

A large class of isoparametric foliations arises from orbits of isotropy representations of Riemannian symmetric spaces of rank $2$.
Concretely, Let $G/K$ be a compact irreducible symmetric space of rank $2$ with Cartan decomposition $\fg = \fk \oplus \fp$.
The isotropy representation is the $K$-action on $\fp$, and restricting to the unit sphere $S(\fp)$ produces a cohomogeneity one action due to $\text{rank}(G/K) = 2$, hence the orbit decomposition produces an isoparametric foliation, which we call \textit{homogeneous}.
A point in the interior of a Weyl chamber gives a principal orbit $K/K_0$, thus a homogeneous isoparametric hypersurface, while points on the two chamber walls give the two singular orbits $K/K_{\pm}$, namely the focal manifolds with sphere bundle structures $K/K_0 \to K/K_{\pm}$.
The homogeneous class includes all foliations with $m_1=1$ or $m_2=1$, and is classified according to the rank-$2$ symmetric spaces in~\cite{hsiang-lawson-jdg}.
The construction of $G$-invariant minimal surfaces in the sphere is the main objective of the \textit{low cohomogeneity} program initiated by Hsiang, Hsiang, and Lawson~\cites{ hsiang-lawson-jdg , hsiang-hsiang ,  spherical-bernstein , spherical-bernstein-2 , hsiang-sterling }.

Summarizing the above results, we have $g \in \{ 1 , 2 , 3 , 4 , 6 \}$ with corresponding characteristic triples $(g,m_1,m_2)$, which produce the following isoparametric foliations and associated focal submanifolds:
\begin{enumerate}[$\bullet$]
\item $g=1$: The isoparametric foliation comes from the equatorial minimal sphere $\bS^{n-2} \subset \bS^{n-1}$ and its parallel surfaces, corresponding to the spherical suspension $\bS^{n-1} = \Sigma  \bS^{n-2}$ of the round sphere.
We denote $m_1 = m_2 = n-2$ and the focal points are the north and south poles.
\item $g=2$: The foliation is given by the Clifford hypertori $\bS^{n-k-1}(\sin s)\times \bS^{k-1}(\cos s)$ for $s \in (0,\frac{\pi}{2})$ with multiplicities $(m_1, m_2) = (n-k-1, k-1)$.
The focal submanifolds are the high codimension equatorial spheres $M_1 \cong \bS^{n-k-1}$ and $M_2 \cong \bS^{k-1}$ lying in complementary orthogonal subspaces. 
\item $g=3$: We have $m_1 = m_2 \in \{1,2,4, 8\}$, with focal submanifolds given by the two standard Veronese embeddings of a projective plane
$\bF \bP^2 \hookrightarrow \bS^{3m+1}$, related by the antipodal map.
Here, $\bF = \bR , \bC , \bH , \bO$ is one of the four real division algebras of dimension $m = 1,2,4,8$, respectively.
The regular leaves $\{ M_s \}_{s \in (0, \frac{\pi}{3})}$ of the associated foliation are the distance tubes around the focal submanifolds, given by one of the Wallach flag manifolds corresponding to $m = 1,2,4,8$,
\[
\textup{SO}(3) / (\bZ_2 \oplus \bZ_2) , \qquad \textup{SU}(3) / \bT^2, \qquad \textup{Sp}(3) / \textup{Sp}(1)^3, \qquad F_4 / \textup{Spin}(8).
\]
\item $g=4$:
Every isoparametric family is either of OT-FKM type~\cites{ferus-karcher-munzner,OTg4} or one of the two exceptional homogeneous families with multiplicity pairs $(m_1, m_2) \in \{ (2,2) , (4,5) \}$.

\smallskip \noindent \textbf{OT-FKM:}
For the OT-FKM family, $(m_1, m_2) = (m , k \, \delta(m) - m - 1)$, where $\delta(m)$ is the dimension of an irreducible $C_{m-1}$-module over the Clifford algebra with $(m-1)$ generators.
The values of $\delta(m)$ are computed using Bott periodicity as
\begin{equation}\label{eqn:delta(m)-values}
\delta ( 8k +j) = 2^{4k} \delta(j), \qquad \text{where} \quad \delta(1) = 1, \; \delta(2) = 2, \; \delta(3) = \delta(4) = 4, \; \delta(j) = 8, \; 5 \leq j \leq 8.
\end{equation}
The focal manifolds are expressible as sphere bundles
\[
M_1 \cong S(\eta)\to \bS^{k \delta(m)-1},
\qquad
M_2\cong S(\xi)\to \bS^{m},
\]
of vector bundles of rank $(k \delta(m) -m)$ over $\bS^{k\delta(m)-1}$ and rank $k \delta(m)$ over $\bS^m$, cf.~\cites{wang-topology-clifford , ojm}.
For example, when $m=1$, the foliation is homogeneous and has a focal submanifold $M_1\cong \bV(2,k \, \delta(m))$ where $\bV(k,n)$ is the Stiefel manifold of orthonormal $k$-frames in $\bR^n$, with
\begin{equation}\label{eqn:stiefel-manifold}
\bV(k,n) := \{ A \in \bR^{n \times k} : A^{\top} A = I_k \} \cong \textup{O}(n)/ \textup{O}(n-k) \cong \textup{SO}(n) / \textup{SO}(n-k), \quad \text{for } \; k < n.
\end{equation}
More generally, given a representation of the Clifford algebra $C_{m-1}$ on $\bR^{k \delta(m)}$ by skew-symmetric matrices $E_1, \dots, E_{m-1}$ satisfying $E_i E_j + E_j E_i = - 2 \delta_{ij}I$, we have $M_1 \cong \mathbb{V}(2, C_{m-1})$ where 
\begin{equation}\label{eqn:CliffordStiefel}
\bV(2, C_{m-1}) = \bigl\{ (u,v) \in \bS^{2 k \delta(m) - 1} : |u| = |v| = \tfrac{1}{\sqrt{2}}, \; \la u,v \rg = 0, \; \la E_i u, v \rg = 0, \; 1 \leq i \leq m-1 \}
\end{equation}
is called the \textit{Clifford-Stiefel manifold} of Clifford orthonormal $2$-frames, cf.~\cite{cecil}*{(11)}.

The OT-FKM isoparametric hypersurfaces, also called the \textit{Clifford family} due to their construction via Clifford algebras, form an infinite family that exhibits many interesting topological properties.
We refer the reader to~\cites{wang-topology-clifford , topology-g-4 } for a detailed presentation of these behaviors.

\smallskip \noindent \textbf{Homogeneous families:}
In the homogeneous exceptional cases, we have 
\begin{align*}
(m_1,m_2)=(2,2)&:\qquad
M_1\cong \bC\bP^3,
&&M_2\cong \widetilde{\bG}(2,5) \cong Q^3,\\
(m_1,m_2)=(4,5)&:\qquad
M_1\cong \textup{U}(5)/( \textup{Sp}(2)\times \textup{U}(1)),&&M_2\cong \textup{U}(5)/(\textup{SU}(2)\times \textup{U}(3)),
\end{align*}
where $\widetilde{\bG}(k,n)$ denotes the Grassmannian of oriented $k$-planes in $\bR^n$ and $Q^n$ denotes the complex quadric $n$-fold, given by $Q^n := \{ [z] \in \bC \bP^{n+1} : \la z, z \rg = 0 \}$.
This satisfies the diffeomorphism
\begin{equation}\label{eqn:oriented-grassmannian-quadric}
    \widetilde{\bG} (2 , n+2) \cong Q^n \cong \textup{SO}(n+2) / (\textup{SO}(n) \times \textup{SO}(2) ) \, .
\end{equation}
The two homogeneous foliations come from the rank $2$ symmetric spaces $G/K$ with $(G,K) = (\textup{SO}(5) \times \textup{SO}(5) / \Delta \textup{SO}(5) )$ and $(\textup{SO}(10), \textup{U}(5))$, whose principal orbits produce the regular leaves
\[
M_{(4,2,2)} \cong \textup{SO}(5) / \bT^2, \qquad M_{(4,4,5)} \cong \textup{U}(5)/ ( \textup{SU}(2) \times \textup{SU}(2) \times \textup{U}(1))
\]
where $\bT^2$ is the maximal torus in $\textup{SO}(5)$.

\item $g=6$: We have $m_1 = m_2 \in \{ 1, 2\}$.
For $m_1 = 1$, the unique isoparametric foliation is homogeneous, given by the isotropy representation of $G_2 / \textup{SO}(4)$, with leaves $\{ M_s \} \subset \bS^7$ arising as inverse images under the Hopf fibration $\bS^7 \to \bS^4$ of the foliation $(\tilde{M}_s)$ of $\bS^4$ with $(g,m_1,m_2) = (3,1,1)$, coming from the Veronese embedding of $\bR \bP^2$~\cites{dorfmeister-neher , siffert-PAMS}.
The two focal submanifolds are non-congruent minimal homogeneous embeddings of $\bR \bP^2 \times \bS^3 \hookrightarrow \bS^7$ with different Ricci spectra, hence non-isometric~\cites{cecil , topology-g-4 }.
The regular leaf is $M \cong \textup{SO}(4) / \bZ_2^{\oplus 2}$, cf.~\cite{miyaoka}.

For $m_1 =2$, the only known example in $\bS^{13}$ is homogeneous and comes from the adjoint representation of $G_2$ on $\fg_2$, namely the isotropy representation of $(G_2 \times G_2)/ \Delta G_2$, with regular leaf $G_2/(\textup{U}(1) \times \textup{U}(1)) \cong G_2/\bT^2$, the flag manifold of $G_2$.
The focal leaves are respectively diffeomorphic to maximally parabolic quotients $G_2/ \textup{U}(2)^{\pm}$, where $\textup{U}(2)^{\pm}$ are the two non-conjugate subgroups of $G_2$ such that $G_2/\textup{U}(2)^+$ is the twistor space of $G_2/ \textup{SO}(4)$ and $G_2/\textup{U}(2)^- $ is the reduced $G_2$-twistor space of $\bS^6 = G_2/ \textup{SU}(3)$.
The latter manifold can be identified with the complex quadric $Q^5$, which satisfies the diffeomorphisms of~\eqref{eqn:oriented-grassmannian-quadric}.
Moreover, $Q^5 \cong \bP_{\bC} ( T \bS^6)$ is the projectivization of the tangent bundle of $\bS^6$ viewed as a complex rank-$3$ bundle.
The classification of isoparametric hypersurfaces with $(g,m) = (6,2)$ appears to not be fully understood~\cites{miyaoka , siffert-AGAG}.
\end{enumerate} 
Each isoparametric hypersurface $M_s$ separates the sphere $\bS^{n-1}$ into two connected components $D_1, D_2$ with $D_1 \cup D_2 = \bS^{n-1}$ and $D_1 \cap D_2 = M_s$, which are disk bundles over the focal manifolds $M_1, M_2$.
Taking $M_i$ to have codimension $m_i + 1$, we obtain the disk bundles
\[
\bD^{m_1 + 1} \to D_1 \to M_1, \qquad \bD^{m_2 + 1} \to D_2 \to M_2.
\]
The volume density of the leaves $M_s$ is computed as 
\begin{equation}\label{eqn:isoparametric-volume-density}
v(s) = C \left( \sin \tfrac{gs}{2} \right)^{m_1} \left( \cos \tfrac{gs}{2} \right)^{m_2},
\end{equation}
so the mean curvature of the leaf $M_{s}$ is constant and given by
\begin{equation}\label{eqn:h-sigma}
\cH(s) = - \frac{d}{d\sigma}\bigg\rvert_{\sigma = s} \log v(\sigma) = \frac{g}{2} \left( m_2 \tan \frac{g s}{2} - m_1 \cot \frac{gs}{2} \right).
\end{equation}
For any $F$-invariant function $\phi = \phi(s)$ on $\bS^{n-1}$ and $g_s$ the induced metric on $M_s$, we have
\begin{equation}\label{eqn:laplacian-of-f}
\Delta_{\bS^{n-1}} \phi = \phi'' + ( \partial_s \log \sqrt{\det g_s} ) \phi' =  \phi'' - \cH \phi' = \phi'' + \tfrac{v'}{v} \phi'.
\end{equation}
We refer the reader to~\cite{siffert-wuzyk}*{\S 2.2} for important properties of isoparametric hypersurfaces and a more detailed summary of their classification.

\subsection{The minimal surface equation}\label{sec:minimal-surface-equation}
 
Consider the Euclidean space $\bR^{n+1}_+ = \{ (x,z) \in \bR^n \times \bR : z \geq 0 \}$ with boundary plane $\Pi = \{ z = 0 \}$ and let $\bS^n_+ := \bS^n \cap \{ z \geq 0 \}$ denote the upper hemisphere.
For $M \subset \bS^{n-1}$ an isoparametric hypersurface as in Section~\ref{subsec:isoparametric}, we define $s$ to be the natural parameter given by the normal distance to the first focal submanifold, so $\{ M_s \}_{s \in (0,s_g)}$ are the parallel leaves with $s=0$ and $s_g = \frac{\pi}{g}$ corresponding to the focal submanifolds $M_1$ and $M_2$.

A cone $\mathbf{C} \subset \bR^{n+1}_+$ is a capillary minimal cone with contact angle $\theta$ along $\Pi$ if and only if its link $\Sigma := \mathbf{C} \cap \bS^n_+$ is a minimal hypersurface with boundary lying in the equator $\bS^{n-1} = \bS^n_+ \cap \Pi$ and meeting the container plane $\Pi$ at a constant angle $\theta$.
In the free-boundary case $\theta = \frac{\pi}{2}$, doubling across the equator produces a closed embedded $\bZ_2$-symmetric (reflection-invariant) minimal hypersurface in the sphere.
We will examine $F$-invariant cones $\mathbf{C} \subset \bR^{n+1}_+ = \bR^n \times \bR_{\geq 0}$ expressed as the graphs of $U(\rho, \omega) = \rho \phi( s (\omega))$, where $\rho = |x|$ and $s$ is the parameter along the isoparametric foliation:
\[
\mathbf{C} = \bigl\{ ( \rho \omega , \rho \phi(s) ) : \rho \geq 0, \; \omega \in M_s, \; s \in I \bigr\},
\]
where $I \subset [0,\frac{\pi}{g}]$ is an interval.
The spherical link is given by
\begin{equation}\label{eqn:spherical-link-of-Sigma}
\Sigma = \mathbf{C} \cap \bS^n_+ =  \left\{ \left( \frac{\omega}{\sqrt{1 + \phi(s)^2}}, \frac{\phi(s)}{\sqrt{1 + \phi(s)^2}} \right) : \omega \in M_s, \; s \in I  \right\} \subset \bS^n_+ \, .
\end{equation}
A more convenient parametrization is given by the meridional variable $t := \sin \frac{gs}{2}$, for which the volume element~\eqref{eqn:isoparametric-volume-density} corresponds to the weighted measure
\[
v(s) \, ds = \tfrac{2}{g} C t^{m_1} (1-t^2)^{\frac{m_2 - 1}{2}} \, dt, \qquad t = \sin \tfrac{gs}{2}, \qquad dt = \tfrac{g}{2} \sqrt{1-t^2} \, ds.
\]
Let $f(t) := \phi(s)$ be the corresponding function in~\eqref{eqn:spherical-link-of-Sigma}, defined on an interval $J = (t_1, t_2) \subset [0,1]$.

\begin{definition}\label{def:linear-legendre-operator}
We denote by $\cL_{M}$ the linear Legendre-type operator with self-adjoint expression
\begin{equation}\label{eqn:self-adjoint-legendre}
\begin{split}
    \cL_{M} f &:= (1-t^2) f'' + \bigl( \tfrac{m_1}{t} - (m_1 + m_2 + 1) t \bigr) f' + \tfrac{4(n-1)}{g^2} f \\
    &\;= \frac{1-t^2}{p_{M}(t)} \left[ \bigl( p_{M}(t) f'(t) \bigr)' + \frac{4 (n-1) }{g^2} \frac{p_{M} (t)}{1 -t^2} f(t)  \right] ,
\end{split}
\end{equation}
where $p_{M}(t) = t^{m_1} (1 - t^2)^{\frac{m_2 + 1}{2}}$.
We also define the function 
\begin{equation}\label{eqn:A-gm1m2}
    A_{M}(t) := \tfrac{2}{g(m_1+ m_2)} \sqrt{1-t^2} \, \cH(t(\omega)) = t - \tfrac{m_1}{m_1 + m_2} t^{-1} .
\end{equation}
\end{definition}
The operator $\cL_M$ is related to the Laplacian in $\bS^{n-1}$ applied to functions that are invariant along an isoparametric foliation with parameters $(g,m_1,m_2)$, and it also arises as the linearization of the following nonlinear equation \eqref{eqn:ode-star}.

\begin{proposition}\label{prop:mean-curvature-equation}
The hypersurface $\Sigma = \mathbf{C} \cap \bS^n_+$ has constant mean curvature $H_0$ if and only if the function $f$ is positive on an interval $(t_1, t_2) \subset (0,1)$ and satisfies the equation
\begin{equation}\tag{$\textup{C}\star$}\label{eqn:CMC-equation}
        \begin{split}
            &(1-t^2) f'' + (f - tf') + \frac{4-g^2}{g^2} f + (n-2) \left( \frac{4}{g^2} + (1-t^2) \frac{(f')^2}{1+f^2} \right) \left( f - \frac{g}{2} A_M(t) f' \right) \\
            &\;\;\quad = -\frac{4 H_0}{g^2} \left( 1 + \frac{g^2}{4} (1-t^2) \frac{(f')^2}{1+f^2} \right)^{\frac{3}{2}} \, .
        \end{split}
    \end{equation}
In particular, the cone $\mathbf{C}$ is minimal if and only if $f$ satisfies
\begin{equation}\label{eqn:ode-star}\tag{$\star$}
	(1-t^2) f'' + (f - tf') + \frac{4-g^2}{g^2} f + (n-2) \left( \frac{4}{g^2} + (1-t^2) \frac{(f')^2}{1+f^2} \right) \left( f - \frac{g}{2} A_{M}(t) f' \right) = 0.
\end{equation}
If $t_1 \neq 0$ and $t_2 \neq 1$, then $\Sigma$ has capillary free boundary with contact angle $\theta \in (0 , \frac{\pi}{2}]$ if and only if
\begin{equation}\label{eqn:capillary-angle-condition}
    \sqrt{1 - t_1^2} \, f'(t_1) = - \sqrt{1 - t_2^2} \, f'(t_2) = \tfrac{2}{g} \tan \theta
\end{equation}
where $f(t_1) = f(t_2) = 0$.
If $t_1 = 0$ or $t_2 = 1$, then $\Sigma$ extends smoothly across the focal submanifold $M_1$ (resp.~$M_2$) if and only if $f'(0) = 0$ (resp.~$f'(1) = 4 \frac{(n-1) f(1) + H_0}{g^2 (m_2 + 1)}$).
\end{proposition}
\begin{proof}
The equation~\eqref{eqn:ode-star} is obtained using standard geometric properties for functions along an isoparametric foliation, as computed in~\cite{FTW-1}*{Proposition 3.2}.
We first compute the corresponding equation for the function $\phi(s)$ in~\eqref{eqn:spherical-link-of-Sigma}.
Let us express the points of $\bR^n \setminus \{ 0 \}$ as $x = \rho \omega$, where $\rho = |x|$ and $\omega \in \bS^{n-1}$.
Using the fact that the isoparametric parameter $s$ on $\bS^{n-1}$ is orthogonal to each leaf $M_s$, we can express the Euclidean metric as $g_{\bR^n} = d \rho^2 + \rho^2 ( ds^2 + g_s)$, for $g_s$ the induced metric on $M_s$.
We lift $s$ from $\bS^{n-1}$ to $\bR^n$ linearly by homogeneity.
For a function $u = u(\rho,s)$ that is constant on leaves, we take the orthonormal frame $\{ e_{\rho}, e_s \} = \{ \partial_{\rho}, \frac{1}{\rho} \partial_s \}$ to compute
\begin{equation}\label{eqn:grad-u-computation}
\nabla u = u_{\rho} e_{\rho} + \tfrac{1}{\rho} u_s e_s, \qquad \nabla_{e_{\rho}} e_j = 0, \qquad \nabla_{e_s} e_{\rho} = \tfrac{1}{\rho} e_s, \qquad \nabla_{e_s} e_s = - \tfrac{1}{\rho} e_{\rho}.
\end{equation}
We therefore obtain
\begin{align}
    |\nabla u|^2 &= u_{\rho}^2 + \tfrac{1}{\rho^2} u_s^2, \qquad \Delta_{\bR^n} u = u_{\rho \rho} + \tfrac{n-1}{\rho} u_{\rho} + \tfrac{1}{\rho^2} ( u_{ss} - \cH(s) u_s ), \label{eqn:modNabluaUsquare}\\
    D^2 u (e_{\rho}, e_{\rho}) &= u_{\rho \rho}, \quad D^2 u (e_{\rho}, e_s) = \partial_{\rho} \bigl( \tfrac{u_s}{\rho} \bigr) = \tfrac{1}{\rho} u_{\rho s} - \tfrac{1}{\rho^2} u_s, \quad D^2 u(e_s, e_s) = \tfrac{1}{\rho^2} u_{ss} + \tfrac{1}{\rho} u_{\rho}.
\end{align}
Combining these computations with~\eqref{eqn:grad-u-computation}, we arrive at
\begin{equation}\label{eqn:Q(u)-expression}
Q(u) := D^2 u(\nabla u, \nabla u) = u_{\rho}^2 u_{\rho \rho} + \tfrac{2}{\rho^2} u_{\rho} u_s u_{\rho s} - \tfrac{1}{\rho^3} u_{\rho} u_s^2 + \tfrac{1}{\rho^4} u_s^2 u_{ss}.
\end{equation}
When $u = \rho \phi(s)$ is a homogeneous function, the above computations specialize to
\[
\rho \, \Delta u = \phi'' - \cH \phi' + (n-1) \phi, \qquad |\nabla u|^2 = \phi^2 + (\phi')^2, \qquad \rho \, Q(u) = \sum u_{ij} u_i u_j = (\phi')^2 ( \phi'' + \phi).
\]
We denote by $\cM(u)$ the minimal surface operator for graphs over $\bR^n$, so that
\begin{equation}\label{eqn:MSE-for-graphs}
    \sqrt{ 1 + |\nabla u|^2 }  \cM(u) = \sqrt{ 1 + |\nabla u|^2 }  \on{div} \left( \frac{\nabla u}{\sqrt{1 + |\nabla u|^2}} \right) = \Delta u - \frac{Q(u)}{1 + |\nabla u|^2}.
\end{equation}
The cone $\mathbf{C}$ over the hypersurface $\Sigma$ has Euclidean mean curvature $H_{\mathbf{C}}(p) = |p|^{-1} H_{\Sigma}(\frac{p}{|p|})$; in our situation, $\cM(u) = - H[ \on{graph} u]$ is the negative of the scalar mean curvature with respect to the upward normal.
The graph of $\rho \phi$ consists of points $p = (\rho \omega, \rho \phi)$, so $|p| = \rho \sqrt{1+\phi^2}$ and $H_{\Sigma} = H_0$ becomes equivalent to the equation $\cM(\rho \phi) =- \frac{H_0}{\rho \sqrt{1+ \phi^2}}$.
Combining the above expressions shows that this condition is equivalent to the equation
\begin{equation}\label{eqn:equation-for-phi(s)}
    (1 + \phi^2)^{\frac{3}{2}} (\phi'' - \cH \phi' + (n-1) \phi ) + (\phi')^2 (1 + \phi^2)^{\frac{1}{2}}\bigl( (n-2) \phi - \cH \phi' \bigr) = -H_0  ( 1+\phi^2 + (\phi')^2)^{\frac{3}{2}}.
\end{equation}
Taking $t = \sin \frac{gs}{2}$ as above, we write $\phi(s) = f(t)$, so that
\begin{equation}\label{eqn:phi-phi''-H-of(t)}
    \phi' = \frac{g}{2} \sqrt{1-t^2} f', \qquad \phi'' = \frac{g^2}{4} \bigl( (1-t^2) f'' - tf' \bigr), \qquad \cH(t(s)) = \frac{g}{2} \frac{(m_1+ m_2) t - \frac{m_1}{t}}{\sqrt{1-t^2}} \, .
\end{equation}
Substituting these expressions into the equation~\eqref{eqn:equation-for-phi(s)} produces the claimed ODE~\eqref{eqn:CMC-equation}.
The $F$-invariant minimal surface equation~\eqref{eqn:ode-star} follows from this expression upon taking $H_0 = 0$.

To study the boundary conditions for $f$, let $\Omega = \{\rho \omega : \rho > 0, s \in (s_1, s_2), \omega \in M_s\}$ where $\phi(s_i) = 0$.
The graph of $u$ has upward-pointing normal vector given by $\nu = \frac{( - \nabla u, 1)}{\sqrt{1 + |\nabla u|^2}}$, so $\la \nu, e_z \rg = \frac{1}{\sqrt{1 + |\nabla u|^2}}$.
Along the free boundary $\partial \{ u > 0 \}$, this graph forms a contact angle $\theta$ if and only if $\la \nu, e_z \rg = \cos \theta$, so $|\nabla u| = \tan \theta$; for $u = \rho \phi(s)$, the computation~\eqref{eqn:modNabluaUsquare} makes $|\nabla u| = |\phi'(s_i)|$ on $\partial \{ u > 0\}$.
The inward pointing conormal at the boundary is $e_s$ at $s_1$ and $-e_s$ at $s_2$, so we require $\phi'(s_1) = \tan \theta$ and $\phi'(s_2) = - \tan\theta$.
Applying $\phi' = \tfrac{g}{2}\sqrt{1-t^2}f'(t)$, we obtain the capillary boundary condition~\eqref{eqn:capillary-angle-condition}.

Regarding the smooth extendability of $\Sigma$ across the focal submanifolds, when $t=0$, the singularity of the term $A_M(t)$ requires $f'(0) = 0$.
Conversely, using $f'(0) = 0$ in equation~\eqref{eqn:CMC-equation} implies that $f''(0) = - 4\frac{H_0 + (n-1) f(0)}{g^2 (m_1 +1)}$, and computing iteratively shows that $f(t) = \sum_{k=0}^{\infty} b_k t^{2k}$ has a series expansion in terms of even powers of $t$, so $f(t) = F(t^2)$ near $t=0$.
This implies that the equation~\eqref{eqn:CMC-equation} becomes regular in the variable $t^2$, so $\Sigma$ has a smooth extension across $M_1$.

For $M_2$ and $t_2 = 1$, we define $\tilde{f}(r) := f( \cos \frac{gr}{2})$ in terms of the geodesic normal distance $r := s_g - s \geq 0$ from $M_2$, where $s_g = \frac{\pi}{g}$.
Then, the smooth extension of $\Sigma$ across the focal manifold $M_2$ is equivalent to $\tilde{f}$ being smooth and even near $r=0$, so that $\tilde{f}_r(0) = 0$ and $\frac{\tilde{f}_r(r)}{r} \to \tilde{f}_{rr}(0)$ as $r \to 0$.
We can expand the functions in terms of the geodesic normal distance $r = s_g - s$ to $M_2$,  
\begin{equation}\label{eqn:t-expansion-r}
t = \sin \tfrac{gs}{2} = \sin ( \tfrac{\pi}{2} - \tfrac{gr}{2}) = 1 - \tfrac{g^2}{8} r^2 + O(r^4), \qquad 1 - t = \tfrac{g^2}{8} r^2 + O(r^4).
\end{equation}
In terms of $r$, the mean curvature $\cH$ of the leaves becomes
\[
\tilde{\cH}(r) := \cH(s_g - r) = \tfrac{g}{2} \bigl( m_2 \cot \tfrac{gr}{2} - m_1 \tan \tfrac{gr}{2} \bigr) = m_2 r^{-1} + O(r)
\]
due to $\cot \frac{gr}{2} = \frac{2}{gr} + O(r)$ and $\tan \frac{gr}{2} = O(r)$ near $r=0$.
We therefore find $\tilde{\cH}(r) \tilde{f}_r(r) = ( \frac{m_2}{r} + O(r)) \tilde{f}_r(r) \to m_2 \tilde{f}_{rr}(0)$ and $\tilde{f}_r(r), \tilde{\cH}(r) \tilde{f}_r(r)^2 \to 0$.
Therefore, using the change of variables~\eqref{eqn:t-expansion-r} in $r$ in equation~\eqref{eqn:CMC-equation} and taking limits as $r \to 0$ forces $(m_2 + 1) \tilde{f}_{rr}(0) + (n-1) \tilde{f}(0) = -H_0$ after canceling a factor of $1+ \tilde{f}(0)^2$.
Finally, the chain rule~\eqref{eqn:t-expansion-r} shows that
\[
\tilde{f}_r(r) = - \tfrac{g}{2} f' \bigl( \cos \tfrac{gr}{2} \bigr) \sin \tfrac{gr}{2}, \qquad \tfrac{\tilde{f}_r(r)}{r} = - \tfrac{g}{2} f' \bigl( \cos \tfrac{gr}{2} \bigr) \tfrac{\sin ( \frac{gr}{2})}{r} \to - \tfrac{g}{2} f'(1) \cdot \tfrac{g}{2} = - \tfrac{g^2}{4} f'(1).
\]
Consequently, $\tilde{f}_{rr}(0) = - \frac{g^2}{4} f'(1)$ and $\tilde{f}(0) = f(1)$, so the above condition on $\tilde{f}_{rr}(0)$ rearranges to $f'(1) = 4 \frac{(n-1) f(1) + H_0}{g^2(m_2 + 1)}$, proving the necessary condition.
To see that this condition is sufficient, recall that for hypersurfaces that are invariant along leaves of the isoparametric foliation, smooth extension across $M_2$ is equivalent to the profile being a smooth function of $r^2$.
The function $1-t$ is a smooth quadratic defining function for the focal submanifold $M_2$, so~\eqref{eqn:t-expansion-r} together with $f'(1)= 4 \frac{(n-1) f(1) + H_0}{g^2 (m_2 + 1)}$ produce a regular Taylor expansion in $1-t$, therefore in $r^2$, near $M_2$; this proves the smooth extension of $\Sigma$. 
\end{proof}

We will focus on the minimal surface equation~\eqref{eqn:ode-star} for the majority of the paper, where $H_0 = 0$.
In Section~\ref{section:CMC-in-the-sphere}, we discuss the adaptations needed to extend Theorem~\ref{thm:new-minimal-surfaces} to the CMC setting.

For the spherical suspension and Clifford hypertori examples, we can explicitly compute the volume forms and recover the equations in~\cites{FTW-1, new-minimal-surfaces}.
The volume of the suspension metric $dr^2 + \sin^2(r) g_{\bS^{n-2}}$ is $\sin(r)^{n-2}\omega_{n-2}$ which in this framework produces
\[ 
v(s) = C \bigl( \sin \tfrac{s}{2} \bigr)^{n-2} \bigl(\cos \tfrac{s}{2} \bigr)^{n-2} = 2^{2-n} \, C  (\sin s)^{n-2}.
\]
Notably, this expression recovers equation $\cL_{\bS^{n-2}} f =0$ (denoted $\cL_{n,n-2}$ in~\cite{FTW-1}), the ODE for the axisymmetric De Silva-Jerison cone~\cite{desilva-jerison-cones}.
More generally, the isoparametric leaves $M_s \subset \bS^{n-1}$ generated by the Clifford hypersurfaces have volume form
\[
v_{n,k}(s) = C (\sin s)^{n-k-1} ( \cos s)^{k-1},
\]
which recovers the Legendre operator $\cL_{\bS^{n-k-1}\times \bS^{k-1}}$ studied in~\cite{FTW-1}*{Definition 7} denoted $\cL_{n,k}$ there.
On the other hand, the nonlinear equation~\eqref{eqn:ode-star} does not admit explicit algebraic solutions when $g \geq 3$.
Unlike the cases $g \in \{ 1,2 \}$ with $m_1 = n-k-1 \in [1,n-2]$, where $f(t) = \sqrt{\frac{k - (n-1) t^2}{n-k-1}}$ produces a free boundary solution corresponding to the halved Clifford hypersurface $\bS^{n-k-1} \times \bS^k$ in $\bS^n$, no such description is possible when $g \geq 3$.

On the other hand, for $g=1$ the $F$-invariant surfaces of the spherical foliation are rotationally invariant, and therefore correspond to the $\textup{O}(n-1)$-invariant (axisymmetric) profile equation~\eqref{eqn:axisymmetric-equation-fake} obtained in the construction of area-minimizing capillary cones of~\cite{FTW-1}*{\S 3}, namely
\begin{equation}\label{eqn:axisymmetric-equation-fake}
    f'' + (f - tf') \Bigl( \frac{n-1}{1-t^2} + (n-2) \frac{(f')^2}{1 + f^2} \Bigr) = 0.
\end{equation}
The equation~\eqref{eqn:axisymmetric-equation-fake} is equivalent to~\eqref{eqn:ode-star} under a change of coordinates in $\bR^n$, expressing $x = (x', x_n) \in \bR^{n-1} \times \bR$ and setting $t := \frac{x_n}{\sqrt{|x'|^2 + x_n^2}}$ as in the notation of Theorem~\ref{thm:uniqueness-of-axisymmetric}.
In this case, the unique free-boundary cone is given by the half-Lawson cone $C(\bS^{n-2} \times \bS^1_+)$ with profile $\hat{f}_{n,1}(t) = \sqrt{\frac{1 - (n-1) t^2}{n-2}}$, see~\cite{FTW-1}*{\S 3.3}, while $f = 0$ produces the plane $\Pi \subset \bR^{n+1}$ with link the equatorial sphere $\bS^{n-1} \subset \bS^n$ and $f = t_+$ corresponds to the half-space solution $u(x) = (x_n)_+$.
Notably, there exists only one non-trivial solution to equation~\eqref{eqn:axisymmetric-equation-fake}, resulting in a topological torus $M \times \bS^1 = \bS^{n-2} \times \bS^1$.
Under the transformation between ambient and isoparametric coordinates, the function $\hat{f}_{n,1}(t)$ corresponds to the unique solution of equation~\eqref{eqn:ode-star} given by
\begin{equation}\label{eqn:clifford-solution}
\bar{f}_{\bS^{n-2}}(t) = \sqrt{\tfrac{4}{n-2} t^2 (1-t^2) - (1 - 2t^2)^2} \, . 
\end{equation}
We therefore see that the axisymmetric equation~\eqref{eqn:axisymmetric-equation-fake} has different properties from~\eqref{eqn:ode-star}, so we consider $g \geq 2$ in what follows and prove the uniqueness of axisymmetric cones in Section~\ref{section:axisymmetric}.

To examine the capillary cone equation~\eqref{eqn:ode-star} alongside the linear problem $\cL_{M} f = 0$ in a unified way, we consider the equation satisfied by the function $f_{\lambda} := \frac{f}{\sqrt{\lambda}}$.
This rescaling has the effect of changing the nonlinear term of~\eqref{eqn:ode-star} by a factor $\lambda$, producing the equation
\begin{equation}\label{eqn:rescaled-equation-lambda-ODE}\tag{$\star_{\lambda}$}
    \cL_{M} f +(n-2) (1-t^2) \frac{\lambda (f')^2}{1 + \lambda f^2} \left( f - \frac{g}{2} A_{M} (t) f' \right) = 0.
\end{equation}
The study of equation~\eqref{eqn:rescaled-equation-lambda-ODE} is motivated by the limiting behavior of the capillary and one-phase problems presented in Section~\ref{section:preliminaries}; indeed, the one-phase problem $\cL_M f = 0$ is recovered in the limit $\lambda \downarrow 0$.
We refer the reader to~\cites{desilva-jerison-cones , one-phase-simon-solomon , FTW-1 , FTW-stability-one-phase , six-way} for further analysis of the resulting homogeneous solutions.

\begin{lemma}\label{lemma:change-of-variables}
For given $(g, m_1, m_2)$, the function $f(t)$ solves equation~\eqref{eqn:rescaled-equation-lambda-ODE} if and only if the function $f (\sqrt{1-t^2})$ solves this equation for $(g, m_2, m_1)$.
Any capillary cones produced by solving this equation with the appropriate free boundary condition are isometric via the involution $t \leftrightsquigarrow \sqrt{1-t^2}$; in particular, they form the same contact angle.
The same property holds for solutions of equation~\eqref{eqn:CMC-equation} and for the capillary CMC surfaces produced by their graphs.
\end{lemma}
\begin{proof}
Working with the function $\phi(s)$ as in Proposition~\ref{prop:mean-curvature-equation}, let $\tilde{s} := s_g - s = \frac{\pi}{g} - s$ and $\tilde{\phi}(\tilde{s}) := \phi(s_g - s)$, so that $\tilde{\phi}'(\tilde{s}) = - \phi'(s) , \tilde{\phi}''(\tilde{s}) = \phi''(s)$.
For the foliation with multiplicities interchanged, the mean curvature is $\tilde{\cH}(\tilde{s}) = \frac{g}{2} ( m_1 \tan \frac{g \tilde{s}}{2} - m_2 \cot \frac{g \tilde{s}}{2}) = - \cH(s)$ due to $\frac{g \tilde{s}}{2} = \frac{\pi}{2} - \frac{gs}{2}$, so $\tilde{\phi}'' - \tilde{\cH} \tilde{\phi}' = \phi'' - \cH \phi'$ and all the terms of equation~\eqref{eqn:equation-for-phi(s)} are preserved.
The variable $\tilde{s}$ corresponds to $\tilde{t} := \sqrt{1-t^2}$ for $t = \sin \frac{gs}{2}$, and the boundary conditions on $(t_1,t_2)$, $(0,t_2)$, and $(t_1,1)$ clearly correspond to analogous boundary conditions on $(\tilde{t}_2,\tilde{t}_1)$, $(\tilde{t}_2,1)$ and $(0,\tilde{t}_1)$, since $\sqrt{1-t^2} f'(t) = -\sqrt{1-\tilde{t}^2} \tilde{f}'(\tilde{t})$.
\end{proof}

When $g \in \{1,3,6 \}$, we always have $m_1=m_2$, see \cites{Cartan1939, Abresch}; when $g=2$, we have $m_1=m_2$ when the isoparametric hypersurfaces are the product of two spheres of the same dimension.
For Type-II solutions in these cases, the symmetry $t \leftrightsquigarrow \sqrt{1-t^2}$ implies that $f'(\frac{1}{\sqrt{2}})=0$ and that $f$ reaches its peak at $\frac{1}{\sqrt{2}}$.
On the other hand, the isoparametric hypersurface $M_s$ corresponding to $\sin \frac{gs}{2} =: t = \frac{1}{\sqrt{2}}$ is the distinguished \emph{minimal} hypersurface, due to \eqref{eqn:h-sigma} and the condition $m_1=m_2$.
These properties will be used in the proof of the smooth interpolation Theorem~\ref{thm:capillary-interpolation}.

\section{Topology of the minimal surfaces}\label{section:topology}

We now discuss the topological properties of the links $\Sigma_{M_i, \theta}, \bar{\Sigma}_{M,\theta}$ of the capillary cones constructed in Theorem~\ref{thm:capillary-interpolation}.
Applying this analysis to the new minimal surfaces $\mathbf{S}_{M_1}, \mathbf{S}_{M_2}, \bar{\mathbf{S}}_M$ of Theorem~\ref{thm:new-minimal-surfaces}, we obtain examples with novel, sophisticated topologies in the sphere.

As demonstrated in~\cites{FTW-1, new-minimal-surfaces}, there are two regimes of profile curves $f$ solving equation~\eqref{eqn:rescaled-equation-lambda-ODE} that produce one-phase free boundaries or capillary surfaces.
Their existence is given by finding a solution of equation~\eqref{eqn:ode-star} with zeroes at $t_i$ and prescribed value $(1-t_i^2)f'(t_i)^2$.
If $g \geq 2$, then the function $A_M(t)$ produces a singularity at $t=0$ in equation~\eqref{eqn:ode-star}, so smooth solutions existing up to $t=0$ require $f'(0) = 0$ and must be even.
The term $(1-t^2)$ also produces a singularity at $t=1$, resulting in the second condition of Proposition~\ref{prop:mean-curvature-equation}.
The solutions whose positive phase includes $t=0$ or $t=1$ correspond to hypersurfaces defined across a focal submanifold $M_1$ or $M_2$ of the isoparametric foliation.
We call this the Type I case, and refer to the alternative regime $0 < t_1 < t_2 < 1$ as Type II.
These solutions describe the following situations:
\begin{enumerate}[(I)]
    \item Type I: the domain of $f$ is a cap around a focal submanifold, meaning that $\Omega = [0,t_1)$ or $\Omega = (t_1, 1]$ for $t_1 \in (0,1)$.
    The symmetry of equation~\eqref{eqn:ode-star} under the involution $(m_1, m_2, t) \leftrightsquigarrow (m_2, m_1, \sqrt{1-t^2})$ described in Lemma~\ref{lemma:change-of-variables} shows that we can study solutions of Type I starting from $t=0$ with $f'(0) = 0$, hence extending smoothly across the focal submanifold.
    Moreover, Lemma~\ref{lemma:reach-zero-before-hypergeometric} shows that the positive phase must be a strict sub-interval of $[0,1)$, so the free boundary of the solution is precisely the regular leaf $\{ t = t_1 \}$ with capillary condition $\sqrt{1-t_1^2} \, f'(t_1) = -\frac{2}{g} \tan \theta$.
    
    \item Type II: The domain is a strip between two regular leaves, meaning that $\Omega = (t_1, t_2)$ for $t_1, t_2 \in (0,1)$.
    The capillary minimal cone problem for solutions of this type is equivalent to finding a solution $f$ of equation~\eqref{eqn:ode-star} satisfying $f(t_i) = 0$ and $(1-t_i^2) f'(t_i)^2 = \frac{4}{g^2} \tan^2 \theta$, namely the condition~\eqref{eqn:capillary-angle-condition}.
    We will approach the construction of such solutions as an ODE shooting problem to find the appropriate $t_1$ for each angle $\theta$. 
\end{enumerate}

\begin{figure}
    \centering
    \includegraphics[width=0.3\linewidth]{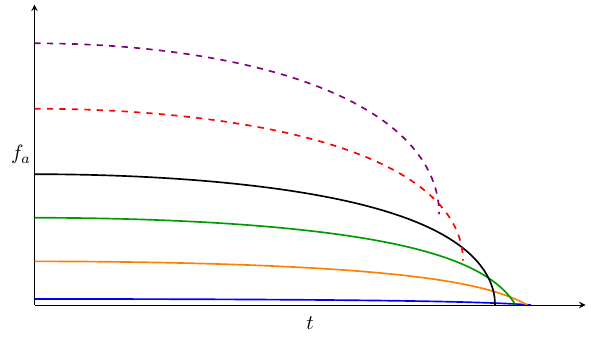}
    \includegraphics[width=0.3\linewidth]{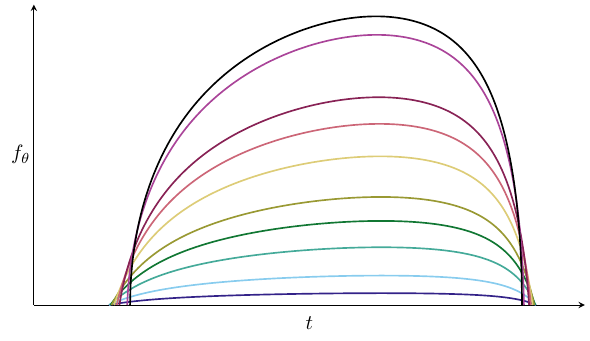}
    \includegraphics[width=0.3\linewidth]{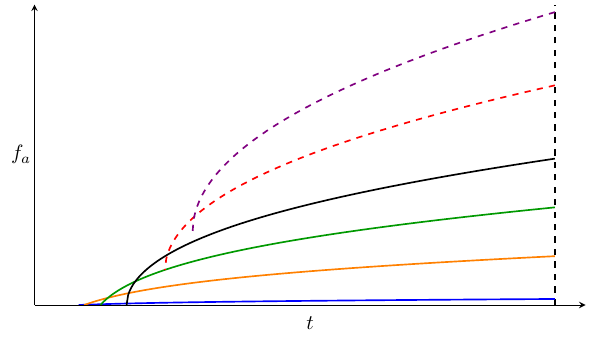}
    \caption{We display numerically computed profile curves for varying angle Type I and II cones for $(g,m_1, m_2) = (4,2,5)$ corresponding to $\Sigma_{M_1, \theta}, \bar{\Sigma}_{M,\theta}, \Sigma_{M_2,\theta}$ respectively.
    The dashed lines indicate non-geometric Type I solutions to~\eqref{eqn:ode-star} which blow up before attaining 0.
    The smoothness condition for $\Sigma_{M_2,\theta}$ requires the compatibility condition $f_{M_2,\theta}'(1) =  \frac{4(n-1)}{g^2(m_2+1)}f_{M_2,\theta}(1) > 0$ on the terminal data, which can be observed visually.}
    \label{fig:typeI&IIconeProfiles}
\end{figure}

In the Type II case, every leaf $M_s$ is a smooth parallel of the initial isoparametric surface.
In contrast, the Type I solutions contain the focal submanifolds $M_1$ or $M_2$, which is a collapsed degeneration of $M_s$.
If the involution $t \leftrightsquigarrow \sqrt{1 - t^2}$ and $m_1 \leftrightsquigarrow m_2$ is not an isometry, $M_1$ and $M_2$ are different surfaces, the two leaves produce genuinely different Type I examples.
This fact was explored in~\cite{FTW-1}, where the cones $\mathbf{C}_{n,k,\theta}$ and $\mathbf{C}_{n,n-k,\theta}$ or the one-phase solutions $U_{n,k}$ and $U_{n,n-k}$ have different topology and variational behavior.
Notably, $U_{7,1}, U_{7,2}$ are proved to be strictly minimizing, while $U_{7,5}, U_{7,4}$ do not appear to be; the same variational phenomena hold for capillary cones with sufficiently small angle. 
For large capillary angle near $\frac{\pi}{2}$, these distinct behaviors do not occur because the doubles of $\mathbf{C}_{n,k,\frac{\pi}{2}}$ or $\mathbf{C}_{n,n-k,\frac{\pi}{2}}$ are the same minimal hypercone $C(\bS^{k-1} \times \bS^{n-k-1})$, so the capillary cones with $|\tfrac{\pi}{2} - \theta| < \ve$ will have the same variational behavior with respect to stability and area-minimality. 

\subsection{Type I}\label{section:type-I-topology}
The Type I solutions have profile curves emanating from a focal leaf and ending at an interior zero on a regular leaf.
The topology of the resulting spherical hypersurface is determined by the corresponding focal submanifold $M_1 \subset \bS^{n-1}$ at $t=0$; the discussion for $t=1$ is analogous.
We recall from Section~\ref{subsec:isoparametric} that every regular isoparametric leaf is a tube around each focal submanifold, and the sphere decomposes as the union of the two disk bundles over the two focal manifolds.
Consequently, each link $\Sigma_{M_i, \theta}$ in Theorem~\ref{thm:capillary-interpolation} is topologically the disk bundle $D(\nu_{M_i}) \subset \bS^{n-1}$, where $\nu_{M_i}$ denotes the normal bundle of $M_i \subset \bS^{n-1}$.
The boundary of the link is the regular leaf at the zero $t_*$, namely $\partial \Sigma_{M_i, \theta} \cong M_{t_*} = S(\nu_{M_i})$, the unit sphere bundle of the normal bundle $\nu_{M_i} \to M_i$ of the embedding $M_i \subset \bS^{n-1}$.
The minimal surface $\mathbf{S}_{M_i} \subset \bS^n$ in Theorem~\ref{thm:new-minimal-surfaces} formed by doubling $\Sigma_{M_i,\frac{\pi}{2}}$ across the equatorial hypersphere is therefore diffeomorphic to
\begin{equation}\label{eqn:S-Mi-diffeomorphism-type}
\mathbf{S}_{M_i} \cong D(\nu_{M_i}) \cup_{M} D(\nu_{M_i}) \cong S(\nu_{M_i} \oplus \mathbf{1})
\end{equation}
where $\mathbf{1}$ is the trivial line bundle, and $S(L)$ is the unit sphere bundle of a vector bundle $L$.
Notably, the resulting surface is topologically an $\bS^{m_1+1}$-bundle over the focal submanifold $M_1$.
The sphere bundle structure comes from the fiberwise realization of $\bS^{m_1+1}$ as the doubling of the disk $\bD^{m_1+1}$.
The leaves $M_s$ are given by the points of distance $s$ away from $M_1$, so we can describe the topology of $\mathbf{S}_{M_1}$ through the following procedure. 
The profile curve $f_{M_1,\theta}$ vanishes at some $t_{M_1,\theta} < 1$, so the link $\Sigma_{M_1,\theta} \cong M_1 \cup \left(\bigcup_{s \leq s_{M_1,\theta}}M_s\right)$ is a closed tubular neighborhood of all points within a fixed distance of $M_1$ inside $\bS^n$, hence diffeomorphic to the disk bundle of the normal bundle $\nu_{M_1} \to M_1$.

The minimal and CMC surfaces $\mathbf{S}_{M_i},  \mathbf{S}^H_{M_i}$ are therefore realized as $\bS^{m_i+1}$ bundles over $M_i$, so identifying their topology requires a precise understanding of this bundle structure.
Several directions and techniques in topology have been created to answer these types of questions, and we employ methods from obstruction theory, characteristic classes, $K$-theory, and stable homotopy theory to distinguish the topology of the bundles $\mathbf{S}_{M_i}$.
Thus, we fully classify the resulting sphere bundles $\mathbf{S}_{M_i}$, showing that the $\mathbf{S}_{M_i}$ are not homotopy equivalent to products over $M_i$ outside certain exceptional cases, and realize a large range of topological behaviors.
In most cases, $\mathbf{S}_{M_i}$ is very different from a product $M_i \times \bS^r$ in the strongest sense: the two spaces are not even \textit{stably homotopy equivalent}.
This equivalence, denoted $X \simeq_s Y$, means that $\Sigma^{\infty} X_+ \simeq \Sigma^{\infty} Y_+$ are homotopy equivalent as spectra; see, for example,~\cite{hatcher-book-alg-top}*{\S 4.F}.
This notion is significantly broader than homotopy equivalence $X \simeq Y$.

The main theoretical tools in this classification are Lemmas~\ref{lemma:obstructions-to-triviality} and~\ref{lemma:nontrivial-J-map}, obstructing stable homotopy equivalence via the Adams $J$-map, and Lemmas~\ref{lemma:cohomology-ring} and~\ref{lemma:fiberwise-homotopy-equivalent}, using the cohomology ring structure of the sphere bundle and the Serre spectral sequence to upgrade homotopy equivalence to \textit{fiber} homotopy equivalence, discussed in Definition~\ref{def:fiber-homotopy-trivial}.
In Proposition~\ref{prop:all-nontrivial-isoparametric}, we combine these tools together with computations from the Atiyah-Hirzebruch spectral sequence in $K$-theory and the study of stable homotopy groups to obstruct stable homotopy equivalence.

Therefore, the novel surfaces of Theorems~\ref{thm:new-minimal-surfaces} and~\ref{thm:non-triviality-theorem} produce more sophisticated and identifiable topologies, coming from the list of Section~\ref{subsec:isoparametric}; see Table~\ref{table:diffeomorphism-types} for some examples.
We first present a summary of the topological properties of the surfaces $\mathbf{S}_{M_i} \to M_i$ in~\eqref{eqn:S-Mi-diffeomorphism-type}, which are obtained in Lemmas~\ref{lemma:obstructions-to-triviality} through~\ref{lemma:fiberwise-homotopy-equivalent} and Propositions~\ref{prop:g=3Topology} through~\ref{prop:surfaces-are-distinct}.
\begin{enumerate}[$\bullet$]
\item $g=1$: The spherical (axisymmetric) case produces only the $\textup{O}(n-1)$-invariant capillary cones of~\cite{FTW-1}.
Accordingly, the Clifford torus $\bS^{n-2}(\sqrt{\frac{n-2}{n-1}}) \times \bS^1( \sqrt{\frac{1}{n-1}})$ and the sphere $\bS^{n-1}$ are the only rotationally invariant minimal hypersurfaces of $\bS^n$, as proved in Theorem~\ref{thm:uniqueness-of-axisymmetric}. 
\item $g=2$: In the Clifford case, the focal submanifolds are $(M_1, M_2) = (\bS^{n-k-1}, \bS^{k-1})$ and $\mathbf{S}_{M_1} , \mathbf{S}_{M_2}$ are Clifford tori in one dimension higher.
The profiles $f_{M_1, \frac{\pi}{2}}(t), f_{M_2, \frac{\pi}{2}}(t)$ are explicit, with
\[
\mathbf{S}_{M_1} = \bS^{n-k-1} \Bigl( \sqrt{\tfrac{n-k-1}{n-1}} \Bigr) \times \bS^k \Bigl( \sqrt{\tfrac{k}{n-1}} \Bigr), \qquad \mathbf{S}_{M_2} = \bS^{n-k} \Bigl( \sqrt{\tfrac{n-k}{n-1}} \Bigr) \times \bS^{k-1} \Bigl( \sqrt{\tfrac{k-1}{n-1}} \Bigr).
\]
\item $g=3$: The focal submanifolds are Veronese embeddings $M_i \cong \bF \bP^2$, where $\bF \in \{ \mathbb{R} , \mathbb{C}, \mathbb{H}, \mathbb{O} \}$ with $\dim_{\bR}\bF=m\in\{1,2,4,8\}$, so $\mathbf{S}_{M_1} \cong \mathbf{S}_{M_2}$ are non-trivial $\bS^{m+1}$-bundles over $\bF \bP^2$ inside $\bS^{3m+2}$, which are in fact isometric due to Lemma~\ref{lemma:change-of-variables} and the isometry between the two Veronese embeddings $M_1 \cong M_2$ of $\bF \bP^2$ inside $\bS^{3m+2}$.
In Proposition~\ref{prop:g=3Topology}, we identify the surfaces $\mathbf{S}_{\bF} := \mathbf{S}_{M_i}$ as twisted products of Lie groups with dimension $3m+1 \in \{4, 7, 13, 25 \}$ in $\bS^{3m+2}$.
\item $g=4$: We examine the pairs $(M_1, M_2)$ of focal submanifolds in the OT-FKM family, as well as the two exceptional cases $(m_1, m_2) \in \{ (2,2), (4,5) \}$ which exhibit distinct behaviors.
\begin{enumerate}[$-$]
    \item OT-FKM $M_1$: Since the normal bundle of $M_1$ in the ambient sphere is trivial, we always have $\mathbf S_{M_1}\cong M_1\times \bS^{m_1+1}$.
Moreover, $M_1\cong S(\eta)\to \bS^{\ell-1}$, so $\mathbf S_{M_1}\cong S(\eta)\times \bS^{m_1+1}$.
The bundle $\eta \to \bS^{\ell-1}$ is trivial if and only if
\begin{equation}\label{eqn:pairs-trivial-bundle}
(m_1,m_2)\in
\{(1,2),(2,1),(1,6),(6,1),(2,5),(5,2),(3,4)\} \qquad \text{or} \qquad (m_1,m_2) = (4,3)^{\textup{ind}}
\end{equation}
where $(4,3)^{\textup{ind}}$ denotes the foliation coming from the indefinite Clifford representation with $(m_1,m_2) = (4,3)$ and index $q=0$, defined in~\eqref{eqn:index-representation} below.
Equivalently, $M_1$ is diffeomorphic to $\bS^{\ell-1}\times \bS^{\ell-m_1-1}$ exactly in those cases, where $\ell = k \, \delta(m_1)$.
In particular, for $m_1=1$,
\[
M_1\cong \bV(2, m_2 +2),
\qquad
\mathbf S_{M_1}\cong \bV(2, m_2 + 2) \times \bS^2.
\]
When $m_2=2,6$, the Stiefel manifold is a product of spheres, leading to
\[
\bV(2,4)\cong \bS^3\times \bS^2,
\quad
\bV(2,8)\cong \bS^7\times \bS^6 \quad 
 \implies \quad 
\mathbf S_{M_1}\cong \bS^3\times \bS^2\times \bS^2,
\quad
\mathbf S_{M_1}\cong \bS^7\times \bS^6\times \bS^2.
\]
For the other $m_1=1$ cases, $\bV(2,m_2+2)$ is not a product of spheres.
    \item OT-FKM $M_2$: Either $\mathbf{S}_{M_2}$ is not homotopy equivalent to any product $M_2 \times \bS^r$, or $M_2 \cong \bS^{m_1} \times \bS^{m_1 + m_2}$ and $\mathbf{S}_{M_2} \cong \bS^{m_1} \times \bS^{m_1+m_2} \times \bS^{m_2+1}$ as smooth sphere bundles, with
\[
(m_1,m_2)\in\{ (1, 2d) , (2d,1) , (2,5),(5,2),(3,4) , (4,3) \}, \qquad d \in \bN^*.
\]
    See Lemma~\ref{lemma:ot-fkm-family-I} and Proposition~\ref{prop:trivial-product-bundles-ot-fkm} for the proof of this classification.
    \item $(m_1,m_2) = (2,2)$: In this exceptional case, we have $M_1\cong \bC\bP^3$ and $M_2\cong \widetilde{\bG}(2,5) \cong Q^3$.
    In either case $\mathbf{S}_{M_i}$ is an $\bS^3$-bundle over the corresponding focal manifold.
    \item $(m_1,m_2) = (4,5)$: In this exceptional case, $M_1\cong \textup{U}(5)/(\textup{Sp}(2)\times \textup{U}(1))$ and $M_2\cong \textup{U}(5)/(\textup{SU}(2)\times \textup{U}(3))$, so $\mathbf S_{M_1}$ is an $\bS^5$-bundle over $M_1$, while $\mathbf S_{M_2}$ is an $\bS^6$-bundle over $M_2$.
\end{enumerate}

\item $g=6$:
If $m=1$, then $M_1, M_2$ are both diffeomorphic to $\bS^3\times \bR\bP^2$, but are not isometric, and the surfaces $\mathbf{S}_{M_i}$ are non-trivial $\bS^2$-bundles over $\bS^3\times \bR\bP^2$.
Similarly, the minimal surfaces $\mathbf{S}_{M_1}$ and $\mathbf{S}_{M_2}$ are not isometric.
If $m=2$, then $M_1 \cong \widetilde{\bG}(2,7) \cong Q^5$ while $M_2$ is diffeomorphic to the twistor space of $G_2 / \textup{SO}(4)$.
The sphere bundles $\mathbf{S}_{M_i}$ over the focal manifolds are non-trivial and not homotopy equivalent to each other.
\end{enumerate}

The isoparametric hypersurfaces $\{ M_s \}$ of OT-FKM type exhibit an infinite family with complicated non-product topologies.
While $\mathbf{S}_{M_1} \cong M_1 \times \bS^{m+1}$, the focal submanifold $M_1$ is homotopy equivalent to a product of spheres only if $(m_1, m_2)$ belong to the low-dimensional list~\eqref{eqn:pairs-trivial-bundle}.
Moreover, the regular leaves $M$ are not homotopy equivalent to products of spheres for $(m_1,m_2)$ outside this list; see, for example,~\cite{topology-g-4}*{Theorems 2.3 and 2.4}.

\begin{proposition}\label{prop:g=3Topology}
    When $g = 3$, the minimal surfaces $\mathbf{S}_{\mathbb{F}} :=\mathbf{S}_{M_i}$ are topologically
    \begin{align*}
    \mathbf{S}_{\bR} &\cong \textup{SO}(3) \times_{\textup{O}(2)} \bS^2,  & \mathbf{S}_{\bC} &\cong \textup{SU}(3) \times_{\textup{U}(2)}\bS^3, \\ 
    \mathbf{S}_{\bH} &\cong \textup{Sp}(3) \times_{\textup{Sp}(2)\textup{Sp}(1)} \bS^5,  & \mathbf{S}_{\bO} &\cong F_4 \times_{\textup{Spin}(9)}\bS^9,
    \end{align*}
    for $m = 1,2,4,8$ respectively.
    Here, $X \times_G Y$ denotes the twisted product given by the quotient relation $(x,y) \sim (xg, g^{-1}y)$, which is a $Y$-bundle over $X/G$. 
    We also have the description
    \[
\mathbf{S}_{\bR} \cong \bS^2 \times \bS^2/ \sigma, \qquad \sigma(u,v_1,v_2,v_3) = (-u,-v_1, v_2, v_3)
\]
\end{proposition}
\begin{proof}
    For these isoparametric hypersurfaces, the leaves $M_i$ are
    \begin{equation}\label{eqn:symmetric-space-identifications}
    \textup{SO}(3)/(\bZ_2 \oplus \bZ_2),\qquad \textup{SU}(3) /\bT^2,\qquad \textup{Sp}(3)/\textup{Sp}(1)^3,\qquad F_4/\textup{Spin}(8)
    \end{equation}
    as established earlier in Section~\ref{subsec:isoparametric}.
    Fix a basepoint $p \in \bF\bP^2$ and let $G_\bF$ be the transitive symmetry group of $\bF\bP^2$ with $K$ the stabilizer of $p$, so $\bF\bP^2 = G/K$.
    Because $\bF\bP^2$ is $G$-equivariant inside $\bS^{3m+1}$, its normal bundle is $\nu_{\bF\bP^2}\cong G \times_K \nu_{p}$, so $\nu_{\bF\bP^2} \oplus \mathbf{1} \cong G \times_K (\nu_p \oplus \mathbf{1})$.
    Therefore, $\mathbf{S}_{\bF} \cong G \times_K S(\nu_p \oplus \mathbf{1})$, which we now identify.
    For isoparametric surfaces, the regular fibers $M_s$ are identified with $S(\nu_{M_i})$ of the focal manifolds, and~\eqref{eqn:S-Mi-diffeomorphism-type} expresses $S(\nu_p \oplus \mathbf{1}) \cong \Sigma S(\nu_p) \cong \bS^{m+1}$ for $\Sigma X$ the spherical suspension of $X$. 
    The proof concludes by the standard identification of each $\bF\bP^2$ with the corresponding symmetric space as in~\eqref{eqn:symmetric-space-identifications}.

    To see the final characterization of $\mathbf{S}_{\bR}$, we argue as in~\cite{CAGstablyparallel}.
    For $(g,m) = (3,1)$, the Veronese embedding $M_1 = \bR \bP^2 \hookrightarrow \bS^4$ with normal bundle $\nu \to M_1$ satisfies $T \bR \bP^2 \oplus \mathbf{1} = 3 \eta$, where $\eta$ is the Hopf line bundle over $\bR \bP^2$, hence $3 \eta \oplus \nu = \mathbf{1}^{\oplus 5}$ and $\nu \oplus \mathbf{1}^{\oplus 4} \cong \eta \oplus \mathbf{1}^{\oplus 5}$ due to $\eta^{\oplus 4} \cong \mathbf{1}^{\oplus 4}$.
    Recall from~\cite{milnor-stasheff} that rank-$3$ vector bundles over $\bR \bP^2$ are classified by the Stiefel-Whitney classes $w_1, w_2$; in our situation, $w(\eta) = 1+a$ where $H^{\bullet}(\bR \bP^2; \bZ_2) \cong \bZ_2[a]/(a^3)$.
    Moreover, $w(T \bR \bP^2) = (1+a)^3= (1+a)^{-1}$, so
    \[
    w(\nu \oplus \mathbf{1}) = w(\nu) = w( T \bR \bP^2)^{-1} = 1+a = w(\eta) =  w(\eta \oplus \mathbf{1}^{\oplus 2}).
    \]
    Therefore, $\nu \oplus \mathbf{1} \cong \eta \oplus \mathbf{1}^{\oplus 2}$ and $\mathbf{S}_{\bR} \cong S( \eta \oplus \mathbf{1}^{\oplus 2}) \cong \bS^2 \times \bS^2 / \sigma$  is the claimed involution quotient.
\end{proof}
We next show that the surfaces $\mathbf{S}_{M_i}$ in~\eqref{eqn:S-Mi-diffeomorphism-type} are not homotopy equivalent to products $M_i \times \bS^r$, further illustrating their topological complexity.
Given a sphere bundle $\bar{N} = S(\xi) \xrightarrow{\pi} N$ with fiber $\bS^{r-1}$ and underlying rank-$r$ vector bundle, we can express
\begin{equation}\label{eqn:sphere-bundle-stably-parallelizable}
    T \bar{N} \oplus \mathbf{1} \cong \pi^* ( TN \oplus \xi)
\end{equation}
A vector bundle $L \to B$ is \textit{stably parallelizable} if $L \oplus \mathbf{1}^{\oplus k} \cong \mathbf{1}^{\oplus k'}$ is trivial for some $k,k' \geq 1$.
We say that $B$ is stably parallelizable if its tangent bundle $L = TB$ is, the latter property being equivalent to the triviality of $L \oplus \mathbf{1}$.
Notably, all spheres are stably parallelizable, with $T \bS^{m-1} \oplus \mathbf{1} \cong \mathbf{1}^{\oplus m}$.

To study \textit{stable} homotopy equivalence, we consider the stable $J$-homomorphism introduced by Adams \cite{adams}.
We refer the reader to~\cite{infinite-loop-space}*{\S 14} and~\cite{ranicki}*{\S 9.3} for background results used in the following Proposition.
Let $BG$ denote the classifying space for stable spherical fibrations, which is an infinite loop space with $BG \simeq \Omega^{\infty}E$ for a spectrum $E$.
Let $BO := \varinjlim BO(n)$ denote the classifying space for stable real vector bundles, meaning that stable real vector bundles $E^s \to M$ are classified by maps $f_E : M \to BO$.
The \textit{Adams $J$-class} of a bundle $E \to M$ is defined as
\begin{equation}\label{eqn:adams-j-class}
    J(E) := [J \circ f_E] \in [M,BG],
\end{equation}
for $J: BO \to BG$ the canonical map from stable vector bundles to stable spherical fibrations.

We also use the properties of characteristic classes from~\cite{milnor-stasheff}.
Given a real vector bundle $E$, we denote by $c(E_{\bC}), p(E)$, and $w(E)$ the total Chern, Pontryagin, and Stiefel-Whitney classes.
\begin{lemma}\label{lemma:obstructions-to-triviality}
    The hypersurface $\mathbf{S}_{M_i} \subset \bS^n$ has the following properties.
    \begin{enumerate}[(i)]
        \item If there is a diffeomorphism $\mathbf{S}_{M_i} \cong M_i \times \bS^r$, then $M_i$ is stably parallelizable.
        \item If there is a homotopy equivalence $\mathbf{S}_{M_i} \simeq M_i \times \bS^r$, then $w(TM_i) = 1$.
        \item If there is a stable homotopy equivalence $\mathbf{S}_{M_i} \simeq_s M_i \times \bS^r$, then $J(TM_i) = 0$.
    \end{enumerate}
\end{lemma}
\begin{proof}
    First, we recall that any smooth embedded hypersurface $\Sigma^{n-1} \subset\bS^{n}$ is stably parallelizable.
Indeed, we have $T \Sigma \oplus \eta \cong T \bS^{n}|_{\Sigma}$ for $\eta$ the normal bundle of $\Sigma^{n-1}$, and since the sphere is stably parallelizable, we can write $T \bS^n \oplus \mathbf{1} \cong \mathbf{1}^{\oplus (n+1)}$.
Restricting to $\Sigma$, we find
\begin{equation}\label{eqn:sphere-restriction}
T \bS^n|_{\Sigma} \oplus \mathbf{1} \cong \mathbf{1}^{\oplus (n+1)}|_{\Sigma}, \qquad \implies \qquad T \Sigma \oplus \eta \oplus \mathbf{1} \cong \mathbf{1}^{\oplus(n+1)}.
\end{equation}
Since $\Sigma^{n-1} \subset \bS^n$ is a closed embedded hypersurface its normal bundle $\eta$ is trivial, so the above property implies $T \Sigma \oplus \mathbf{1}^{\oplus 2} \cong \mathbf{1}^{\oplus (n+1)}$ and $\Sigma$ is stably parallelizable as claimed.
Applying this result in our situation, we find that $\mathbf{S}_{M_i}$ is stably parallelizable.

\smallskip \noindent \textbf{Part $(i)$.}
Given a diffeomorphism $\mathbf{S}_{M_i} \cong M_i \times \bS^r$, we consider the projection maps $\textup{pr}_1, \textup{pr}_2$ to write
\[
T\mathbf{S}_{M_i} = \textup{pr}_1^* TM_i + \textup{pr}_2^*T\bS^r \implies T\mathbf{S}_{M_i} \oplus \mathbf{1}  \cong \textup{pr}_1^* (TM_i \oplus \mathbf{1}) \oplus \mathbf{1}^{\oplus (r+1)}
\]
since $T \bS^r \oplus \mathbf{1} \cong \mathbf{1}^{\oplus (r+1)}$.
Pulling back by the section $s: M_i \to M_i \times \bS^r$ given by $s(x) = (x,p_0)$, we conclude that $TM_i \oplus \mathbf{1}$ is the trivial bundle and $M_i$ is stably parallelizable.

\smallskip \noindent \textbf{Part $(ii)$.}
We prove that $M \times \bS^r \simeq Y$ for $Y = \mathbf{S}_{M_i}$ a stably parallelizable manifold implies $w(TM) = 1$.
Stably parallelizable manifolds satisfy $TZ \oplus \mathbf{1}^{\oplus k} \cong \mathbf{1}^{\oplus (\dim Z + k)}$ for some $k$, so they have trivial total Stiefel-Whitney class, $w(TZ) = 1$. 
By Wu's theorem~\cite{milnor-stasheff}, Stiefel-Whitney classes are homotopy invariants of closed manifolds, so $w(T(M \times \bS^r)) = 1$.
Writing $T(M \times \bS^r) \cong \textup{pr}_1^* TM \oplus \textup{pr}_2^* T \bS^r$ and recalling that $w(T \bS^r) = 1$ by stable triviality, we find
\[
1 = w(T ( M \times \bS^r)) =\textup{pr}_1^* w(TM) \smile \pi^*_2 w(T \bS^r) =\textup{pr}_1^* w(TM)
\]
which forces $w(TM) =1$ as desired.

\smallskip \noindent \textbf{Part $(iii)$.}
We prove that $M \times \bS^r \simeq_s Y$ for $Y = \mathbf{S}_{M_i}$ a stably parallelizable manifold implies $J(TM) = 0$.
Recalling the infinite loop space structure of the classifying space $BG$, namely $BG \simeq \Omega^{\infty}E$ for a spectrum $E$, we find that any stable homotopy equivalence $\Sigma^{\infty} Y_+ \simeq \Sigma^{\infty} Z_+$ of finite CW complexes $Y,Z$ induces a bijection
\[
[Y, BG] \cong [\Sigma^{\infty}Y_+, E] \cong [\Sigma^{\infty} Z_+, E] \cong [Z, BG],
\]
hence the class in $[-,BG]$ associated to a stable spherical fibration is a stable homotopy invariant by~\cite{infinite-loop-space}.
For a smooth manifold $Z$, its stable normal bundle $\nu^s_Z$ is mapped to the class of the underlying stable spherical fibration by the stable $J$-homomorphism.
Since $Y$ is stably parallelizable, we have $\nu^s_Y = 0$ and $J(\nu^s_Y) = 0 \in [Y, BG]$, so $Y \simeq_s Z$ implies $J(\nu^s_Z) = 0 \in [Z,BG]$.
Moreover, $(TZ)^s \oplus \nu_Z^s$ is stable trivial, so
\[
J(TZ) = J ( (TZ)^s) = - J( \nu^s_Z) = 0 \in [Z,BG].
\]
Applying this fact to $Z = M \times \bS^r$, we write $T(M \times \bS^r) \cong \textup{pr}_1^* TM \oplus \textup{pr}_2^* T \bS^r$ where $T \bS^r \oplus \mathbf{1} \cong \mathbf{1}^{\oplus (r+1)}$ and $J(T \bS^r) = 0$.
Therefore, $0 = J (T(M \times \bS^r)) =\textup{pr}_1^*J (TM)$ implies $J(TM) = 0$ as claimed.
\end{proof}

Manifolds with trivial total Stiefel-Whitney class are orientable and spin due to $w_1(TM) = w_2(TM) = 0$.
Also, the top Stiefel-Whitney class $w_n$ of a closed manifold $M$ satisfies
\begin{equation}\label{eqn:top-stiefel-whitney}
    \la w_n(TM) , [M] \rg \equiv \chi(M) \quad (\on{mod} \; 2)
\end{equation}
when paired with the fundamental class $[M] \in H_n(M;\bZ)$.
Thus, a closed manifold $M$ cannot be stably parallelizable if $\chi(M)$ is odd.
Moreover, complex vector bundles $E$ satisfy
\begin{equation}\label{eqn:complex-vector-bundle}
    w_2(E_{\bR}) \equiv c_1(E) \quad (\on{mod} 2), \qquad \implies \qquad w_2(TX) = c_1(TX) \quad (\on{mod} 2)
\end{equation}
where $c_1$ denotes the first Chern class and $X$ is any complex manifold.

\begin{lemma}\label{lemma:nontrivial-J-map}
Consider a smooth manifold $M$.
\begin{enumerate}[(i)]
    \item If $w_1(TM) \neq 0$, then $J(TM) \neq 0$.
    \item If there exists a map $f: \bS^2 \to M$ with $f^* w_2(TM) \neq 0$, then $J(TM) \neq 0$.
    In particular, if $M$ admits an almost-complex structure and $\la c_1(TM), f_* [ \bS^2] \rg \equiv 1 \; (\on{mod} 2)$, then $J(TM) \neq 0$.
    \item If there exists a map $f: \bS^4 \to M$ with $\la p_1(TM), f_* [\bS^4] \rg \not\equiv 0 \; (\on{mod} 48)$, then $J(TM) \neq 0$.
    If $M$ admits an almost-complex structure and $\la c_2(TM), f_* [ \bS^4] \rg \not\equiv 0 \; (\on{mod} 24)$, then $J(TM) \neq 0$.
\end{enumerate}
\end{lemma}
\begin{proof}
For $(i)$, we recall that Stiefel-Whitney classes for vector bundles extend to spherical fibrations, and agree with the usual ones for spherical fibrations coming from a vector bundle by~\cite{milnor-spherical-fiber-bundles}.
In particular, the class $w_1$ is natural under pullback and vanishes on the trivial spherical fibration, so $J(TM) = 0 \implies w_1(TM) =0$.
Concretely, if $w_1(TM) \neq 0$ then there exists a loop $\alpha: S^1 \to M$ with $\alpha^* w_1(TM) \neq 0 \in H^1(\bS^1 ; \bZ_2)$, meaning that $\alpha^* TM$ is non-orientable over $\bS^1$.
Real line bundles over CW complexes are classified by $H^1(-;\bZ_2)$, and $\widetilde{KO}(\bS^1) \cong \bZ_2$ is generated by the M\"obius line bundle $\mu$, producing the stable isomorphism $\alpha^* TM \cong_s \mu \oplus \mathbf{1}^{(r-1)}$; see, for instance,~\cite{hatcher}.
The spherical fibration over $\bS^1$ associated to $\mu$ has monodromy $-1$, so $J(\mu) \neq 0$ is non-trivial and $\alpha^*J(TM) = J(\alpha^*TM) = J(\mu) \neq 0$ implies $J(TM) \neq 0$.

For $(ii)$, we recall from~\cites{milnor-spherical-fiber-bundles , hatcher ,  ranicki} that the $J$-class over a sphere $\bS^r$ is identified with the image of the stable $J$-homomorphism $\pi_{r-1}(O) \to \pi^S_{r-1}$, which is injective for $r \equiv 1,2 \; (\on{mod} 8)$ by Adams~\cite{adams}.
For $r=2$, we have $\pi_1(O) \cong \bZ_2$ and $\pi_1^S \cong \bZ_2$, generated by the Hopf element; therefore, the non-trivial stable bundle on $\bS^2$ has non-trivial $J$-class.
Since $w_2(E \oplus \mathbf{1}^{\oplus r}) = w_2(E)$ for oriented bundles, the property $f^* w_2(TM) \neq 0$ implies that $f^* TM \to \bS^2$ is the non-trivial stable real line bundle over $\bS^2$, so $f^* J(TM) = J(f^* TM) \neq 0 \in [ \bS^2, BG]$ implies $J(TM) \neq 0$.
Finally, recall that almost complex manifolds have $w_2(TM) \equiv c_1(TM) \; (\on{mod} 2)$, so the assumption automatically implies $f^* w_2(TM) \neq 0$, so $J(TM) \neq 0$ by the same argument.

The argument for $(iii)$ proceeds similarly: the pullback $f^* TM$ of the stable tangent class lies in $\widetilde{KO}(\bS^4) \cong \pi_3 (O)$, and $f^* J(TM)$ is its image under the stable $J$-homomorphism.
Adams' theorem~\cite{adams} shows that the image of $J$ in degree $4s-1$ is cyclic of order $m(2s)$, the denominator of $\frac{B_s}{4s}$.
Since $\pi_3 (O) \cong \bZ$ by~\cite{hatcher} and $m(2) = 24$, the stable $J$-map on $\bS^4$ sends a generator of $\widetilde{KO}(\bS^4) \cong \bZ$ to an order-$24$ element.
The Thomas isomorphism~\cite{thomas-class} identifies $\widetilde{KO}(\bS^4) \cong H^4(\bS^4;\bZ)$ via the spin characteristic class $\frac{p_1}{2}$, so the stable class of $f^* TM$ is encoded in the integer $\la \frac{p_1(f^* TM)}{2} , [\bS^4] \rg =\la \frac{p_1(TM)}{2} , f_* [\bS^4] \rg \; (\on{mod} 24)$.
For $M$ almost complex, the Pontryagin class has
\[
p_1(E_{\bR}) = c_1(E)^2 - 2 c_2(E) \implies p_1(f^* TM) = - 2 \, c_2(f^* TM)
\]
because $c_1$ vanishes on $\bS^4$ since $H^2(\bS^4;\bZ) = 0$.
This proves the claim.
\end{proof}

We now explain how to identify the topological nature of the sphere bundles, starting with obstructions from their cohomology ring and characteristic classes.

\begin{lemma}\label{lemma:cohomology-ring}
    Let $\nu \to M$ be a rank-$r$ real vector bundle over a closed manifold $M$ and consider the $\bS^r$-bundle $\pi : Y :=S( \nu \oplus \mathbf{1}) \to M$.
    Let $s : M \to Y$ be the north pole section
    \[
    s: M \to Y, \qquad s(x) = (0,1) \in \bS^r \cong S(\nu_x \oplus \mathbf{1})
    \]
    and let $w_r(\nu) \in H^r(M)$ be the top Stiefel-Whitney class of $\nu$.
    The cohomology ring of $Y$ is given by
    \begin{equation}\label{eqn:cohomology-ring-of-Y}
        H^{\bullet}(Y ; \bZ_2) = H^{\bullet}(M; \bZ_2)[u]/ (u^2 - w_r(\nu) u), \qquad |u|=r,
    \end{equation}
    for $u := \textup{PD}_Y(s(M)) \in H^r(Y;\bZ_2)$, the Poincar\'e dual of the codimension-$r$ submanifold $s(M) \subset Y$.
    The same result holds for cohomology with integer coefficients if $M, \nu$ are orientable.
\end{lemma}
\begin{proof}
    We use some standard tools from algebraic topology, namely the long exact sequence of a pair, the Thom isomorphism, and the Gysin map; we refer the reader to~\cite{bott-tu}*{\S~11\thru14} for details on the properties used here.
    Moreover, we consider cohomology with $\bZ_2$-coefficients and suppress $H^{\bullet}(Y) := H^{\bullet}(Y;\bZ_2)$, which makes the discussion valid even when $M$ is not orientable; if $M$ is orientable, then the same arguments hold for cohomology with $\bZ$-coefficients.
    
    Consider the inclusion map $j : Y \setminus s(M) \hookrightarrow Y$, where $Y \setminus s(M) \cong \textup{Tot}(\nu)$ is diffeomorphic to the total space of $\nu$, hence deformation-retracts onto the zero section.
    Thus, $H^{\bullet}(Y \setminus s(M)) \cong H^{\bullet}(M)$ and the map $(\pi|_{Y \setminus s(M)})^*$ is an isomorphism, and so is the restriction $j^* \circ \pi^*$, and $j^*$ is surjective.
    Therefore, the long exact sequence of the pair $(Y, Y \setminus s(M))$ 
   \[
   \cdots \to H^k( Y, Y \setminus s(M)) \xrightarrow{i} H^k(Y) \xrightarrow{j^*} H^k( Y \setminus s(M)) \xrightarrow{ \delta} \cdots
   \]
   splits into short exact sequences for each $k$.
   Since $\nu$ is canonically the normal bundle of $s(M)$ in $Y$, a tubular neighborhood of $s(M) \subset Y$ is the disk bundle $D(\nu)$, we have $D(\nu) \setminus s(M) \cong S(\nu)$, so
   \[
   H^k (Y, Y \setminus s(M)) \cong H^k( D(\nu), D(\nu) \setminus s(M)) \cong H^k(D(\nu) , S(\nu))
   \]
   by excision.
   Denoting by $q : D(\nu) \to M$ the bundle projection, the Thom isomorphism amounts to the existence of a Thom class $\tau \in H^r(D(\nu), S(\nu)) $ such that the cup product map 
   \[
   H^{k-r}(M) \to H^k(D (\nu), S(\nu)) \cong H^k(Y, Y \setminus s(M)), \qquad a \mapsto q^* a \smile \tau
   \]
   is an isomorphism.
   The image of the Thom class of the normal bundle is the Poincar\'e dual of $s(M)$, namely $u = \text{PD}_Y(s(M))$ per our definition.

   Letting $\bar{u} := i(\tau) \in H^r(Y)$, so the image of $q^* a \smile \tau$ in $H^k(Y)$ is 
   \[
   i( q^* a \smile \tau) = i(q^* a) \smile i(\tau) = \pi^* a \smile \bar{u}
   \]
   because the projection to $M$ agrees with $\pi$ on $N \cong D(\nu)$.
   In particular, the image $q^* a$ on the tubular neighborhood agrees with the restriction of $\pi^* a$ from $Y$.
   Hence, the natural map $H^{\bullet}(Y, Y \setminus s(M)) \to H^{\bullet}(Y)$ under the above identifications is $a \mapsto \pi^* a \smile u$, and since the long exact sequence for the pair splits into direct sum decompositions for each $k$, we obtain
   \begin{equation}\label{eqn:additive-splitting}
       H^{\bullet}(Y) \cong \pi^* H^{\bullet}(M) \oplus u \, \pi^* H^{\bullet-r}(M) \, .
   \end{equation}
   To obtain the cohomology structure, consider the Gysin map $s_! : H^{\bullet}(M) \to H^{\bullet+r}(Y)$ of the embedding $s: M \hookrightarrow Y$.
   By the construction of the Thom class, we have $s_!(1) = u$ and $s_!(a) = \pi^* a \smile u$ for all $a \in H^{\bullet}(M)$; moreover, we have the self-intersection formula
   \[
   s^* s_!(a) = w_r(\nu) \smile a \qquad \text{for any } \; a \in H^{\bullet}(M)
   \]
   since the normal bundle of $s(M) \subset Y$ is canonically $\nu$.
   In particular, taking $a=1$ with $s_!(1) = u$ shows that $s^* u = w_r(\nu)$, so the projection formula for the Gysin map produces
   \[
   u^2 = u \smile s_!(1) = s_!(s^* u) = s_!( w_r(\nu)) = \pi^* w_r(\nu) \smile u,
   \]
   so $u^2 - w_r(\nu) u = 0$ after suppressing $\pi^* w_r(\nu) \smile u$ to $w_r(\nu) u$.
   The additive splitting~\eqref{eqn:additive-splitting} uniquely expresses every class in $H^{\bullet}(Y)$ as $\pi^* a + \pi^* b u$ for $a,b \in H^{\bullet}(M)$, and combining this with above relation gives the claimed ring structure of $H^{\bullet}(Y;\bZ_2)$.
   If $M, \nu$ are orientable, the same argument applies to the integral cohomology, replacing $w_r(\nu)$ by the Euler class $e(\nu)$.
\end{proof}

\begin{lemma}\label{lemma:cohomology-complex-quadric}
    The complex quadric $Q^n \subset \bC \bP^{n+1}$ has total Chern class $c (TQ^n) = \frac{(1+h)^{n+2}}{1+2h}$.
\end{lemma}
\begin{proof}
Consider the first Chern class $\bar{h} := c_1 (\cO_{\bC \bP^{n+1}}(1)) \in H^2( \bC \bP^{n+1} ; \bZ)$ of the hyperplane bundle.
Under the inclusion map $i: Q^n \hookrightarrow \bC \bP^{n+1}$, its pullback satisfies $h := i^* \bar{h} = c_1(\cO_{Q^n}(1)) \in H^2( Q^n; \bZ)$, since $\cO_{Q^n}(1) = i^* \cO_{\bC \bP^{n+1}}(1)$.
The Euler sequence on $\bC \bP^{n+1}$ yields 
\[
0 \to \cO_{\bC \bP^{n+1}} \to \cO_{\bC \bP^{n+1}}(1)^{\oplus (n+2)} \to T \bC \bP^{n+1} \to 0, \qquad \implies \qquad c (T \bC \bP^{n+1}) = ( 1 + \bar{h})^{n+2}
\]
for the total Chern class.
Since $Q^n$ is defined by a homogeneous degree-$2$ polynomial, its normal bundle in $\bC \bP^{n+1}$ is $N_{Q^n/ \bC \bP^{n+1}} \cong \cO_{Q^n}(2)$, so the normal bundle exact sequence yields
\[
0 \to TQ^n \to T (\bC \bP^{n+1})|_{Q^n} \to \cO_{Q^n}(2) \to 0.
\]
Since $c(\cO_{Q^n}(2)) = 1 + 2h$, the claim follows from taking Chern classes in the exact sequence.
\end{proof}

\begin{proposition}\label{prop:all-nontrivial-isoparametric}
    For every isoparametric foliation with $g \geq 3$ principal curvatures not in the OT-FKM family, the hypersurfaces $\mathbf{S}_{M_i}$ are not stably homotopy equivalent to products $M_i \times \bS^r$. 
\end{proposition}
\begin{proof}
    We treat the cases $g \in \{3,4,6\}$ in steps, following the description of Section~\ref{subsec:isoparametric}.
    Applying Lemma~\ref{lemma:obstructions-to-triviality}, we seek to prove that $J(TM_i) \neq 0$ for the focal manifolds in question.

\smallskip \noindent \textbf{Step 1: $g=3$.}
For the projective planes $M_i = \bF \bP^2$ with $\bF \in \{ \bR, \bC, \bH, \bO \}$ and $m = \dim_{\bR} \bF \in \{1,2,4,8\}$, we have the stable identity $T \bF \bP^2 \oplus \mathbf{1}^{\oplus m} \cong 3 \gamma_{\bF}$ by~\cite{milnor-stasheff}.
    Here, $\gamma_{\bF}$ denotes the tautological $\bF$-line bundle viewed as a real rank-$m$ bundle.
    Therefore,
    \[
    J (T \bF \bP^2) = 3 \, J(\gamma_{\bF}) = 3 \, J( \gamma_{\bF} - \mathbf{1}^{\oplus m}) \in [ \bF \bP^2, BG]. 
    \]
    We observe that $J (T \bF \bP^2) \neq 0$ is a non-trivial class, whose restriction to the bottom cell $\bF \bP^1 = \bS^m$ is the stable class of the corresponding Hopf attaching map, since $i : \bF \bP^1 \hookrightarrow \bF \bP^2$ induces
    \[
    i^* J(T \bF \bP^2) = J( T \bF \bP^1 \oplus \nu_i) = J (T \bS^d) + J(\nu_i) = J(\nu_i)
    \]
    by stable triviality.
    At a line $L \subset \bF^2 \subset \bF^3$, the normal space of $\bF \bP^1 \subset \bF \bP^2$ is $\textup{Hom}_{\bF}(L, \bF^3/\bF^2) \cong \textup{Hom}_{\bF}(L,\bF) \cong L^{\vee}$, so $\nu_i \cong \gamma_{\bF}^{\vee}$ with $S(\nu_i)$ the corresponding Hopf sphere bundle over $\bS^m = \bF \bP^1$.
    Therefore, $J(\nu_i)$ is the attaching map of the top cell of $\bF \bP^2$, namely the stable class of the degree-$2$ attaching map $\bS^1 \to \bS^1$ for $\bR \bP^2$, so $J(T \bR \bP^2) \neq 0$ is detected by the Stiefel-Whitney class $w_1$. 
    For $m \geq 2$, the restriction of the class $J ( T \bF \bP^2) \neq 0$ to $\bF \bP^1 = \bS^m$ is the generator of $\textup{Im}(J) \subset \pi^S_{m-1}$ inside the corresponding stable homotopy group of spheres coming from the $\bF$-Hopf map $\eta_{\bF}: \bS^{2m-1} \to \bS^m$, with $\{ \pi_1^S, \pi^S_3, \pi^S_7 \} = \{ \bZ_2, \bZ_{24}, \bZ_{240} \}$.
    Thus, $J(T \bF \bP^2) \neq 0$ in all cases.

\smallskip \noindent \textbf{Step 2: $g=4$.}
We consider the foliations with $(m_1 , m_2) \in \{ (2,2) , (4,5) \}$ and focal manifolds
\[
M_i \in \{ \bC \bP^3 , Q^3 , \textup{SU}(5)/\textup{Sp}(2) , \textup{U}(5) / ( \textup{SU}(2) \times \textup{U}(3)) \}.
\]
For $\bC \bP^3$, the Euler sequence gives $T \bC \bP^3 \oplus \mathbf{1} \cong 4 \eta$ with $\eta$ the Hopf line bundle.
The computation of Atiyah-Todd~\cite{atiyah-todd} shows that the class of $J(\eta)$ over $\bC \bP^3$ has order $24$, so $J(T \bC \bP^3) = 4 \, J(\eta) \neq 0$.
For $Q^{2n+1}$, the Grassmannian identification~\eqref{eqn:oriented-grassmannian-quadric} shows that $\pi_1(Q^{2n+1}) = 0$ is simply connected and Lemma~\ref{lemma:cohomology-complex-quadric} shows that $c_1(TQ^{2n+1}) = (2n+1)h$, for $h$ a primitive generator of $H^2(Q^{2n+1};\bZ)$.
Therefore, given a map $f: \bS^2 \to Q^{2n+1}$ representing a generator of $H_2(Q^{2n+1}; \bZ)$, we have $\la c_1(TQ^{2n+1}) , f_*[\bS^2] \rg \equiv 1\;(\on{mod} 2)$ and Lemma~\hyperref[lemma:nontrivial-J-map]{\ref{lemma:nontrivial-J-map}$(ii)$} implies $J(TQ^{2n+1}) \neq 0$.

From~\cite{Agricola}, we have the exceptional isomorphism $M_1 =\textup{SU}(5)/\textup{Sp}(2) \cong \textup{SU}(6)/\textup{Sp}(3)$ as manifolds, which is not stably parallelizable from the erratum of Singhof-Wemmer~\cite{singhof-wemmer}.
The classification of compact homogeneous spaces with the integral cohomology of two spheres gives $H^{\bullet}( \textup{SU}(6) / \textup{Sp}(3)) = H^{\bullet}(\bS^5 \times \bS^9) \cong \Lambda(x_5, x_9)$, cf.~\cite{kramer}.
To compute the Grothendieck group $\widetilde{KO}^0(M_1)$, we use the reduced Atiyah-Hirzebruch spectral sequence as in~\cite{atiyah--hirzebruch},
\[
\widetilde{E}^{p,q}_2 = \widetilde{H}^p (M ; KO^q(\ast)) \implies \widetilde{KO}^{p+q}(M),
\]
where $KO^q(\ast)$ is periodic $\on{mod} 8$ and non-zero only for $q \equiv 0,-1,-2,-4 \; (\on{mod} 8)$.
Since $H^{\bullet}(M_1;\bZ) \cong H^{\bullet}(\bS^5 \times \bS^9 ; \bZ)$, the only non-zero reduced cohomology groups occur in degrees $5,9,14$; in particular, the only possible $E_2$-term in the total-degree $0$ diagonal $p+q=0$ is $\widetilde{E}_2^{9,-9} = H^9(M;\bZ_2) \cong \bZ_2$, since $KO^{-5}(\ast) = KO^{-14}(\ast) = 0$.
There are no differentials entering or leaving $(9,-9)$: for $r \neq 9$, $d_r = 0$ for degree reasons, and the only potential source of $d_9$ would be the $p=0$ column, which does not appear in the reduced spectral sequence.
Therefore, $\widetilde{KO}^0(M) \cong \bZ_2$ occurs entirely in filtration $9$, and the stable class $[TM] - 14 \in \widetilde{KO}^0(M) \cong \bZ_2$ is the unique non-zero element because $M$ is not stably parallelizable.
The Adams map $j_M : \widetilde{KO}^0(M) \to J(M)$ induced by $BO \to BG$ acts on the graded piece in filtration $9$ via the group $\pi_9(BO) = \pi_8(O) \cong \bZ_2$, on which it agrees with the stable $J$-homomorphism $J: \pi_8(O) \to \pi^S_8$.
By Adams' theorem~\cite{adams}, this map is injective, so $J(TM) = j_M([TM]-14) \neq 0$ is non-trivial.

Finally, $M_2 = \textup{U}(5) / ( \textup{SU}(2) \times \textup{U}(3))$ has $w_2(TM_2) = 0$ and $w_4(TM_2) \neq 0$ by~\cite{topology-g-4}*{Theorem 2.9}, so it is not stably parallelizable.
Let $B := \bG(2,\bC^5) = \textup{U}(5)/ ( \textup{U}(2) \times \textup{U}(3))$, so the quotient $\textup{U}(2) / \textup{SU}(2) \cong \bS^1$ induces a principal circle bundle
\begin{equation}\label{eqn:principal-circle-bundle}
    \bS^1 \to M_2 \xrightarrow{\psi} B,
\end{equation} 
which is the unit circle bundle of $\det S$, coming from the tautological rank-$2$ bundle $S \to B$.
Moreover, $B$ has a Schubert cell decomposition by even-dimensional cells, with $H^{\bullet}(B)$ generated by the degree-$2$ and degree-$4$ Schubert classes; letting $s_i = c_i(S)$, we have that $s_1^2$ and $s_2$ are linearly independent in degree $4$.
The Euler class $e(\psi) = c_1(\det S) = s_1$ is a primitive generator of $H^2(B;\bZ)$, so $\pi_1(M_2) = \pi_2(M_2) = 0$ by the homotopy exact sequence of the circle bundle~\eqref{eqn:principal-circle-bundle}.
Since $B$ consists of even-dimensional Schubert cells, it has $\pi_3(B) = 0$, so also $\pi_3(M_2) = 0$, meaning that $M_2$ is $3$-connected.
Applying the Gysin sequence to~\eqref{eqn:principal-circle-bundle} as in Lemma~\ref{lemma:cohomology-ring}, we find $H^4(M_2;\bZ) \cong H^4(B;\bZ) / \la s_1^2 \rg \cong \bZ$, so $H^4(M_2;\bZ) = \bZ \cdot \psi^* s_2$ is torsion-free.
Since $M_2$ is $3$-connected and $H_4(M_2) \cong \bZ$, Hurewicz gives $\pi_4(M_2) \cong H_4(M_2) \cong \bZ$, so we can choose $f: \bS^4 \to M_2$ representing a generator to find $\la \frac{p_1(f^*TM_2)}{2} , [\bS^4] \rg = \la \frac{p_1(TM_2)}{2} , f_*[\bS^4] \rg$.
To compute the Pontryagin class, recall that a principal $\bS^1$-bundle has trivial vertical tangent line, so $TM_2 \cong \psi^* TB \oplus \mathbf{1}$, where $TB \cong S^{\vee} \otimes Q$ and $Q$ is the quotient rank-$3$ bundle.
Moreover, 
\begin{equation}\label{pontryagin-class-computation}
p_1(TB_{\bR}) = 3 s_1^2 - 2 s_2 \implies p_1(TM_2) = \psi^* p_1(TB_{\bR}) = \pi^* (3s_1^2 - 2s_2) = - 2 \psi^* s_2
\end{equation}
because $\psi^* s_1 = 0$.
Therefore, $\frac{p_1(TM_2)}{2} = - \psi^*s_2$ is a primitive generator of $H^4(M_2;\bZ)$, so $\la \frac{p_1(TM_2)}{2} , f_*[\bS^4] \rg = \pm 1$ is non-zero $\on{mod} 24$ in the above computation, so $J(TM_2) \neq 0$ by Lemma~\hyperref[lemma:nontrivial-J-map]{\ref{lemma:nontrivial-J-map}$(iii)$}.
The Pontryagin class computation~\eqref{pontryagin-class-computation} for $B = \bG(2, \bC^5)$ follows from Borel-Hirzebruch~\cite{borel-hirzebruch}: pulling back to a splitting space $\tilde{\pi}: \tilde{B} \to B$ where $\tilde{\pi}^* S = L_1 \oplus L_2$ and $\tilde{\pi}^* Q = Q_1 \oplus Q_2 \oplus Q_3$, we write $x_i = c_1(L_i)$ and $y_j = c_1(Q_j)$ to compute
\[
c(S) = (1+x_1)(1+x_2), \qquad c(Q) = (1 + y_1)(1+y_2)(1+y_3), \qquad c(S)c(Q) = 1,
\]
the latter due to $S \oplus Q \cong \mathbf{1}_{\bC}^{\oplus 5} \otimes \cO_B$.
Since $s_1 := c_1(S) = x_1 + x_2$ and $s_2 := c_2(S) = x_1x_2$, we find $c_1(Q) = - s_1$ and $c_2(Q) = s_1^2 - s_2$.
Writing $\tilde{\pi}^* TB \cong \bigoplus_{i,j} L_i^{\vee} \otimes Q_j$, we can apply~\cite{borel-hirzebruch}*{Theorems 10.3 and 10.7} to compute the Pontryagin class as
\allowdisplaybreaks{
\begin{align*}
p_1(TB_{\bR}) &= \prod_{i,j} \bigl( 1 + (y_j - x_i)^2 \bigr) = \sum_{i,j} (y_j - x_i)^2 = 2 \sum_j y_j^2 + 3 \sum_i x_i^2 - 2 \sum_i x_i \sum_j y_j \\
&= 2 \bigl( c_1(Q)^2 - 2 c_2(Q) \bigr) + 3 \bigl( (x_1+x_2)^2 - 2 x_1 x_2 \bigr) - 2 s_1 c_1(Q) \\
&= 2(-s_1^2 + 2s_2) + 3(s_1^2 - 2s_2) - 2(-s_1)(s_1) = 3 s_1^2 - 2 s_2.
\end{align*}}
This proves the expression~\eqref{pontryagin-class-computation}, so $J(TM_2) \neq 0$ by the above argument.

\smallskip \noindent \textbf{Step 3: $g=6$.}
For $m=1$, the focal manifolds are topologically $\bS^3 \times \bR \bP^2$.
Using the stable triviality of $\bS^3$ as in Lemma~\ref{lemma:obstructions-to-triviality}, we find $J( T(\bS^3 \times \bR \bP^2)) = J( T\bR \bP^2) \neq 0$ using Step 1.

For $m=2$, the known focal manifolds are the complex quadric $Q^5$ and $Z = G_2 / \textup{U}(2)^+$, the twistor space of $G_2/ \textup{SO}(4)$.
Using Lemma~\hyperref[lemma:nontrivial-J-map]{\ref{lemma:nontrivial-J-map}$(ii)$}, we proved above that $J(TQ^{2n+1}) \neq 0$.
Likewise, $Z$ is a simply connected K\"ahler manifold with $H^k(Z) = H^k(\bC \bP^5)$ and $c_1(TZ) = 3L$, where $L$ is the positive generator of $H^2(Z;\bZ)$ by~\cite{kotschick-thung}*{Proposition 3 and Lemma 4}.
Letting $f: \bS^2 \to Z$ represent a generator of $H_2(Z;\bZ)$, we find $\la c_1(TZ), f_*[\bS^2] \rg = 3 \equiv 1 \; (\on{mod} 2)$, so $J(TZ) \neq 0$.
\end{proof}

We now recall the construction of the OT-FKM family of isoparametric foliations, with $g = 4$, as recorded in~\cite{wang-topology-clifford}.
Let $P_0, \dots, P_m$ be elements in $O( 2 \ell)$ satisfying $P_i P_j + P_j P_i = 2 \delta_{ij} I$, which generate an orthogonal representation on $\bR^{2 \ell}$ of the Clifford algebra $C_{0,m+1}$ of $(\bR^{m+1},g)$, for $g$ a positive-definite metric.
The $(+1)$-eigenspace of $P_0$ satisfies $E_+(P_0) \simeq \bR^{\ell}$, and is invariant under the elements $E_k = P_1 P_{k+1}$, which belong to $O(\ell)$ and satisfy $E_i E_j + E_j E_i = - 2 \delta_{ij} I$.
Thus, they define an orthogonal $C_{m-1}$-module structure on $E_+(P_0)$, for $C_{m-1}$ the Clifford algebra on $\bR^m$ equipped with a negative-definite metric.
This construction gives a bijective correspondence between orthogonal representations of $C_{0,m+1}$ and $C_{m-1}$, with $\ell = k \delta(m)$ for $k \in \bN^*$, via
\begin{equation}\label{eqn:P0P1}
    P_0 = \begin{pmatrix}
    I & 0 \\ 0 & - I
    \end{pmatrix}, \qquad P_1 = \begin{pmatrix}
        0 & I \\ I & 0
    \end{pmatrix}, \qquad P_{k+1} = \begin{pmatrix}
        0 & E_k \\ - E_k & 0 
    \end{pmatrix}, \qquad E_+(P_0) \cap \bS^{2\ell-1} \cong \bS^{\ell-1}. 
\end{equation}
For $m \equiv 0 \; (\on{mod} \; 4)$, there are two irreducible $C_{m-1}$-modules $\Delta^{\pm}_m$, corresponding to $E_1 E_2 \cdots E_{m-1} = \pm \textup{Id}$ and producing the decomposition $E_+(P_0) = a \Delta^+_m \oplus b \Delta^-_m$ as $C_{m-1}$-modules.
We define the \textit{index} $q$ of the $C_{m-1}$-representation on $\bR^{\ell}$ as
\begin{equation}\label{eqn:index-representation}
    q = a-b, \qquad 2 q\delta(m) = \textup{Tr} (P_0 P_1 \cdots P_m),
\end{equation}
with $q = 0$ if $m\not\equiv 0 \; ( \on{mod} \; 4)$.
The isoparametric foliation $\{ M_s \}$ consists of level sets of the function 
\begin{equation}\label{eqn:F-level-sets}
    F_P: \bS^{2 \ell - 1} \to [-1,1], \qquad F_P(x) := \la x, x \rg^2 - 2 \sum_{i=0}^m \la P_i x, x \rg^2 .
\end{equation}
The functions $F_q$ are equivalent for Clifford systems $\{ P_i \}$ of the same index $q$; for $m \equiv 0 \; (\on{mod} 4)$, and $\ell = k \delta(m)$, there are $(k+1)$ distinct inequivalent values of $q$, hence functions $F_q$ giving rise to regular level sets $M = M (m,\ell,q)$ and focal submanifolds $M_1(m,\ell,q) = F_q^{-1}(1) , M_2(m,\ell,q) = F_q^{-1}(-1)$.

Consider the sphere $B := E_+(P_0) \cap \bS^{2\ell-1} \cong \bS^{\ell-1}$.
For $x \in B$, the curves
\[
x_{\alpha}(\tau) = \cos (\tfrac{\tau}{2}) \, x + \sin ( \tfrac{\tau}{2}) \, P_{\alpha} x, \qquad Q_{\alpha}(\tau) = \cos \tau\, P_0 + \sin \tau \, P_{\alpha} 
\]
satisfy $Q_{\alpha}(\tau) x_{\alpha}(\tau) = x_{\alpha}(\tau)$, so $x_{\alpha}(\tau) \in M_2$ and $P_{\alpha} x \in T_x M_2$ for $\alpha = 1, \dots, m$.
Thus, inside the ambient normal space $E_-(P_0)$, the tangent directions of $M_2$ transverse to the fiber are spanned by $\{ P_{\alpha} \}_{\alpha=1}^m$, which become $\{ x, E_1x, \dots, E_{m-1} x \}$ under the identification $E_-(P_0) \xrightarrow{P_1} E_+(P_0) \cong \bR^{\ell}$.
Let $\eta \to B$ denote the bundle defined by the fiberwise orthogonal complement of these vectors,
\begin{equation}\label{eqn:eta-span-general}
\eta := \textup{span} \{ x, E_1 x, \dots, E_{m-1} x \}^{\perp} \subset \bR^{\ell}.
\end{equation}
The FKM construction shows that $M_1$ is the total space of the sphere bundle $S(\eta) \to B$ and has trivial normal bundle inside $\bS^{n-1}$, so $\mathbf{S}_{M_1} \cong M_1 \times \bS^{m_1+1}$ always~\cite{ferus-karcher-munzner}.
Moreover, the bundle $\eta$ is trivial for the pairs $(m_1, m_2)$ of~\eqref{eqn:pairs-trivial-bundle}, whereby $M_1 \cong \bS^{\ell-1} \times \bS^{\ell-m_1-1}$ in those cases.

Regarding $M_2$, the above identifications via $E_-(P_0) \xrightarrow{P_1} E_+(P_0)$ yield $\nu_{M_2}|_B \cong \eta$.
Moreover, the fields $E_1 x, \dots, E_{m-1} x$ are tangent vector fields spanning a global orthonormal frame; in particular, they span the trivial rank-$(m-1)$ subbundle $\mathbf{1}^{\oplus (m-1)}$.
Since $T_x B = T_x \bS^{\ell-1} = \la x \rg^{\perp}$, the definition~\eqref{eqn:eta-span-general} produces the bundle isomorphism
\begin{equation}\label{eqn:tangent-space}
    T \bS^{\ell-1} \cong \eta \oplus \mathbf{1}^{\oplus (m-1)}.
\end{equation}
We now use these tools to address the possible homotopy equivalences $\mathbf{S}_{M_i} \simeq M_i \times \bS^r$.
It would be interesting to classify the possible hypersurfaces $\mathbf{S}_{M_i}$ arising from OT-FKM foliations only under stable homotopy equivalence $\mathbf{S}_{M_2} \simeq_s M_2 \times \bS^r$, using the techniques employed in Proposition~\ref{prop:all-nontrivial-isoparametric}.

\begin{lemma}\label{lemma:ot-fkm-family-I}
    For the OT-FKM family with $(g,m_1,m_2) = (4,m, \ell - m-1)$, the surface $\mathbf{S}_{M_1} \cong M_1 \times \bS^{m+1}$ is a product bundle.
    If $m \not\equiv 0 \; (\on{mod} 4)$ and the bundle $\mathbf{S}_{M_2} = S(\nu_{M_2} \oplus \mathbf{1})$ is homeomorphic to a product $M_2 \times \bS^r$, then $M_2 \cong S(\xi)$ has
\begin{equation}\label{eqn:xi-trivial-bundle}
    \xi \; \textup{is the trivial bundle}, \quad \Longleftrightarrow \quad
\begin{cases}
m\equiv 3,5,6,7 \pmod 8,\\
m\equiv 1,2 \pmod 8 \; \textup{ and } \; k \textup{ even},\\
m\equiv 0 \pmod 4 \; \textup{ and } \; q=0,
\end{cases}
\end{equation}
for $q$ the index~\eqref{eqn:index-representation}.
For $m \equiv 0 \; (\on{mod} 4)$, this conclusion holds for $\mathbf{S}_{M_2} \simeq M_2 \times \bS^r$ a homeomorphism.
Conversely, if $\xi$ is the trivial bundle, then $\mathbf{S}_{M_2} \simeq_s M_2 \times \bS^{\ell-m}$ are stably homotopy equivalent.
Moreover, $\mathbf S_{M_2} \cong M_2 \times \bS^{\ell-m}$ is a product bundle when $\nu_{M_2} \to M_2$ is trivial, namely
\[
(m_1,m_2)\in\{(1,2),(2,1),(1,6),(6,1),(2,5),(5,2),(3,4) , (4,3)^{\textup{ind}}\}.
\]
\end{lemma}
\begin{proof}
For $M_1$,~\cite{ferus-karcher-munzner} proved that $M_1 \hookrightarrow \bS^{2\ell-1}$ has trivial normal bundle, so $\mathbf{S}_{M_1} = S( \nu_{M_1} \oplus \mathbf{1}) \cong M_1 \times \bS^{m + 1}$ is a product bundle.
Next, $M_2$ is expressible as an $\bS^{\ell-1}$-bundle over $\bS^m$, namely the sphere bundle $M_2\cong S(\xi)\to \bS^m$ of a rank-$\ell$ bundle $\xi$.
By \cite{topology-g-4}*{Theorem 2.6(i)}, we have
\[
M_2 \; \text{ is parallelizable } \Longleftrightarrow \; 
M_2 \; \text{ is stably parallelizable } \Longleftrightarrow \;
\xi\text{ is trivial}.
\]
By~\cite{wang-topology-clifford}*{Corollary 1}, the bundle $\xi$ is trivial precisely in the cases~\eqref{eqn:xi-trivial-bundle}, whereby $M_2 \cong \bS^{\ell-1} \times \bS^m$.
For $m \equiv 1,2 \; (\on{mod} 8)$ and $\mathbf{S}_{M_2} \simeq M_2 \times \bS^r$ a homotopy equivalence, Lemma~\hyperref[lemma:obstructions-to-triviality]{\ref{lemma:obstructions-to-triviality}$(ii)$} implies that $w(TM_2) = 1$, and writing $TM_2 \oplus \mathbf{1} \cong \pi^*(T \bS^m \oplus \xi)$ with $M_2 = S(\xi)$ forces $w(\xi) = 1$, since $w(T \bS^m) = 1$.
For $m \equiv 1,2 \; (\on{mod} 8)$ we have $\widetilde{KO}(\bS^m) \cong \bZ_2$, so there are only two stable classes and the non-zero one is detected by the top Stiefel-Whitney class.
Thus, the nonzero class in $\widetilde{KO}(\bS^m)\cong \bZ_2$ is detected by the top Stiefel-Whitney class, so $w(\xi)=1$ forces $[ \xi ]=0$, so in the OT-FKM case where $\ell \geq m+2$, we have $M_2 = S(\xi)$ stably parallelizable.
The above chain of equivalences implies that $\xi$ is trivial and $k \equiv 0 \; (\on{mod} 2)$.

For $m \equiv 0 \; (\on{mod} 4)$, the homeomorphism $f: \mathbf{S}_{M_2} \xrightarrow{\cong} M_2 \times \bS^r$  implies that $T(M_2 \times \bS^r) \cong f^* T \mathbf{S}_{M_2}$, and arguing as in Lemma~\hyperref[lemma:nontrivial-J-map]{\ref{lemma:nontrivial-J-map}$(i)$} shows that $p(M_2)_{\bQ} = p(M_2 \times \bS^r)_{\bQ} = 0$ because the rational Pontryagin class is invariant under homeomorphism, by Novikov's theorem, and $p(\mathbf{S}_{M_2})_{\bQ} = 0$ because $\mathbf{S}_{M_2}$ is stably parallelizable.
Moreover,~\cite{CAGstablyparallel}*{Lemma 3.1} gives $p_{\frac{m}{4}}(M_2) = q \beta(m)( \frac{m}{2} -1 )! \, \pi_1^* \gamma$ for $\pi^*_1 : H^m(\bS^m;\bZ) \to H^m(M_2;\bZ)$ an isomorphism and $\beta(m) \in \{1,2\}$, so $p(M_2)_{\bQ} = 0$ forces $q=0$.

When the conditions~\eqref{eqn:xi-trivial-bundle} are satisfied and $\xi$ is the trivial bundle, then $M_2 \cong \bS^m \times \bS^{\ell-1}$ and the sphere bundle $\mathbf{S}_{M_2} \xrightarrow{\pi} M_2$ admits the canonical north pole section $s: M_2 \to \mathbf{S}_{M_2}$ coming from the north pole map $x \mapsto (0,1) \in \nu_x \oplus \bR$.
The quotient by this section is the Thom space $S(\nu_{M_2} \oplus \mathbf{1}) / s(M_2) \cong \textup{Th}(\nu_{M_2})$ described in~\cite{hatcher-book-alg-top}*{\S 4.D}.
The space $\textup{Th}(\nu_{M_2})$ fits into the cofiber sequence that splits into a wedge sum after one suspension $\Sigma$, namely 
\[
(M_2)_+ \to S(\nu_{M_2} \oplus \mathbf{1})_+ \to \textup{Th}( \nu_{M_2}), \qquad \Sigma S(\nu_{M_2} \oplus \mathbf{1})_+ \simeq \Sigma (M_2)_+ \vee \Sigma \, \textup{Th}(\nu_{M_2}) \, .
\]
Since $M_2 \cong \bS^m \times \bS^{\ell-1} \subset \bS^{2 \ell-1}$ is parallelizable, we write $TM_2 \oplus \nu_{M_2} \oplus \mathbf{1} \cong T \bS^{2 \ell-1} \oplus \mathbf{1} \cong \mathbf{1}^{\oplus 2 \ell}$ to see that $\nu_{M_2}$ is stably trivial, hence $\textup{Th}(\nu_{M_2}) \simeq_s \Sigma^{\ell-m} (M_2)_+$ are stably equivalent.
Consequently,
\[
\Sigma^{\infty} \mathbf{S}_{M_2} \simeq \Sigma^{\infty} (M_2)_+ \vee \Sigma^{\infty} \Sigma^{\ell-m} (M_2)_+ \simeq \Sigma^{\infty}(M_2 \times \bS^{\ell-m})_+ 
\]
where we used $(M_2 \times \bS^r)_+ \cong (M_2)_+ \wedge \bS^r_+$ and $\bS^r_+ \simeq \bS^0 \vee \bS^r$.
This proves that $\mathbf{S}_{M_2} \simeq_s M_2 \times \bS^{\ell-m}$.

Finally, the focal submanifolds of trivial normal bundle $\nu_{M_2} \to M_2$ adhere to the classification given above, in which case $S( \nu_{M_2} \oplus \mathbf{1}) \cong M_2 \times \bS^{\ell-m}$ splits as a product. 
\end{proof}

\begin{remark}\label{rmk:m=2(p-1)-index}
    If $m = 2(p-1)$ and $q \not\equiv 0 \; (\on{mod} p)$, where $p$ is an odd prime, then $\mathbf{S}_{M_2} \not\simeq M_2 \times \bS^r$ cannot be homotopy equivalent to a product.
The computation~\cite{CAGstablyparallel}*{(3.17)} shows that in this case, the Wu operation $\cP^1$ on mod-$p$ cohomology induces a non-zero map $\cP^1: H^{\ell-m}(\mathbf{S}_{M_2}; \bZ_p) \to H^{\ell}(\mathbf{S}_{M_2} ; \bZ_p)$.
On the other hand, $\mathbf{S}_{M_2} \simeq M_2 \times \bS^r$ requires $r = \ell-m$ by the cohomology ring computation of Lemma~\ref{lemma:cohomology-ring}, hence the degree $(\ell-m)$-class would come from the $\bS^{\ell-m}$-factor.
By naturality, the map $\cP^1$ must vanish because $H^{\ell}(\bS^{\ell-m};\bZ_p) = 0$, producing a contradiction.
\end{remark}
Finally, we study the sphere bundles $\mathbf{S}_{M_2}$ over focal manifolds with $M_2 = S(\xi) \cong \bS^m \times \bS^{\ell-1}$, in which case $M_2$ is stably parallelizable and $\mathbf{S}_{M_2} \simeq_s M_2 \times \bS^r$ is stably homotopy equivalent to a product.
To classify the homotopy type of $\mathbf{S}_{M_2}$, we first introduce an auxiliary notion.
\begin{definition}\label{def:fiber-homotopy-trivial}
Given two bundles $N \xrightarrow{\pi} B$ and $N' \xrightarrow{\pi'} B$, a map $f: N \to N'$ is called \textit{fiber-preserving} if it fibers over $\textup{id}_B$, meaning that $\pi' \circ f = \pi$.
The bundles $N \xrightarrow{\pi} B$ and $N' \xrightarrow{\pi'} B$ are called \textit{fiber homotopy equivalent} if there exist fiber-preserving maps $f: N \to N'$ and $f': N' \to N$ with
\[
f' \circ f \simeq_B \textup{id}_N, \qquad f \circ f' \simeq_B \textup{id}_{N'} \, .
\]
Here, $\simeq_B$ denotes homotopy equivalence through fiber-preserving maps.
A bundle $N \xrightarrow{\pi} B$ with fiber $F$ is \textit{fiber homotopically trivial} if it is fiber homotopy equivalent to the product bundle $B \times F \to B$.
\end{definition}

We refer the reader to the works of Dold, Dold-Lashof, and Milnor-Spanier~\cites{dold , dold-lashof , milnor-spanier } for properties of fiber homotopy equivalences that we will utilize below.
\begin{lemma}\label{lemma:fiberwise-homotopy-equivalent}
    Let $B$ be a connected manifold and $Y \xrightarrow{\pi} B$ be an $\bS^r$-sphere bundle that is homotopy equivalent to $B \times \bS^r$, with $r \geq 1$.
    Suppose that $H^r(B;\bZ) = 0$ and either 
    \begin{enumerate}[(i)]
        \item $H^{r+1}(B;\bZ)$ is torsion-free; or
        \item $H^{r+1}(B;\bZ)$ is a finitely generated abelian group and $H^1(B;\bZ) = 0$.
    \end{enumerate}
    Then, the bundle $Y \xrightarrow{ \pi}B$ is fiber homotopically trivial, i.e., fiber homotopy equivalent to $B \times \bS^r$.
\end{lemma}
\begin{proof}
    Given a homotopy equivalence $h: Y \to B \times \bS^r$, we define $f := \textup{pr}_{\bS^r} \circ h : Y \to \bS^r$ and will prove that the map $F: (\pi, f): Y \to B \times \bS^r$ over $B$ is a fiber homotopy equivalence.

    First, we observe that the bundle $Y \xrightarrow{\pi} B$ is orientable.
    Indeed, the Serre cohomology spectral sequence of the bundle with $\bZ$ coefficients has $q=r$ row given by $E^{p,r}_2 = H^p(B;\cL)$, where $\cL$ denotes the local system $H^r( \bS^r;\bZ)$.
    Since $Y \simeq B \times \bS^r$ and $H^r(B) = 0$, the K\"unneth formula gives $H^r(Y;\bZ) \cong H^r(B \times \bS^r;\bZ) \cong \bZ$, see for example~\cite{hatcher-book-alg-top}*{Ch. 5, \S 5.1}.
    If $\cL$ was non-trivial, then $H^0(B;\cL) = 0$, because $B$ is connected, and $E^{r,0}_2 = H^r(B) = 0$.
    Thus, there would be no cells contributing to total degree $r$ in the $E_2$-page, hence $H^r(Y) = 0$ would lead to a contradiction.
    Therefore, $\cL$ is trivial, and the sphere bundle is orientable.
    Consequently, the cohomology ring computation of Lemma~\ref{lemma:cohomology-ring} holds with $\bZ$-coefficients, producing a well-defined Euler class $e \in H^{r+1}(B;\bZ)$ that satisfies $d_{r+1}(1) = e$ in the Serre spectral sequence.

    Next, we show that the Euler class $e$ vanishes if condition $(i)$ or $(ii)$ holds.
    The assumption $H^r(B) = 0$ implies that the cells of total degree $r$ come only from $E^{0,r}_{\infty}$, so
    \[
    \bZ \cong H^r(Y;\bZ) \cong E^{0,r}_{\infty} = \ker \bigl( d_{r+1} : \bZ \to H^{r+1}(B; \bZ) \bigr).
    \]
    If $d_{r+1}(1) = e \neq 0$ and $H^{r+1}(B;\bZ)$ is torsion-free, then $d_{r+1}$ is injective and the kernel is trivial, which is impossible; thus, $e=0$ in case $(i)$.
    In case $(ii)$, the sphere bundle $Y$ is orientable, so we can use the Gysin sequence as in Lemma~\ref{lemma:cohomology-ring}, with integer coefficients.
    This produces
    \[
    0 \to H^r(Y) \xrightarrow{\pi_!} H^0(B;\bZ) \cong \bZ \xrightarrow{\cup \, e} H^{r+1}(B;\bZ) \xrightarrow{\pi^*} H^{r+1}(Y;\bZ) \to H^1(B;\bZ).
    \]
    Since $H^1(B)= 0$, the above sequence implies that $H^{r+1}(Y;\bZ) \cong A/ \la e \rg$, where $A := H^{r+1}(B;\bZ)$.
    On the other hand, using $Y \simeq B \times \bS^r$ and $H^1(B) =0$, the K\"unneth formula yields $H^{r+1}(Y;\bZ) \cong H^{r+1}(B;\bZ) = A$.
    Combining these two properties, we find $A \cong A/\la e \rg$, which is impossible for a finitely generated abelian group $A$ unless $e=0$.
    Thus, $e =0$ in both cases, so the differential $d_{r+1}$ from $E_{r+1}^{0,r} \cong \bZ$ vanishes because $Y$ is orientable and $e=0$.
    Moreover, $H^r(B) = 0$, whereby the only nonzero $E^{p,q}_{\infty}$ with $p+q=r$ is $E^{0,r}_{\infty} \cong \bZ$, and the total degree-$r$ filtration has
    \[
    F^1H^r(Y) = 0, \qquad F^0H^r(Y) = H^r(Y), \qquad H^r(Y) \cong E^{0,r}_{\infty} \cong \bZ.
    \]
    We will now prove that on each fiber $Y_b = \pi^{-1}(b) \cong \bS^r$, the restriction map $ f|_{Y_b}: Y_b \to \bS^r$ has degree $\pm 1$.
    Let $\iota_r \in H^r(\bS^r;\bZ)$ denote the fundamental class and $i_b: Y_b \hookrightarrow Y$ be the inclusion map, inducing the edge homomorphism $i^*_b : H^r(Y;\bZ) \to H^r(Y_b;\bZ) \cong \bZ$ in the Serre spectral sequence.
    The above description of the  degree-$r$ filtration shows that there is no contribution in total degree $r$ from $E^{p,0}_2$, so the map $i^*_b$ is an isomorphism for every $Y_b$.
    Moreover, the homotopy equivalence $h$ induces an isomorphism $h^*$ on cohomology, hence $\alpha := h^* (\textup{pr}_{\bS^r}^* \iota_r) = f^* \iota_r$ is a generator of $H^r(Y;\bZ)$.
    Restricting to any fiber $Y_b$, we find
    \[
    i^*_b \alpha = i^*_b(f^* \iota_r) = ( f|_{Y_b})^* (\iota_r).
    \]
    Since $i^*_b$ is an isomorphism, $i^*_b(\alpha)$ generates $H^r(Y_b) \cong \bZ$, hence $(f|_{Y_b})^* (\iota_r) = \pm \iota_r$.
    This implies that $\deg (f|_{Y_b}) = \pm 1$, so each fiber map $f|_{Y_b}: Y_b \to \bS^r$ is a sphere map of degree $\pm 1$, hence a homotopy equivalence.
    Finally, $F = (\pi,f): Y \to B \times \bS^r$ is a map over $B$ inducing a homotopy equivalence $f|_{Y_b}$ on every fiber, so it is admissible over $B$ in the sense of Dold-Lashof~\cite{dold-lashof}.
    Thus, Dold's theorem~\cite{dold}*{Theorem 6.3} implies that $F$ is a fiber homotopy equivalence as claimed.
\end{proof}

\begin{proposition}\label{prop:trivial-product-bundles-ot-fkm}
    For the OT-FKM family with $M_2 \cong \bS^m \times \bS^{\ell-1}$, the sphere bundle $\mathbf{S}_{M_2}$ is homotopy equivalent to a product $M_2 \times \bS^r$ if and only if 
    \begin{equation}\label{eqn:(m1,m2)-Sm2-product-list}
(m_1,m_2)\in\{ (1, 2d) , (2d,1) , (2,5),(5,2),(3,4) , (4,3) \}, \qquad d \in \bN^* ,
    \end{equation}
    in which case $M_2 \cong \bS^{m_1} \times \bS^{m_1+m_2}$ and $\mathbf{S}_{M_2} \cong \bS^{m_1} \times \bS^{m_1+m_2} \times \bS^{m_2+1}$ as smooth sphere bundles.
\end{proposition}
\begin{proof}
    We prove the result in three steps: first, we obtain a general restriction on the possible values of $(m,\ell)$ for which a homotopy equivalence $\mathbf{S}_{M_2} \simeq M_2 \times \bS^r$ is possible; next, we examine this finite number of remaining pairs $(m,\ell)$, producing the list of~\eqref{eqn:(m1,m2)-Sm2-product-list} together with the exceptional pair $(m_1, m_2) = (8,7)$.
    Finally, we rule out the pair $(8,7)$ by an adaptation of Lemma~\ref{lemma:fiberwise-homotopy-equivalent}.
    Note that $\mathbf{S}_{M_2} \simeq M_2 \times \bS^r$ forces $r = \ell-m$ by the cohomology ring computation of Lemma~\ref{lemma:cohomology-ring}.

\smallskip \noindent \textbf{Step 1:} First, we prove a homotopy equivalence $\mathbf{S}_{M_2} \simeq M_2 \times \bS^r$ is impossible for $m \geq 2$ and $\ell \neq 2m$, unless $\ell \in \{1,2,4,8\}$.
For closed connected manifolds, since $\mathbf{S}_{M_2} \subset \bS^n$ is a hypersurface, the homotopy equivalence would force $M_2$ to be orientable, hence the cohomology ring and Serre spectral sequence computations of Lemmas~\ref{lemma:cohomology-ring} and~\ref{lemma:fiberwise-homotopy-equivalent} are valid in $\bZ$-coefficients without local systems.
Expressing $M_2 \cong S(\xi) \to \bS^m$ with fiber $\bS^{\ell-1}$, the Serre spectral sequence has the $E_2$-terms
\[
E_2^{p,q} = H^p(\bS^m ; H^q(\bS^{\ell-1};\bZ)), \qquad
E_2^{0,0} = \bZ , \qquad E^{m,0}_2 = \bZ, \qquad E_2^{0,\ell-1} = \bZ, \qquad E^{m ,\ell-1}_2 = \bZ.
\]
The only possible differential is $d_{\ell} : E^{0,\ell-1}_{\ell} \to E^{\ell,0}_{\ell} = H^{\ell}(\bS^m;\bZ) = 0$ for $\ell \geq m+1$, so
\[
H^j(M_2;\bZ) \cong \bZ \quad \text{for } \; j \in \{0,m,\ell-1,m+\ell-1\}, \qquad H^j(M_2;\bZ) = 0 \quad \text{otherwise}.
\]
Notably, $M_2$ has the additive cohomology of $\bS^m \times \bS^{\ell-1}$.
Thus, any $m \geq 2$ and $\ell \neq 2m$ make $H^1(M_2) = H^{\ell-m}(M_2) = 0$, so Lemma~\ref{lemma:fiberwise-homotopy-equivalent} shows that $\mathbf{S}_{M_2} \simeq M_2 \times \bS^{\ell-m}$ would force $S(\nu_{M_2} \oplus \mathbf{1})$ to be fiber homotopy trivial as a sphere bundle over $M_2$.

For sphere bundles over a base $B$, let $\Sigma_B$ denote the fiberwise suspension operator.
If $\xi \to B$ is a vector bundle, then $\Sigma_B$ satisfies the canonical fiberwise homeomorphism $\Sigma_B S(\xi) \cong S(\xi \oplus \mathbf{1})$ via $\Sigma \bS^{m-1} \cong \bS^m$.
The two bundles are identified by the map $[v,t] \mapsto ( \sqrt{1-t^2} \, v , t)$ with inverse sending $S(\xi_b \oplus \bR) \ni (v,s)$
to $\bigl[ \frac{v}{\sqrt{1-s^2}} , s \bigr]$ when $|s|<1$, and to the north or south pole when $s = \pm 1$.
This construction is continuous in $b$ and fibers over the base $B$, so it defines a bundle homeomorphism over $B$.
In our situation, we can iterate this operation to express $S( \nu_{M_2} \oplus \mathbf{1}^{\oplus (m-1)}) \cong \Sigma_B^{m-2} S( \nu_{M_2} \oplus \mathbf{1})$ for $m \geq 2$.
Moreover, applying fiberwise suspension to the fiber homotopy equivalences of Definition~\ref{def:fiber-homotopy-trivial} shows that $S(\nu_{M_2} \oplus \mathbf{1}^{\oplus (m-1)})$ is fiber homotopy trivial whenever $S(\nu_{M_2} \oplus \mathbf{1})$ is.

On the other hand, our earlier discussion on OT-FKM isoparametric hypersurfaces shows that $\nu_{M_2}|_{\bS^{\ell-1}} \cong \eta$ for an embedding $i: \bS^{\ell-1} \hookrightarrow M_2$, where $\eta \oplus \mathbf{1}^{\oplus (m-1)} \cong T \bS^{\ell-1}$ by~\eqref{eqn:tangent-space}.
Therefore,
\[
\eta = i^* \nu_{M_2} \implies S( T \bS^{\ell-1}) \cong S( \eta \oplus \mathbf{1}^{\oplus (m-1)} ) \cong i^* S( \nu_{M_2} \oplus \mathbf{1}^{(m-1)})
\]
is fiber homotopically trivial as a bundle over $\bS^{\ell-1}$, by functoriality.
We conclude that the sphere bundle $S(T \bS^{\ell-1}) \to \bS^{\ell-1}$ is fiber homotopically trivial; by the version of Adams' theorem due to Milnor-Spanier~\cites{adams-theorem , milnor-spanier }, this property requires $\ell-1 = 0$ or $\ell-1 \in \{ 1,3,7\}$.

\smallskip \noindent \textbf{Step 2:}
We now examine the remaining pairs $(m,\ell)$.
For $m=1$, the OT-FKM construction discussed above only uses the matrices $P_0, P_1$ and the points of $M_2$ are expressible as $x = (\cos t \, z, \sin t \, z)$ for $z \in \bS^{\ell-1}$, with the identification $(t,z) \sim (t+\pi , -z)$.
The triviality of $\xi$ from~\eqref{eqn:xi-trivial-bundle} requires $\ell = k \delta(m)=k$ with $k=2d$ even, so $M_2$ is the mapping torus of the antipodal map $A(z) = - z$ on $\bS^{\ell-1}$ with $-I \in \textup{SO}(\ell)$ by parity.
In particular, $A$ is isotopic with the identity and the mapping torus $M_2 \cong \bS^1 \times \bS^{\ell-1}$ is trivial, with tangent and normal spaces given by
\[
T_z M_2 = \textup{span} \{ (-\sin t \, z , \cos t \, z), (\cos t \, v, \sin t \, v) : v \perp z\}, \qquad N_z M_2 = \{ (-\sin t \, v , \cos t \, v) : v \perp z \}.
\]
Consequently, $\nu_{M_2} \cong \textup{pr}_2^* (T \bS^{\ell-1})$ and $\nu_{M_2} \oplus \mathbf{1} \cong \textup{pr}_2^* ( T \bS^{\ell-1} \oplus \mathbf{1}) \cong \mathbf{1}^{\oplus \ell}$ is the trivial rank-$\ell$ bundle, hence $S(\nu_{M_2} \oplus \mathbf{1}) \cong \bS^1 \times \bS^{\ell-1} \times \bS^{\ell-1}$ is the smoothly trivial product bundle.

For $m \geq 2$, we have $\ell \geq m+2 \geq 4$, so either $\ell = 2m$ or $\ell \in \{ 4,8\}$.
Using $\delta(8r+j) = 2^{4r} \delta(j)$ from~\eqref{eqn:delta(m)-values}, we find $\ell \geq \delta(m) > 2m > 8$ for $m \geq 10$, so we are reduced to the finitely many cases
\[
(m , \ell) \in \{ (2,4), (2,8), (3,8), (4 , 8) , (5,8), (6,8) , (8,16) \}.
\]
Expressing $(m_1, m_2) = (m,\ell-m-1)$ precisely recovers the list~\eqref{eqn:(m1,m2)-Sm2-product-list} together with the exceptional pair $(m\ell) = (8,16)$ subject to $q=0$, in which case $(m_1, m_2) = (8,7)$ and $M_2 \cong \bS^8 \times \bS^{15}$.

\smallskip \noindent \textbf{Step 3:}
To rule out the final possibility $(m,\ell) = (8,16)$, we adapt Lemma~\ref{lemma:fiberwise-homotopy-equivalent} to show that $\mathbf{S}_{M_2} \simeq M_2 \times \bS^r$ would again force $Y := S(\nu_{M_2} \oplus \mathbf{1})$ to be fiber homotopically trivial.
Since $Y,M_2, \nu$ are orientable, Lemma~\ref{lemma:cohomology-ring} holds with integer coefficients, producing the decomposition
\[
H^{\bullet}(Y;\bZ) \cong H^{\bullet}(M_2;\bZ) \oplus u \cdot H^{\bullet-8}(M_2;\bZ), \qquad u \in H^r(Y;\bZ) \, .
\]
Here, $u$ satisfies $\pi_!(u) =1$ under the Gysin map and $i^*_b(u) = \pm \iota_8$ under the edge homomorphism on each fiber $Y_b \cong \bS^8$.
As in Lemma~\ref{lemma:cohomology-ring}, let $s: M_2 \to Y$ be the canonical north pole section $b \mapsto (0,1) \in \nu_x \oplus \bR$ coming from the $\mathbf{1}$-summand.
Following the notation of Lemma~\ref{lemma:fiberwise-homotopy-equivalent}, the bundle $Y = S(\nu_{M_2} \oplus \mathbf{1})$ has a section, so the Euler class of the rank-$9$ bundle $\nu_{M_2} \oplus \mathbf{1}$ vanishes.
Replacing $u$ by $u - \pi^* s^* u$, we may also assume that $s^*u=0$.
We consider the generators
\[
x \in H^8 (M_2;\bZ), \qquad y \in H^{15}(M_2;\bZ), \qquad M_2 \cong \bS^8 \times \bS^{15}
\]
of the cohomology ring $H^{\bullet}(M_2;\bZ)$.
Then, $x^2=y^2 = 0$ and $H^{16}(M_2) = 0$, so $u^2 = c \, \pi^*(x) \, u$ for some integer $c$.
Since $M_2 \subset \bS^{31}$ is parallelizable, we write $TM_2 \oplus \nu_{M_2} \oplus \mathbf{1} \cong T \bS^{31} \oplus \mathbf{1} \cong \mathbf{1}^{\oplus 32}$ to see that $\nu_{M_2}$ is stably trivial, so all its Stiefel-Whitney classes vanish.
In particular, $w_8(\nu) = e(\nu) \; (\on{mod} 2)$ implies that $e(\nu) \in H^8(M_2;\bZ) \cong \bZ x$ is even, thus $c = 2d$ for $d \in \bZ$.
Therefore,
\begin{equation}\label{eqn:v=u-dpi}
v := u - d \, \pi^* x \implies v^2 = 0, \quad \pi_!(v) = 1, \quad i^*_b(v) = \pm \iota_8, \quad \implies \quad H^8(Y;\bZ) = \bZ \pi^* x \oplus \bZ v. 
\end{equation}
The degree-$8$ cohomology of $Y$ is thus $H^8(Y;\bZ) = \bZ \pi^* x \oplus \bZ v$, with both generators squaring to zero.
As in Lemma~\ref{lemma:fiberwise-homotopy-equivalent}, let $h: Y \xrightarrow{ \simeq} M_2 \times \bS^8$ be a homotopy equivalence and define
$f := \textup{pr}_{\bS^8} \circ h : Y \to \bS^8$ and $F := (\pi, f): Y \to M_2 \times \bS^8$.
Let $\bar{w}_1 \in H^8( M_2 \times \bS^8; \bZ)$ and $\bar{w}_2 \in H^8(M_2 \times \bS^8 ; \bZ)$ be the generators coming from the $\bS^8$-factor of $M_2$ and from the additional $\bS^8$-factor, respectively, so $\bar{w}_i^2 = 0$ and the pullbacks $h^* \bar{w}_i$ are primitive square-zero classes in $H^8(Y)$, since $h^*$ is a ring isomorphism.
Using the basis $\{ \pi^* x, v \}$ of $H^8(Y)$ from~\eqref{eqn:v=u-dpi}, any class $w$ is expressible as $w = a_1 \, \pi^* x + a_2 \, v$, so
\[
w^2 = (a_1 \, \pi^* x + a_2 \, v)^2 = a_1^2 \pi^*(x^2) + a_2^2 v^2 + 2 a_1 a_2 \, (\pi^* x)v = 2 a_1 a_2 \, (\pi^*x ) v
\]
due to $x^2=v^2=0$.
Since $H^{16}(Y) \cong \bZ \cdot (\pi^* x)v$ is torsion-free, we conclude that $w^2 = 0$ if and only if $a_1 a_2=0$.
Thus, the only primitive zero-square classes are $\{ \pm \pi^* x, \pm v \}$, so the $\{ h^* \bar{w}_i \}_{i=1,2}$ must coincide with those classes in some order.
Thus, up to swapping the two $\bS^8$-factors of $M_2 \times \bS^8$ and the signs, we may write $h^* \bar{w}_1 = \pi^* x$ and $h^* \bar{w}_2 = v$.
Consequently,
\[
f^* (\iota_8) =v, \qquad (f|_{Y_b})^* (\iota_8) = i_b^* (v) = \pm \iota_8, \qquad \text{for every fiber } \; Y_b \cong \bS^8,
\]
so each fiber map $f|_{Y_b}: Y_b \to \bS^8$ has degree $\pm 1$, thus it is a homotopy equivalence.
Continuing as in Lemma~\ref{lemma:fiberwise-homotopy-equivalent}, we deduce that $F := (\pi, f) :Y \to M_2 \times \bS^8$ is a fiberwise homotopy equivalence, and arguing as in Step 1 shows that $S( T \bS^{15}) \to  \bS^{15}$ must be fiber homotopically trivial, contradicting Milnor-Spanier~\cite{milnor-spanier}.
Thus, $\mathbf{S}_{M_2} \not\simeq M_2 \times \bS^r$ and our classification is complete.
\end{proof}
In the discussion following Theorem~\ref{thm:new-minimal-surfaces}, we noted that the two Type I surfaces $\mathbf{S}_{M_1}$ and $\mathbf{S}_{M_2}$ constructed over the focal submanifolds $M_1, M_2$ are isometric when $M_1 \cong M_2$; this applies to the Cartan-M\"unzner isoparametric hypersurfaces for $g=3$ and the symmetric Clifford tori $\textup{O}(k) \times \textup{O}(k)$ for $g=2$.
In all other cases, we prove that the two resulting surfaces $\mathbf{S}_{M_1}$ and $\mathbf{S}_{M_2}$ have distinct topological types by studying their cohomology rings.
\begin{proposition}\label{prop:surfaces-are-distinct}
    For all known isoparametric foliations with $g \in \{ 4,6 \}$ except for the one coming from the isotropy representation of $G_2 / \textup{SO}(4)$, the minimal surfaces $\mathbf{S}_{M_1}$ and $\mathbf{S}_{M_2}$ are not homotopy equivalent.
    In the latter case, the surfaces $\mathbf{S}_{M_1}$ and $\mathbf{S}_{M_2}$ are diffeomorphic but not isometric.
\end{proposition}
\begin{proof}
The surfaces $\mathbf{S}_{M_i}$ are sphere bundles $S(\nu \oplus \mathbf{1})$ over $M_i$, so their cohomology ring follows the description of Lemma~\ref{lemma:cohomology-ring}.
Recall from Section~\ref{subsec:isoparametric} that the first non-zero reduced cohomology $\widetilde{H}^{\bullet}$ for the focal submanifolds occurs in degree $m_2$ for $M_1$ and degree $m_1$ for $M_2$.
Using Lemma~\ref{lemma:cohomology-ring}, we find $\widetilde{H}^{\bullet}(\mathbf{S}_{M_i}) \cong \widetilde{H}^{\bullet}(M_i) \oplus u_i {H}^{\bullet-m_i - 1}(M_i)$, so
\[
\min \{ q > 0 : \widetilde{H}^q(\mathbf{S}_{M_1}) \neq 0 \} = \min \{ m_2, m_1 + 1 \}, \quad \min \{ q > 0 : \widetilde{H}^q(\mathbf{S}_{M_2}) \neq 0 \} = \min \{ m_1, m_2 + 1 \}.
\]
If $m_1 \neq m_2$, the two expressions produce $\min \{ m_1, m_2 \}$ and $\min \{ m_1, m_2 \} + 1$, respectively, so the surfaces $\mathbf{S}_{M_i}$ have different cohomology rings, thus $\mathbf{S}_{M_1} \not\simeq \mathbf{S}_{M_2}$.
Comparing with the admissible triples $(g, m_1, m_2)$ of Section~\ref{subsec:isoparametric}, this proves the claim in all cases except $m := m_1 = m_2 = 2$.

If $(g,m) = (4,2)$, then $\{ M_s \}$ is the homogeneous isoparametric foliation with $M_1 \cong \bC \bP^3$ and $M_2 \cong \widetilde{\bG}(2,5)$.
As in Lemma~\ref{lemma:cohomology-complex-quadric}, we have the cohomology ring $H^{\bullet}( \bC \bP^3 ; \bZ) \cong \bZ[h]/(h^4)$ for a primitive generator $h \in H^2$.
On the other hand, $H^{\bullet}( \widetilde{\bG}(2,5); \bZ) \cong H^{\bullet}(Q^3; \bZ) \cong \bZ[x,y]/ (x^2 - 2y,y^2)$ where $|x|=2$ and $|y|=4$, by a standard computation~\cite{borel-hirzebruch}.
Passing to the sphere bundles $\pi_i : \mathbf{S}_{M_i} \to M_i$, the pullback $\pi^*_i$ is an isomorphism in degrees $2$ and $6$ by the cohomology ring computation in Lemma~\ref{lemma:cohomology-ring}, so $\mathbf{S}_{M_1}$ has a primitive class in $H^2$ whose cube is primitive in $H^6$, coming from the generator $h^3 \in H^6(\bC \bP^3)$.
On the other hand, we have $x^3 = x(x^2) = 2xy \in H^6(\widetilde{\bG}(2,5))$, so $\mathbf{S}_{M_2}$ has a primitive class in $H^2$ whose cube is divisible by $2$ in $H^6$.
Combining these considerations, we deduce that the two sphere bundles $\mathbf{S}_{M_1} \not\simeq \mathbf{S}_{M_2}$ are not homotopy equivalent.

If $(g,m) = (6,2)$, the only known isoparametric foliation $\{ M_s \}$ has $M_1 \cong Q^5$ and $M_2$ is diffeomorphic to the twistor space of $G_2 / \textup{SO}(4)$, as discussed earlier.
As above, the cohomology ring description in Lemma~\ref{lemma:cohomology-ring} again shows that the maps $\pi^*_i$ induce isomorphisms $H^k (M_i ) \cong H^k( \mathbf{S}_{M_i})$ in degrees $k \in \{2,10\}$.
We again find $H^{\bullet}(Q^5;\bZ) \cong \bZ[x,y] / (x^3 - 2y  , y^2)$ with $|x| = 2$ and $|y| = 6$, by~\cite{borel-hirzebruch}, so $x^2y$ generates $H^{10}(Q^5;\bZ)$ and $x^5 = x^2 x^3 = 2x^2 y$ exhibits a fifth power of a degree-$2$ class that is divisible by $2$, and not by $3$.
On the other hand, the topology of the space $Z := G_2/ \textup{U}(2)^+$ is studied in~\cite{kotschick-thung}*{Proposition 3}.
It is known that $H^k(Z) = H^k(\bC \bP^5)$ and for $L \in H^2(Z)$ a primitive class, the elements
\[
\tfrac{1}{3} L^2 \in H^4(Z), \quad \tfrac{1}{6} L^3 \in H^6(Z), \quad \tfrac{1}{18} L^4 \in H^8(Z), \quad \tfrac{1}{18} L^5 \in H^{10}(Z)
\]
are integral generators of the higher-degree cohomology groups.
Notably, $L^5 = 18 \eta_Z$ for $\eta_Z \in H^{10}(Z)$ exhibits a fifth power of a degree-$2$ class divisible by $3$, hence the $H^{\bullet}(\mathbf{S}_{M_i})$ are not isomorphic and $\mathbf{S}_{M_1} \not\simeq \mathbf{S}_{M_2}$ are not homotopy equivalent.

For $(g,m) = (6,1)$, the unique isoparametric foliation comes from the isotropy representation of $G_2 / \textup{SO}(4)$.
    Recall the discussion of Section~\ref{subsec:isoparametric}, whereby the two focal manifolds of the foliation are diffeomorphic, but not isometric~\cites{ cecil , miyaoka }.
In this case, the interchanging of $m_1$ and $m_2$ leaves equation~\eqref{eqn:ode-star} unchanged, therefore the profile curves satisfy $f_{M_1,\theta} = f_{M_2,\theta}$.
However, since $M_1$ and $M_2$ are not isometric, this means that the constructed $\Sigma_{M_1,\theta}$ and $\Sigma_{M_2,\theta}$ are geometrically distinct, so in particular $\mathbf{S}_{M_1}$ and $\mathbf{S}_{M_2}$ are diffeomorphic, but not isometric, minimal surfaces inside $\bS^8$.
\end{proof}

\begin{proof}[Proof of Theorem~\ref{thm:non-triviality-theorem}]
    The first part of the theorem, on the (non-)product structure of the sphere bundles $\mathbf{S}_{M_i}$, follows using Proposition~\ref{prop:all-nontrivial-isoparametric} to find $\mathbf{S}_{M_i} \not\simeq_s M_i \times \bS^r$.
    For the OT-FKM family, this characterization comes from Lemma~\ref{lemma:ot-fkm-family-I} and Proposition~\ref{prop:trivial-product-bundles-ot-fkm}.
    The description of the surfaces $\mathbf{S}_{\bF} := \mathbf{S}_{M_i}$ for $g = 3$ is given in Proposition~\ref{prop:g=3Topology}.
    Finally, Proposition~\ref{prop:surfaces-are-distinct} shows that the surfaces $\mathbf{S}_{M_i}$ have distinct homotopy types except for the foliation coming from the isotropy representation $G_2 / \textup{SO}(4)$; in that case, they are diffeomorphic but not isometric.
\end{proof}

\subsection{Type II}
For cones of Type II, the positive phase $(t_1, t_2) \subset (0,1)$ corresponds to graphs over a conical annulus in the base, having the form
\[
\Gamma := \bigl\{ (\rho, \omega) : t_1 < \tfrac{\sin gs(\omega)}{2} < t_2 \bigr\} = \bigl\{ (\rho, \omega) : \tfrac{2}{g} \arcsin t_1 < s(\omega) < \tfrac{2}{g} \arcsin t_2 \bigr\}.
\]
The free boundary has two components in $\Pi = \{ z \geq 0 \}$, namely the cones $\{ s(\omega) = \frac{2}{g} \arcsin t_i \}_{i=1,2}$.
Topologically, the link of such a cone is diffeomorphic to $\bar{\Sigma}_{M,\theta} \cong M \times I$, where $M \subset \bS^{n-1}$ is a regular leaf on the base.
Consequently, the resulting free-boundary minimal cones $\bar{\mathbf{C}}_{M,\theta}$ can be doubled to a minimal hypercone in $\bR^{n+1}$ whose link $\bar{\mathbf{S}}_{M}$ is topologically $M \times \bS^1$; the $\bS^1$-factor comes from closing the interval $[t_1, t_2]$ into a $\bZ_2$-invariant loop by doubling across the equator.

Following the discussion of Section~\ref{section:type-I-topology}, the topology of $\bar{\mathbf{S}}_{M}$ can also be understood as a sphere bundle $\bar{\mathbf{S}}_M \cong S(\nu_M \oplus \mathbf{1})$, this time over a regular leaf $M$.
Since $M \subset \bS^{n-1}$ is a closed embedded oriented hypersurface, the normal bundle $\nu_M$ is trivial and $S(\nu_M \oplus \mathbf{1}) \cong M \times \bS^1$ is smoothly isomorphic to the product $\bS^1$-bundle.

\section{Construction of capillary minimal surfaces}

We now prove the existence of the hypersurfaces $\Sigma_{M_i, \theta}, \bar{\Sigma}_{M, \theta}$ and capillary minimal cones $\mathbf{C}_{M,\theta} , \bar{\mathbf{C}}_{M,\theta}$ of both types I and II, as described in Theorem~\ref{thm:capillary-interpolation}.
We employ a shooting-continuity argument to construct solutions of the capillary equation~\eqref{eqn:ode-star}.
It will be essential to establish a bound for our shooting method, namely  that shots with large initial height cannot reach zero; this will be proved in Proposition~\ref{prop:dont-shoot-too-high}.
First, we analyze equation~\eqref{eqn:ode-star} and its rescaled counterpart~\eqref{eqn:rescaled-equation-lambda-ODE}.

\subsection{ODE analysis}\label{section:analysis-of-capillary-equation}

Following the structure in~\cite{new-minimal-surfaces}, we examine equation~\eqref{eqn:rescaled-equation-lambda-ODE} by defining the quantity $h_f(t) := f(t) - \frac{g}{2}A_M f'(t)$, which satisfies the relation
\begin{equation}\label{eqn:ODE-for-h}
    (1-t^2) h_f' = A S_\lambda h_f - T f' , \qquad A(t) := A_M(t) = t - \tfrac{m_1}{m_1+m_2} t^{-1}.
\end{equation}
We denote $\alpha := \frac{g m_1}{2(n-2)} =  \frac{m_1}{m_1 + m_2}$ (so $A_M(\sqrt{\alpha}) = 0$) and $q_{\lambda} := (1-t^2) \frac{\lambda (f')^2}{1 + \lambda f^2}$, with which we define
\[
S_{\lambda}(t) := \tfrac{2(n-1)}{g} + \tfrac{g(n-2)}{2} q_\lambda(t) , 
\qquad T(t) := \alpha\bigl(\tfrac{g}{2}-\alpha\bigr) t^{-2} + \tfrac{1}{2} (g-2)(1-2\alpha).
\]
Define $\psi(t) = |t^2 - \alpha|^\frac{1}{g}$, so $\frac{\psi'}{\psi} = \frac{2}{g}A^{-1}_M$.
Moreover, 
\begin{equation}\label{eqn:2.6NMS}
h_f = f - \frac{\psi}{\psi'} f', \qquad 
\left(\frac{f}{\psi}\right)' = \frac{f'}{\psi}- \frac{\psi'}{\psi}\frac{f}{\psi} = -\frac{2}{g}\frac{h_f}{A \psi}.
\end{equation}
We observe that $S_{\lambda} > 0$, while $T \geq 0$ follows from $\alpha \in (0,1)$ and $\frac{g}{2} + \alpha = \frac{g}{2} ( 1 + \frac{m_1}{n-2}) > 1$, so
\begin{align*}
    T &\geq \alpha \bigl( \tfrac{g}{2} - \alpha \bigr) + \tfrac{1}{2} (g-2)(1-2 \alpha) = (1-\alpha) \bigl( \tfrac{g}{2} + \alpha-1) > 0.
\end{align*}
We now prove some coercive behaviors of the solutions based on these properties.
\begin{lemma}\label{lemma:general-properties-of-solutions}
    Let $f$ be a positive solution of~\eqref{eqn:rescaled-equation-lambda-ODE} on $(t_1, b) \subset (0,1)$ with $\{f(t_1)\geq 0, f'(t_1) > 0\}$.
    Then, if $\alpha = \frac{m_1}{m_1 + m_2}$ and $t_1 < \sqrt{\alpha}$, the following properties hold:
    \begin{enumerate}[$(i)$]
        \item Either $f$ is strictly increasing for all time or has a unique critical point $t_c$. 
        Furthermore, $h_f > 0$ on $[t_1,b]$ and $h_f' < 0$ for $t < \min \{t_c, \sqrt{\alpha}\}$, while $h_f' > 0$ for $t > \max \{ t_c ,\sqrt{\alpha}\}$.

        \item If $f(t_2) = 0$, then $t_2 \geq \sqrt{\alpha}$.
        If $f(t_1) = f(t_2) = 0$, then $t_1 < \sqrt{\alpha} < t_2$.

        \item The function $f$ is strictly increasing if and only if $h_f$ has a zero at $t_h$ and $h_f, h_f' < 0$ for $t > t_h$. 
    \end{enumerate}
\end{lemma}
\begin{proof}
The proof is a direct adaptation of~\cite{new-minimal-surfaces}*{Proposition 2.2}; the only necessary ingredients are equations~\eqref{eqn:ODE-for-h}\thru\eqref{eqn:2.6NMS}, where~\eqref{eqn:ODE-for-h} requires no specific information about $S_{\lambda}, T$ beyond their positivity.
Indeed, $h(t_1) > 0$ and~\eqref{eqn:ODE-for-h} shows that $\on{sgn} h'(t_*) = - \on{sgn} f'(t_*)$ at a zero of $h$, while~\eqref{eqn:2.6NMS} shows that the region
\[
R := \{ t \geq \sqrt{\alpha} \; : \; h(t) > 0, \; f'(t) < 0 \}
\]
is forward-invariant in $t$, meaning that $t_1 \in R \implies [t_1, b] \subseteq R$.
Moreover, $h$ is positive while $f$ is positive, so $0 < h(t_2) = - \tfrac{g}{2}A(t_2) f'(t_2)$ at a zero $t_2$ of $f$ implies that $t_2 > \sqrt{\alpha}$.
We refer the reader to~\cite{new-minimal-surfaces}*{Proposition 2.2} for the details of these arguments.
\end{proof}

\begin{lemma}\label{lemma:reach-zero-before-hypergeometric}
Consider the hypergeometric function
\begin{equation}\label{eqn:hypergeometric-isoparametric}
    f_M(t) := {}_2 F_1 \Bigl( \frac{n-1}{g} , - \frac{1}{g} ; \frac{m_1 + 1}{2} ; t^2 \Bigr) = {}_2F_1 \Bigl( \frac{m_1 + m_2}{2} + \frac{1}{g} , - \frac{1}{g} ; \frac{m_1 + 1}{2} ; t^2 \Bigr)
\end{equation}
which is the solution of $\cL_M f = 0$ with $\{ f(0) = 1, f'(0) = 0\}$.
This function has a zero $t_M \in (0,1)$, and any solution of equation~\eqref{eqn:rescaled-equation-lambda-ODE} with initial data $\{ f(0) > 0, f'(0) = 0 \}$ either reaches zero or blows up while positive at a point $b \leq t_M$.
\end{lemma}
\begin{proof}
    The function $f_M$ is strictly decreasing and strictly concave with a zero $t_M \in (0,1)$ by~\cite{one-phase-isoparametric}*{Lemma 3.1}.
    Let $f$ be a solution of equation~\eqref{eqn:rescaled-equation-lambda-ODE} with $\{ f(0) > 0, f'(0) = 0\}$, and suppose that it is well-defined for all $t \in [0,t_M]$.
    Then, $f$ satisfies equation~\eqref{eqn:rescaled-equation-lambda-ODE} with $\lambda = 1$, where $f - \frac{g}{2} A_M f'> 0$ on $[0,t_M]$ by Lemma~\ref{lemma:general-properties-of-solutions}.
    Consequently, $\cL_M f < 0$ for $t \in (0,t_M]$ and the expression~\eqref{eqn:self-adjoint-legendre} implies that
    \[
    (p_M f')' + \tfrac{4(n-1)}{g^2} \tfrac{p_M}{1-t^2} f < 0, \qquad (p_M f_M')' + \tfrac{4(n-1)}{g^2} \tfrac{p_M}{1-t^2} f_M = 0.
    \]
    On the interval where $f,f_M$ are both non-negative, we may subtract the two expressions to obtain
    \[
    \bigl(p_M f_M^2 \bigl(\tfrac{f}{f_M} \bigr)' \bigr)' = f_M (p_M f')' - f(p_M f'_M)' < 0 \, ,
    \]
    so $f$ cannot remain positive up to $t = t_M$.
\end{proof}

We have the following compactness results for solutions of the equation~\eqref{eqn:rescaled-equation-lambda-ODE}.

\begin{lemma}\label{lemma:slope-infinity-solution}
    Fix some $t_1 < \sqrt{\frac{m_1}{m_1 + m_2}}$ and let $f^{(N)}$ denote the solution of equation~\eqref{eqn:rescaled-equation-lambda-ODE}, for $\lambda>0$, with initial data $\{ f^{(N)}(t_1) =0, f^{(N)}{}'(t_1) = N \}$.
    Then, there exists a solution $f^{(\infty)}$ of equation~\eqref{eqn:rescaled-equation-lambda-ODE} defined on an interval $[t_1,b]$ such that $f^{(\infty)}(t_1) = 0$ and $f^{(N)} \to f^{(\infty)}$ in $C^{\infty}_{\textup{loc}}(t_1,b)$ as $N \to \infty$.
    We call $f^{(\infty)}$ the solution of equation~\eqref{eqn:rescaled-equation-lambda-ODE} with initial data $\{ f^{(\infty)}(t_1) =0 , f^{(\infty)}{}'(t_1) = \infty\}$.
\end{lemma}

\begin{lemma}\label{lemma:convergence-of-solutions}
    Consider an interval $[d_1,d_2] \subset (0, \sqrt{\frac{m_1}{m_1 + m_2}})$ and fix $a \in (0,\infty]$.
    Suppose that the solution $f_{t_1}$ of equation~\eqref{eqn:rescaled-equation-lambda-ODE} with initial data $\{ f(t_1) = 0, f'(t_1) = a \}$ reaches a second zero $\tau(t_1) > \sqrt{\frac{m_1}{m_1+m_2}}$ with finite derivative when $t_1 = d_1$.
    Then, either this property holds for all $t_1 \in [d_1, d_2]$, or there exists a $t_* \in (d_1,d_2]$ so that the solution $f_{t_*}$ has a zero $\tau(t_*) > \sqrt{\frac{m_1}{m_1 + m_2}}$ with $f'_{t_*}(t) \to - \infty$ as $t \uparrow \tau(t_*)$.
\end{lemma}
\begin{proof}
    These results are direct adaptations of the corresponding compactness properties obtained in~\cite{new-minimal-surfaces}*{Lemmas 2.5 \& 2.6} mutatis mutandis upon substituting the appropriate quantities in equations~\eqref{eqn:ODE-for-h}\thru\eqref{eqn:2.6NMS} and ODE~\eqref{eqn:rescaled-equation-lambda-ODE} for arbitrary $(g,m_1,m_2)$.
\end{proof}

\subsection{Auxiliary quantities}\label{section:auxiliary-quantities}
To prepare for the main result, we introduce the Lyapunov function
\begin{equation}\label{eqn:Psi}
\Psi := f h - \tfrac{1}{n-2} = f(t) \bigl[ f(t) - \tfrac{g}{2} A_M(t) f'(t) \bigr] - \tfrac{1}{n-2} \, .
\end{equation}
The function $\Psi$ is the adaptation to our setting of the key quantity in~\cite{FTW-1}*{\S 3.3}, which was used to prove the uniqueness of the Clifford torus among minimal surfaces with a bi-orthogonal symmetry and establish a ``blowup-or-zero'' dichotomy for solutions of equation~\eqref{eqn:ode-star}.
In our general situation, the quantity $\Psi$ does not vanish on free-boundary solutions, unlike the case of the Lawson cone; however, $\Psi$ is constructed to satisfy $\Psi(t_*) = 0$ at a point where $\{ f(t_*) = 0, f'(t_*) = - \infty \}$, which is sufficient for the desired uniform bound.
\begin{proposition}\label{prop:Psi-one-sided}
    Suppose that $g \geq 3$ and consider a point $t_* \geq \sqrt{\frac{m_1}{m_1 + m_2}}$ and a solution $f$ of equation~\eqref{eqn:ode-star} with $\{ f(t_*) > 0, f'(t_*) \leq 0 \}$.
    \begin{enumerate}[(i)]
        \item If $\Psi'(t_*) = 0$, then $A_M(t_*) = f'(t_*) = 0$ or $\Psi(t_*) < 0$.
        \item If $\Psi(t_*) = 0$, then $A_M(t_*) = f'(t_*) = 0$ or $\Psi'(t_*) > 0$.
    \end{enumerate}
\end{proposition}
\begin{proof}
Let $\mu := m_1 + m_2$ for brevity.
Recall that $A'(t) = 1 + \frac{m_1}{\mu} t^{-2}$, so $A'(t) + t^{-1} A(t) = 2$ and
\begin{align*}
    \Psi' = - (g-2) ff' + \tfrac{g}{2} A \bigl( t^{-1} ff' - (f')^2 - f f'' \bigr).
\end{align*}
Let us denote $q := \frac{f'}{f}$ in what follows, so $\Psi = f^2 (1 - \frac{g}{2} Aq) - \frac{1}{n-2}$.
If $A(t_*) = 0$, then $\Psi'(t_*) = 0$ forces $f'(t_*) = 0$ above; thus, we may consider $t_* > \sqrt{\frac{m_1}{\mu}}$, so $A(t_*) > 0$.
Using equation~\eqref{eqn:ode-star} and the relation $n-2 = \frac{g \mu}{2}$, we substitute $f''$ in the expression for $\Psi'$ to obtain
\begin{align*}
    \frac{\Psi'}{Af^2} &= - \frac{g-2}{A} q + \frac{g}{2} \left( \frac{1-2t^2}{t(1-t^2)} - q \right) q + \! \frac{2}{g(1-t^2)} \! + \! \frac{\mu}{1-t^2} \left( 1 +  \frac{g^2}{4} (1-t^2) \frac{q^2 f^2}{1+f^2} \right) \left( 1 - \frac{g}{2} A q \right) .
\end{align*}
We therefore define the expression
\begin{equation}\label{eqn:F(t,q)-expression}
\begin{split}
F(t,q) = P_1(t) q + P_0(t) - \tfrac{g}{2} q^2 , \quad P_0(t) := \tfrac{g \mu + 2}{g(1-t^2)}, \quad P_1(t) := - \tfrac{g-2}{A} + \tfrac{g}{2} \tfrac{1- 2t^2}{t(1-t^2)} - \tfrac{g \mu}{2(1-t^2)} A,
\end{split}
\end{equation}
which leads us to
\begin{equation}\label{eqn:Psi-prime-expression}
    \frac{\Psi'}{Af^2} = F(t,q) + \frac{g^2 \mu}{4} q^2 \left( 1 - \frac{g}{2} A q \right) \frac{f^2}{1+f^2}  \, .
\end{equation}
We now examine parts $(i)$ and $(ii)$ of the claim separately.

\smallskip \noindent \textbf{Proof of $(i)$.}
If $\Psi'(t_*) = 0$, we solve the equality~\eqref{eqn:Psi-prime-expression} in terms of $f^2$ to write
\begin{equation}\label{eqn:f2-expression}
f^2 = - \frac{F(t,q)}{F(t,q) + \frac{g^2 \mu}{4} q^2 ( 1 - \frac{g}{2} Aq)} \, .
\end{equation}
Writing $\frac{1}{n-2} = \frac{2}{g \mu}$, we therefore obtain $\Psi(t_*) = f^2 ( 1 - \frac{g}{2} Aq) - \frac{1}{n-2}$ as
\[
\Psi(t_*) = - \frac{(1 - \frac{g}{2} Aq) F(t,q) + \frac{2}{g\mu} \bigl[ F(t,q) + \frac{g^2 \mu}{4} q^2 (1 - \frac{g}{2} Aq) \bigr]}{F(t,q) + \frac{g^2 \mu}{4} q^2( 1 - \frac{g}{2} A q)}.
\]
Since $1 - \frac{g}{2} A q > 0$, the expression $\Psi'(t_*) = 0$ in terms of $F(t,q)$ shows that $F(t,q) < 0$ at $t_*$, so $F(t,q) + \frac{g^2 \mu}{4} q^2(1 - \frac{g}{2} Aq) > 0$ due to the relation~\eqref{eqn:f2-expression}. 
In particular, since $P_0(t) = \frac{g \mu+2}{g(1-t^2)}$ makes $F(t,0) = P_0(t) > 0$ in~\eqref{eqn:F(t,q)-expression}, the assumption $f'(t_*) \leq 0$ forces $q<0$ strictly.
Moreover, $\Psi(t_*) < 0$ amounts to the positivity of the numerator expression $N(t,q)$, given by
\begin{align*}
N(t,q) &:= \tfrac{2}{g \mu} F(t,q) + \bigl( \tfrac{g}{2} q^2 + F(t,q) \bigr) (1 - \tfrac{g}{2} Aq) \\
&\;= \tfrac{2}{g \mu} \bigl( P_1(t) q + P_0(t) - \tfrac{g}{2} q^2 \bigr) + \bigl( P_1(t) q + P_0(t) \bigr) \bigl( 1 - \tfrac{g}{2} A q \bigr) \\
&\;= \tfrac{(g \mu + 2)^2}{g^2 \mu (1-t^2)} + \left[ \bigl( 1 + \tfrac{2}{g \mu} \bigr) P_1 - \tfrac{g}{2} A P_0 \right] q - \bigl[ \tfrac{1}{\mu} + \tfrac{g}{2} A P_1 \bigr] q^2 \, .
\end{align*}
Since $q < 0$ and the first term is strictly positive, the property $N(t,q) > 0$ will follow from showing that both bracketed terms are non-positive.
Expressing $A = \frac{\mu t^2 - m_1}{\mu t}$, we obtain
\begin{align*}
    \bigl( 1 + \tfrac{2}{g \mu} \bigr) P_1 - \tfrac{g}{2} A P_0 &= - \tfrac{g \mu+2}{2 g \mu t (1-t^2) (\mu t^2- m_1)} Q_1(x), \qquad
    \tfrac{1}{\mu} + \tfrac{g}{2} A P_1 = - \tfrac{1}{4 \mu t^2(1-t^2)} Q_2(x),
\end{align*}
where $x := \mu t^2 - m_1$ and $Q_1(x), Q_2(x)$ are quadratics given by
\begin{align*}
    Q_1(x) &:= (2 g \mu + 4) x^2 + (g-4) (m_2 - m_1) x + 2 (g-2) m_1 m_2, \\
    Q_2(x) &:= (g \mu+2)^2 x^2 + \bigl( g \mu(g-4) - 4 \bigr) (m_2 - m_1) x + 2 \bigl( g \mu (g-2) - 2 \bigr) m_1 m_2 \, .
\end{align*}
The claim for $N(t,q)$ is reduced to proving that $Q_1, Q_2 \geq 0$ are non-negative; equivalently, we show that their discriminants satisfy $\Delta_1, \Delta_2 \leq 0$.
Since $m_1, m_2$ are positive integers with $m_1 + m_2 = \mu$, we have $(m_2 - m_1)^2 \leq (\mu - 2)^2$ and $m_1 m_2 \geq \mu-1$, so
\begin{align*}
    \Delta_1 &\;= (g-4)^2 (m_2 - m_1)^2 - 8 (2 g \mu + 4) (g-2) m_1 m_2 \\
    &\;\leq (g-4)^2 (\mu-2)^2 - 8 (2 g\mu + 4) (g-2) (\mu-1) =: - R_1(\mu,g), \\
R_1(\mu,g) &:= 15 \mu^2 g^2 - 24 \mu^2 g - 16 \mu^2 - 12 \mu g^2 + 32 \mu g - 4 g^2 \, .
\end{align*}
For $\mu = m_1 + m_2 \geq 2$ and $g \geq 3$, we find
\[
\partial_{\mu} R_1 = 2 ( 15 g^2 - 24 g - 16) \mu - 12 g^2 + 32 g \geq 16 (3g^2 - 4g - 4) > 0,
\]
so $R_1$ is increasing in $\mu$.
Since $R_1(2,g) = 32(g-2) (g+1) > 0$, we conclude that $\Delta_1 \leq - R_1(\mu,g) < 0$, whereby $Q_1(\mu t^2 - m_1) > 0$ as claimed.
For $Q_2$, we likewise bound
\allowdisplaybreaks{
\begin{align*}
   \Delta_2 &\;= \bigl( g \mu(g-4) - 4 \bigr)^2 (m_2 - m_1)^2 - 8 (g \mu+2)^2 \bigl( g \mu(g-2) - 2 \bigr) m_1 m_2 \\
   &\;\leq \bigl( g \mu(g-4) - 4 \bigr)^2 (\mu - 2)^2 - 8 (g \mu+2)^2 \bigl( g \mu(g-2) - 2 \bigr) (\mu-1) =: - \mu^2 R_2(\mu,g), \\
   R_2(\mu,g) &:= 7 \mu^2 g^4 - 8 \mu^2 g^3 - 16 \mu^2 g^2 - 4 \mu g^4 + 16 \mu g^3 - 8 \mu g^2 - 32 \mu g - 4 g^4 + 16 g^2 - 16 .
\end{align*}}
For $\mu = m_1 + m_2 \geq 2$ and $g \geq 3$, we again compute
\begin{align*}
    \partial_{\mu} R_2 &= 2 \mu (7 g^4 - 8 g^3 - 16 g^2) - 4 g^4 + 16 g^3 - 8 g^2 - 32 g \\
    &\geq 4 ( 7 g^4 - 8 g^3 - 16 g^2) - 4g^4 + 16 g^3 - 8 g^2- 32g = 8 g \bigl( 3g^3 - 2g^2 - 9g - 4) ,
\end{align*}
which is a positive quantity since the resulting cubic equals $32$ for $g=3$ and is strictly increasing for $g \geq 3$.
Hence, $\partial_{\mu} R_2 > 0$ for $\mu \geq 2, g \geq 3$, and $R_2(\mu,g) \geq R_2(2,g) = 16 (g+1)^2 (g^2 - 2g-1) > 0$; this implies that $\Delta_2 \leq 0$ and $Q_2(\mu t^2 - m_1) > 0$ as claimed.
We conclude that $N(t,q) > 0$, so $\Psi(t_*) < 0$ at any critical point of $\Psi$; this proves the first part.

\smallskip \noindent \textbf{Proof of $(ii)$.}
If $\Psi(t_*) = 0$, then we substitute $f^2( 1 - \frac{g}{2} Aq) = \frac{2}{g \mu}$ in the expression~\eqref{eqn:Psi-prime-expression} to obtain
\[
\Psi'(t_*) = \frac{A f^2}{n-1- \frac{g^2 \mu}{4} Aq} \tilde{L}(t,q), \qquad \tilde{L}(t,q) := (P_1 q + P_0) \bigl( n-1- \tfrac{g^2 \mu}{4} Aq \bigr) - \tfrac{g}{2} q^2.
\]
First, $A(t_*) = 0$ results in $\Psi'(t_*) = - (g-2) (ff')(t_*)$ as above, so $\Psi'(t_*) > 0$ unless $f'(t_*) = 0$.
For $t > t_*$, the denominator is $n- 1 - \frac{g^2 \mu}{4} Aq = \frac{g \mu}{2} (1 - \frac{g}{2} Aq) + 1 > 0$ and $A(t_*) > 0$, so $\on{sgn} \Psi'(t_*) = \on{sgn} \tilde{L}(t,q)|_{t=t_*}$.
We set $x := \mu t^2_* - m_1$ and $y := - \frac{g}{2} A q$, so $q = - \frac{2 \mu t}{g(\mu t^2 - m_1)}y$.
    Expanding $P_0, P_1$ as in part $(i)$, we obtain the quadratic expression
    \begin{equation}\label{eqn:N-N-tilde-equation}
    \begin{split}
    \tilde{L}(t,q) &= \frac{L(x,y)}{2g ( \mu t_*^2 - m_1)^2 (1-t_*^2)}, \qquad
    L(x,y) = Q_2(x) y^2 + (g \mu + 2) Q_1(x) y + (g \mu+2)^2 x^2,
    \end{split}
    \end{equation}
    for $Q_2, Q_1$ the polynomials computed in part $(i)$.
    These satisfy $Q_i \geq 0$ by the computation of part $(i)$, so $L(x,y) > 0$ and $\Psi'(t_*) > 0$ as claimed.
\end{proof}

\begin{lemma}\label{lemma:psi-detect-blowup}
Let $f$ be a solution of equation~\eqref{eqn:ode-star} with $g \geq 3$. 
If there exists a point $t_0 \geq \sqrt{\frac{m_1}{m_1+m_2}}$ where $\Psi(t_0) \geq 0$ and $\Psi'(t_0) \geq 0$, then $f$ does not reach zero in its forward interval of definition.
\end{lemma}
\begin{proof}
We show that there exists a $\delta \geq 0$ such that $\Psi, \Psi'>0$ on $(t_0 + \delta, t_2)$.
First, we obtain
\begin{equation}\label{eqn:condition-on-Psi}
    \Psi(t_0 + \delta) > 0, \qquad \Psi'(t_0 + \delta) > 0, \qquad \text{for some } \; \delta \geq 0.
\end{equation}
If this property is not already satisfied for $\delta = 0$, then either $\Psi(t_0) = 0$ or $\Psi'(t_0) = 0$.
Proposition~\ref{prop:Psi-one-sided} shows that if $t_0 > \sqrt{\frac{m_1}{m_1 + m_2}}$, then 
\[
\Psi(t_0) =0\implies \Psi'(t_0) >0, \qquad \Psi'(t_0) = 0 \implies \Psi(t_0) < 0,
\]
contradicting the above assumptions. 
Therefore, $t_0 = \sqrt{\frac{m_1}{m_1 + m_2}}$ and the computation of $\Psi'$ from Proposition~\ref{prop:Psi-one-sided} gives $0 = \Psi'(t_0) = - (g-2) (ff')(t_0)$, so $f'(t_0) = 0$.
Moreover, equation~\eqref{eqn:ode-star} implies $f''(t_0) = - \frac{4 (n-1) (m_1+m_2)}{m_2 g^2} f(t_0)$, so
\begin{align*}
    \Psi'' (t_0) &= - (g-2) f f'' + \tfrac{g}{2} A' (t_0) (- f f'') = - 2(g-1) (ff'') (t_0) \\
    &= \tfrac{8 (m_1 + m_2) (g-1) (n-1)}{m_2 g^2} f (t_0)^2 > 0.
\end{align*}
Since $\{ \Psi(t_0) \geq 0, \Psi'(t_0) = 0, \Psi''(t_0) > 0 \}$, we again obtain~\eqref{eqn:condition-on-Psi} for $\delta>0$ small.
At a first critical point $t_* > \sqrt{\frac{m_1}{m_1 + m_2}}$ of $\Psi$, we would have $\Psi' > 0$ on $[t_0, t_*)$, so $\Psi(t_*) > 0$, contradicting Proposition~\ref{prop:Psi-one-sided}; thus, we must have $\Psi, \Psi'>0$ on $(t_0 + \delta,t_2)$.

Suppose for contradiction that the function $f$ reaches a zero at a point $t_2$, so $h = f - \frac{g}{2}Af' > 0$ on $[t_0, t_2]$ by Lemma~\ref{lemma:general-properties-of-solutions}.
If $f$ crossed zero with finite derivative $|f'(t_2)| < \infty$, then~\eqref{eqn:Psi} yields $\Psi(t_2) = - \frac{1}{n-2} < 0$, a contradiction. 
If $f$ has derivative blowup at $t_2$, we study the asymptotic behavior of the solution of~\eqref{eqn:rescaled-equation-lambda-ODE} with $f'(t_2) = - \infty$ near $t_2$ and match dominant terms to obtain 
\[
(1-t_2^2) f'' \sim - (1-t_2^2) \tfrac{\lambda (n-2)}{1 + \lambda f(t_2)^2} \bigl( - \tfrac{g}{2} \bigr) A_M(t_2) (f')^3 \implies f'' \sim C (f')^3,
\]
for $C = \frac{(n-2) \lambda g A_M(t_2)}{2 ( 1 + \lambda f(t_2)^2)}$.
Writing this relation as $\bigl( - \frac{1}{2} (f')^{-2} \bigr)' \sim C + o(1)$, we find that
\begin{equation}\label{eqn:f-prime-t,f-t}
 \begin{split}
 f(t) &= f(t_2) + 2 \sqrt{ \tfrac{1 + \lambda f(t_2)^2}{(n-2) \lambda g \, |A_M(t_2)| }} \, |t_2 - t|^{\frac{1}{2}} + O ( |t_2 -t|), \\
 f'(t)& = - \on{sgn} A_M(t_2) \, \sqrt{\tfrac{1 + \lambda f(t_2)^2}{(n-2) \lambda g \, |A_M(t_2)|}} \, |t - t_2|^{- \frac{1}{2}} + O (1),
\end{split}
\end{equation}
and the function $f(t)$ admits a power series expansion in $\sqrt{t_2 - t}$, cf.~\cite{new-minimal-surfaces}*{Lemma 2.6}.
In our situation with $(\lambda,f(t_2)) = (1,0)$, for $t$ near $t_2$ this computation leads to
\[
\Psi = f \bigl(f - \tfrac{g}{2} A_M f' \bigr) - \tfrac{1}{n-2} = \tfrac{g}{2} A_M(t_2) \cdot 2 \Bigl( \sqrt{\tfrac{1}{(n-2) g A_M(t_2)}} \Bigr)^2 - \tfrac{1}{n-2} + O ( \sqrt{t_2 - t}) = O (\sqrt{t_2 - t}),
\]
so $\lim_{t \uparrow t_2} \Psi(t) = 0$.
However, this is also impossible due to $\Psi, \Psi' > 0$ on $(t_0 + \delta, t_2)$.

It remains to consider the case $t_2=1$, where the above computation for $\Psi$ applies verbatim if $|f'(1)| < \infty$; we show that equation~\eqref{eqn:ode-star} does not admit a solution with $\{ f(1) = 0, f'(1) = - \infty\}$.
For $t$ close to $1$, we use $f(1) = 0$ and $A_M(1) = \frac{m_2}{m_1 + m_2}$ to find $f - \frac{g}{2} A_M(t) f' > - \frac{g m_2}{4(m_1 + m_2)} f'$, as well as $f - tf' + \frac{4-g^2}{g^2} f > - \frac{1}{2} f'$ and $\frac{1}{1+f^2} > \frac{1}{2}$.
Using $\frac{4(n-2)}{g^2} \cdot \frac{g m_2}{4 (m_1 + m_2)} = \frac{m_2}{2}$, we obtain
\[
f'' - \tfrac{m_2 + \frac{1}{2}}{2(1-t)} (f')  + (n-2) (f')^3 < 0.
\]
We now consider the function $p := \frac{1}{(f')^2}$, which has $p' = - 2 \frac{f''}{(f')^3}$ and
\[
- \tfrac{1}{2} p' - \tfrac{3}{4(1-t)} p + (n-2) > 0, \qquad \text{whereby } \qquad \bigl[ (1-t)^{- \frac{3}{2}} p(t) \bigr]' \leq -(1-t)^{- \frac{3}{2}},
\]
for $n \geq 3$.
Moreover, $p(1) = 0$ due to $f'(1)  = - \infty$.
Integrating the above bound on $(t_1,t)$, we find
\[
0 \leq (1-t)^{- \frac{3}{2}} p(t) \leq (1-t_1)^{- \frac{3}{2}} p(t_1) - \int_{t_1}^t (1-\tau)^{- \frac{3}{2}} \, d \tau.
\]
Taking $t \uparrow 1$ here produces a contradiction, since $\int_{t_1}^t (1-\tau)^{- \frac{3}{2}} \, d \tau \to \infty$.
Therefore, $\{ f(1) = 0 , f'(1) = - \infty \}$ cannot occur, so $\lim_{t \uparrow t_2} \Psi(t) \leq 0$ in all cases if $f$ reaches zero, which we have ruled out whenever~\eqref{eqn:condition-on-Psi} holds.
This proves our claim.
\end{proof}

The other key ingredient will be an analysis of the local behavior of solutions of equation~\eqref{eqn:ode-star} near the point $t = \sqrt{\frac{m_1}{m_1+m_2}}$, which has the significance of being the turning point where the dominant cubic term $(f')^3$ on the right-hand side of~\eqref{eqn:ode-star} changes sign.
Geometrically, this point satisfies $A ( \sqrt{\frac{m_1}{m_1+m_2}} ) = 0$, so it corresponds via~\eqref{eqn:A-gm1m2} to the distinguished minimal leaf along the isoparametric foliation. 
Moreover, $t > \sqrt{\frac{m_1}{m_1+m_2}}$ corresponds to mean-convex leaves, which must contain the free boundary due to the Hopf lemma. 

\begin{proposition}\label{prop:i-will-take-a-limit}
Let $f_i$ be a sequence of positive solutions to equation~\eqref{eqn:ode-star} defined over a sub-interval of $\bigl[ 0, \sqrt{\frac{m_1}{m_1 + m_2}} \bigr]$, with $f_i - \frac{g}{2} A_M f'_i > 0$ and $\limsup_i f_i( \sqrt{\frac{m_1}{m_1 + m_2}}) < + \infty$ as $i \to \infty$.
Given a sequence of $\epsilon_i \downarrow 0$, we define the rescaled variables $r$ and functions $y_i(r)$ by
\[
r := \epsilon_i^{-1} \bigl( t - \sqrt{\tfrac{m_1}{m_1 + m_2}} \bigr), \qquad y_i(r) := f_i(t).
\]
Then, a subsequence of the functions $y_i$ converges in $C^{\infty}_{\textup{loc}}$ to a function $y_{\infty}$ solving the equation
\begin{equation}\label{eqn:phi-limit-equation} \tag{$\ddagger$}
        y''(r) + c\frac{y'(r)^2}{1 + y(r)^2} \bigl[ y(r) - gry' (r) \bigr] = 0, \qquad \text{where } \; c = n-2,
\end{equation}
satisfying $y_{\infty}(0) = L$.

Any non-constant solution of equation~\eqref{eqn:phi-limit-equation} is strictly monotone, with inverse $r(y)$ given by
\begin{equation}\label{eqn:r(y)-inverse}
\begin{split}
    r(y) &\;= C_0 \, \cdot {}_2F_1 \bigl( a_+ , a_- ; \tfrac{1}{2} ; - y^2 \bigr) + C_1 y \cdot {}_2F_1 \bigl( a_+ + \tfrac{1}{2}, a_- + \tfrac{1}{2}; \tfrac{3}{2} ; - y^2), \\
    a_{\pm} &:= \tfrac{1}{4} \bigl( - (c+1) \pm \sqrt{(c+1)^2 - 4cg} \bigr) \, .
\end{split}
\end{equation}
\end{proposition}
\begin{proof}
We denote $t_{\alpha} := \sqrt{\frac{m_1}{m_1 + m_2}}$, for brevity, and follow the steps of~\cite{new-minimal-surfaces}*{Proposition 2.8}.
Consider the rescaled variable by $t_i(r) := t_{\alpha} + \epsilon_i r$ and denote $h_i := f_i - \frac{g}{2}A_M f'_i$, so that
    \[
    y_i(r) := f_i(t_i(r)), \quad \text{so that} \quad y'_i(r) = \epsilon_i f'_i(t_i(r)), \qquad y''_i(r) = \epsilon_i^2 f''_i(t_i(r)).
    \]
    Moreover, $A_M(t_{\alpha}) = 0$ and $A'_M(t_{\alpha}) = 2$ expresses $A_M(t_i(r)) = 2 \epsilon_i r + O(\epsilon_i^2)$, and
    \begin{equation}\label{eqn:hi-term-to-h-infty}
h_i = y_i(r) - \epsilon_i^{-1} \tfrac{g}{2} A_M(t_i(r)) y'_i(r) = y_i(r) - g r y'_i(r) - O(\epsilon_i) y'_i.
\end{equation}
Applying these transformations in~\eqref{eqn:ode-star}, we obtain the rescaled equation
\begin{equation}\label{eqn:Ri-phii(r)-bounds}
\begin{split}
(1- t_i(r)^2) y''_i &+ (n-2)(1-t_i(r)^2) \frac{y'_i(r)^2}{1 + y_i(r)^2} \bigl( y_i(r)  - g r y'_i(r) \bigr) = \epsilon_i E_i(r), \\
& E_i(r) := t_i(r) y'_i(r) - \tfrac{4}{g^2} \epsilon_i y_i(r) - \tfrac{2(m_1+m_2)}{g} \epsilon_i h_i(r).
\end{split}
\end{equation}
Using our assumption that $h_i > 0$ in the above equation, we deduce that
\[
(1 - t^2) y''_i \leq - (n-2) (1-t^2) \tfrac{(y'_i)^2}{1 + y_i^2} h_i + \epsilon_i E_i \leq C \epsilon_i |y'_i| + C \epsilon_i^2 |y_i|,
\]
which implies that $\limsup_{\epsilon_i \downarrow 0} \| y''_i \|_{C^0(K)}$ is uniformly bounded in $i$ on any compact interval $K$ by a constant depending only on $K$.
Therefore, the functions $\{ y_i\}, \{ y'_i \}$ are also uniformly bounded, so we can apply the Arzel\`a-Ascoli theorem and a diagonal argument to extract a subsequential limit $y_{i_k} \to y_{\infty}$ in $C^2$, so also in $C^{\infty}$, as solutions of elliptic equations with right-hand sides converging uniformly.
The uniform bounds on $\| y_i \|_{C^{2,\alpha}}$ imply that the remainder terms $E_i$ of~\eqref{eqn:Ri-phii(r)-bounds} are bounded uniformly in $i$, so $\epsilon_i E_i \to 0$ uniformly as $\epsilon_i \to 0$ and the limit $y_{\infty}$ satisfies the equation~\eqref{eqn:Ri-phii(r)-bounds} with zero right-hand side.
We may thus cancel the factor $1-t^2$, which is uniformly positive due to $t_i(r) \to t_{\alpha}$ as $\epsilon_i \downarrow 0$, to see that $y_{\infty}$ satisfies the limit equation~\eqref{eqn:phi-limit-equation} with $c = n-2$ and 
\[
y_{\infty}(0) = \lim_{i \to \infty} y_i(0) = \lim_{i \to \infty} f_i (t_i(0)) = \lim_{i \to \infty} f_i(t_{\alpha}) = L
\]
proving the first part.
For the second part, we use equation~\eqref{eqn:phi-limit-equation} to see that $y'(r_*) = 0 \implies y''(r_*) = 0$ at any point $r_*$, hence ODE uniqueness would force $y \equiv y(r_*)$ identically since the constant function $\{ y = y(r_*) \}$ is a solution of~\eqref{eqn:phi-limit-equation}.
Consequently, any non-constant solution of~\eqref{eqn:phi-limit-equation} is strictly monotone, so it has a well-defined monotone inverse function $r(y)$, which satisfies $y'(r) = \frac{1}{r'(y)}$ and $y''(r) = - \frac{r''(y)}{r'(y)^3}$.
The left-hand side of~\eqref{eqn:phi-limit-equation} becomes $- r''(y) + c \frac{y r'(y) - gr}{1+y^2}$, so
\begin{equation}\label{eqn:r-ode-in-terms-of-y}
    (1+y^2) r'' - c y r' + cg r = 0
\end{equation}
produces a linear ODE for $r = r(y)$.
The substitution $z := -y^2$ transforms~\eqref{eqn:r-ode-in-terms-of-y} into
\[
z (1-z) \tilde{r}'' + \bigl( \tfrac{1}{2} + \tfrac{c-1}{2} z \bigr) \tilde{r}' - \tfrac{cg}{4} \tilde{r} = 0, \qquad \tilde{r}(z) := r(y),
\]
which is the Gauss hypergeometric ODE
\[
z (1-z) \tilde{r}'' + \bigl[ \gamma - (a_+ + a_- + 1) z \bigr] \,\tilde{r}' - a_+ a_- \tilde{r} = 0, \qquad (\gamma, a_+ + a_-,a_+ a_-) = (\tfrac{1}{2}, - \tfrac{c+1}{2}, \tfrac{cg}{4}).
\]
Solving these relations for $a_{\pm}$ produces the first solution in~\eqref{eqn:r(y)-inverse}.
The second, odd solution is obtained analogously, upon considering $z := -y^2$ and $\hat{r}(z) := y^{-1} r(y)$ and computing the even hypergeometric solution for the resulting Gauss ODE satisfied by $\hat{r}(z)$.
\end{proof}

\begin{lemma}\label{lemma:f-infty-analysis}
Let $f_a$ denote the solution of equation~\eqref{eqn:ode-star} with initial data $\{ f_a(0) = a, f'_a(0) = 0 \}$.
If $(n-1)^2 < 4g(n-2)$, then there is a constant $c_0 > 0$ such that $f_a \bigl( \sqrt{\frac{m_1}{m_1 + m_2}} \bigr) > c_0 a$.
In particular, $f_a \bigl( \sqrt{\frac{m_1}{m_1 + m_2}} \bigr) \to \infty$ as $a \uparrow \infty$.
\end{lemma}
\begin{proof}
Let us denote $t_{\alpha} := \sqrt{\frac{m_1}{m_1 + m_2}}$ for brevity.
We define $f_{\infty}$ to be the solution of equation~\eqref{eqn:rescaled-equation-lambda-ODE} upon formally taking $\lambda \to \infty$, namely
    \begin{equation}\label{eqn:f-infty-solution}
        (1-t^2) f'' + (f-tf') + \frac{4-g^2}{g^2} f + (n-2) \left( \frac{4}{g^2} + (1-t^2) \frac{(f')^2}{f^2} \right) \left( f - \frac{g}{2} A_M(t) f' \right) = 0 
    \end{equation}
    with initial data $\{ f_{\infty}(0) = 1, f'_{\infty}(0) = 0 \}$.   
    Applying the results of the ODE Lemma~\ref{lemma:general-properties-of-solutions} for equation~\eqref{eqn:rescaled-equation-lambda-ODE} and letting $\lambda \to \infty$, we see that $f_{\infty}$ satisfies $f_{\infty} - \frac{g}{2} A_M f'_{\infty} > 0$, it is strictly decreasing and strictly concave, and is strictly positive on its maximal interval of existence.

    Our proof will proceed in three steps: first, we show that $f_{\infty}$ remains smooth on $[0, t_\alpha )$; next, we prove that $f_a \geq a f_\infty$; finally, we show that $f_\infty (t_\alpha) \geq c_0 > 0$ if $(n-1)^2<4g(n-2)$.

\smallskip \noindent \textbf{Step 1:}
The quantity $q := \frac{f'_{\infty}}{f_{\infty}}$ satisfies the equation
\begin{equation}\label{eqn:q-ODE-f-infty}
    (1-t^2) (q' + q^2) + (1 - tq) + \tfrac{4-g^2}{g^2} + (n-2) \bigl( \tfrac{4}{g^2} + (1-t^2) q^2 \bigr) \bigl( 1 - \tfrac{g}{2} A_M(t) q \bigr) = 0
\end{equation}
with initial condition $q(0) = 0$ and $q < 0$ for $t>0$.
At a point $t_*$ where $\limsup_{t \to t_*} |q(t)| = + \infty$, we have $q,q' \to - \infty$ as well, and the dominant terms of equation~\eqref{eqn:q-ODE-f-infty} produce
\[
(1-t^2)(q' + q^2) \sim - (n-2)(1-t^2) q^2 \bigl(1 - \tfrac{g}{2} A_M(t) q \bigr).
\]
If $t_* < t_{\alpha}$, this would lead to $q' \sim \tfrac{(n-2) g}{2} A_M(t_*) q^3$, producing a contradiction because $A_M(t_*) < 0$ so the right-hand term tends to $+ \infty$.
Thus, $f_{\infty}$ has no blow-up for $t < t_{\alpha}$.
Moreover, $f_{\infty} ( t_{\alpha} ) = 0$ if and only if $f_{\infty}$ becomes singular there, as seen by equation~\eqref{eqn:f-infty-solution}.

\smallskip \noindent \textbf{Step 2.}
Next, we claim that
\begin{equation}\label{eqn:lower-bound-at-ta}
    f_a(t) \geq a \cdot f_\infty(t), \qquad \text{for } \; t \in [0,t_\alpha].
\end{equation}
To this end we consider the quantity $w := \frac{f_a}{f_{\infty}}$, which has $\{ w(0) = a, w'(0) = 0 \}$ and satisfies a differential equation derived from~\eqref{eqn:ode-star} and~\eqref{eqn:f-infty-solution}.
At a critical point $t_*$ of the function $w$, we compute
\[
f_a = w f_{\infty}, \qquad f'_a = w  f'_{\infty}, \qquad f''_a = w'' f_{\infty} + w f''_{\infty}, \qquad f_a - \tfrac{g}{2} A_M f'_a = w \bigl( f_{\infty} - \tfrac{g}{2} A_M f'_{\infty} \bigr).
\]
Expressing equation~\eqref{eqn:ode-star} in terms of $(w, f_{\infty})$ at $t_*$ and subtracting off the equation for $f_{\infty}$, we find
\[
(1-t_*^2) (f_{\infty} w'')(t_*) - (n-2) w(t_*) (1 - t_*^2) \frac{f'_{\infty}(t_*)^2}{f_{\infty}(t_*)^2 (1 + f_a(t_*)^2)} \Bigl( f_{\infty} - \tfrac{g}{2} A_M f'_{\infty} \Bigr) = 0.
\]
To obtain the second step, we simplified the difference of the nonlinear factors
\[
\bigl( \tfrac{4}{g^2} + (1-t_*^2) \tfrac{(f'_a)^2}{1 + f_a^2} \bigr) - \bigl( \tfrac{4}{g^2} + (1-t_*^2) \tfrac{(f'_{\infty})^2}{f_{\infty}^2} \bigr) = - (1-t_*^2) \tfrac{f'_{\infty}(t_*)^2}{f_{\infty}(t_*)^2 ( 1 + f_a(t_*)^2) }
\]
due to $f'_a = w f'_{\infty}$ at the critical point $t_*$.
Using the property $f_{\infty} - \frac{g}{2} A_M f'_{\infty} > 0$, we conclude that $w''(t_*) > 0$ holds at any critical point of $w$, so it must be a strict local minimum.

On the other hand, we can prove that $w'(0) = w''(0) = w'''(0)=0$ and $w^{(4)}(0) > 0$.
To this end, we note that the functions $f_a, f_{\infty}$ are real-analytic near $t=0$, as solutions of elliptic equations with
\[
(1-t^2) f'' + (f-tf') + \tfrac{4-g^2}{g^2} f + (n-2) \bigl( \tfrac{4}{g^2} + (1-t^2) \tfrac{(f')^2}{\kappa + f^2} \bigr) \bigl( f - \tfrac{g}{2} A_M f' \bigr) = 0,
\]
where $\kappa = 0$ for $f_{\infty}$ and $\kappa = 1$ for $f_a$.
Consequently, they admit even power series expansions $f = b_{\kappa,0} + b_{\kappa,2} t^2 + b_{\kappa,4} t^4 + O(t^6)$ for $\kappa \in \{0,1\}$ as above.
Then, we may expand
\begin{align*}
f'(t) &= 2 b_{\kappa,2}t + 4 b_{\kappa,4} t^3 + O(t^5), \qquad f''(t) = 2 b_{\kappa,2} + 12 b_{\kappa,4} t^2 + O(t^4), \\
(A_M f')(t) &= \bigl( t - \tfrac{m_1}{m_1 + m_2} t^{-1} \bigr) \bigl( 2 b_{\kappa,2} t + 4 b_{\kappa,4} t^3 + O(t^5) \bigr) = - 2 \tfrac{m_1}{m_1 + m_2} b_{\kappa,2} + 2 \bigl( b_{\kappa,2} - \tfrac{2 b_{\kappa,4}m_1}{m_1 + m_2} \bigr) t^2 + O(t^4).
\end{align*}
Examining the constant term of equations~\eqref{eqn:ode-star} and~\eqref{eqn:f-infty-solution}, we therefore find
\[
2 (1 + m_1) b_{\kappa,2} + \tfrac{4(n-1)}{g^2} b_{\kappa,0}=0, \qquad \implies \qquad b_{\kappa,2} = - \tfrac{2(n-1)}{(m_1+ 1) g^2} b_{\kappa,0},
\]
leading to the expansions $1 - \frac{2(n-1)}{(m_1 + 1) g^2} t^2 + O(t^4)$ for both $\frac{f_a}{a}$ and $f_{\infty}$ near $t=0$.
This shows that $w''(0) = 0$, and since $f_a, f_{\infty}$ are both even functions, $w = \frac{f_a}{f_{\infty}}$ is as well, so $w'(0) = w'''(0) = 0$.

We write the above expansions as
\[
f_a(t) = a \bigl( 1 - \tfrac{2(n-1)}{(m_1 + 1) g^2} t^2 + \hat{b}_1 t^4 + O(t^6) \bigr), \qquad f_{\infty}(t) = 1 - \tfrac{2(n-1)}{(m_1+1) g^2} t^2 + \hat{b}_0 t^4 + O(t^6),
\]
to write $w(t) = a \bigl( 1 + (\hat{b}_1 - \hat{b}_0) t^4 + O(t^6) \bigr)$.
To compute $\hat{b}_{\kappa}$ for $\kappa \in \{0,1\}$, we use the coefficients $b_{\kappa,0}, b_{\kappa,2}$ computed above to match the coefficient of $t^2$ in equations~\eqref{eqn:ode-star} and~\eqref{eqn:f-infty-solution}, which produces
\[
\hat{b}_{\kappa} = - \frac{\bigl( g(m_1 + m_2) + 2 \bigr) \bigl[ b_{\kappa,0}^2 R + \kappa (g-1) (m_1 + 1)^2 (g (m_1 + m_2) + 2g + 2) \bigr] }{2 g^4 (m_1 + 1)^3 (m_1 + 3) (b_{\kappa,0}^2 + \kappa)},
\]
where $b_{1,0} = a$ for $\kappa=1$, and $R$ is the coefficient
\begin{align*}
R
&=  g^2 \bigl(
m_1^3 + m_1^2 m_2 + 5 m_1^2 + 4 m_1 m_2 + 5 m_1 + m_2^2 + m_2 + 2
\bigr) \\
&\quad - g \bigl(
m_1^3 + m_1^2 m_2 + 4 m_1^2 + 4 m_1 m_2 - m_1 - m_2
\bigr) - 2(m_1^2 + 4m_1 + 1).
\end{align*}
Subtracting the two coefficients $\hat{b}_1 - \hat{b}_0$, we arrive at the factorization
\[
\hat{b}_1 - \hat{b}_0 = \tfrac{2 (n-1)^2 (g m_2 + (g-2) m_1) }{g^4(1+a^2) (m_1 + 1)^3 (m_1 + 3)}, \qquad
w^{(4)}(0) = 24 ( \hat{b}_1 - \hat{b}_0) > 0,
\]
since $g m_2 + (g-2) m_1 \geq g m_2 > 0$ for $g \geq 2$.
Using $w'(0) = w''(0) = w'''(0) = 0$, we conclude that $w^{(i)}(t) > 0$ for $1 \leq i \leq 4$ and small $t>0$, so $w$ is initially strictly increasing.
If $w$ had a critical point, this would force $w''(t_*) \leq 0$, contradicting our earlier computation $w''(t_*) > 0$; thus, $w' > 0$ for all $t \in [0, t_{\alpha})$.
Then, $\frac{f_a}{f_{\infty}}= w(t) > w(0) = a$, so $f_a \geq a f_{\infty}$ for all $t \in [0,t_{\alpha})$ proves~\eqref{eqn:lower-bound-at-ta}.

\smallskip \noindent \textbf{Step 3.}
Finally, we show that $f_\infty(t_\alpha) > 0$ if $(n-1)^2 < 4g(n-2)$.
If $f_{\infty} ( t_{\alpha} ) = 0$, then $f'_{\infty} \to - \infty$ as $t \uparrow t_{\alpha}$ by equation~\eqref{eqn:f-infty-solution}.
We let $r := t_{\alpha} - t$, so $A'_M(t_{\alpha}) = 2$ implies $A_M(t) = - 2r + O(r^2)$, and denote $\tilde{q}(r) := r q(t_{\alpha} - r)$, which satisfies the equation
\begin{equation}\label{eqn:q-tilde-equation}
    r \tilde{q}' = \tilde{q} \bigl( 1  + (n-1) \tilde{q} + g(n-2) \tilde{q}^2 \bigr) + O(r) \bigl( 1 + |\tilde{q}| + \tilde{q}^2 + |\tilde{q}|^3 \bigr).
\end{equation}
coming from the relation~\eqref{eqn:q-ODE-f-infty} for $q$.
Since $q < 0$, we have $\tilde{q} < 0$ as well; moreover, for $|\tilde{q}| > C(M)$ and $r \in (0,\ve)$, the cubic term $g(n-2) \tilde{q}^3 = - g(n-2) |\tilde{q}|^3$ dominates in the right-hand side, making $r \tilde{q}' < 0$ for $r \in (0,\ve)$.
Since $r=0$ corresponds to $t \uparrow t_{\alpha}$, this would make $|\tilde{q}(r)| = | (t_{\alpha} - t) q(t) |$ decreasing on approach to $t \uparrow t_{\alpha}$, so the limit $\beta := - \lim_{t \uparrow t_{\alpha}} |(t_{\alpha} - t) q(t)| =  \lim_{r \downarrow 0} \tilde{q}(r)$ exists.
Taking $r \downarrow 0$ sufficiently small in~\eqref{eqn:q-tilde-equation}, we find $\tilde{q}' = \beta(1 + (n-1) \beta + g(n-2) \beta^2 ) r^{-1} + O(1)$ for small $r>0$.
Since the limit $\lim_{r \downarrow 0} \tilde{q}(r)$ exists, the coefficient of $r^{-1}$ must be zero, so
\[
1 + (n-1) \beta + g(n-2) \beta^2 = 0 \implies \beta = - \tfrac{(n-1) \pm \sqrt{(n-1)^2 - 4g(n-2)}}{2g(n-2)} \qquad \text{or} \qquad \beta = 0.
\]
For $(n-1)^2 < 4g(n-2)$, the discriminant of the quadratic is negative, so the first two possibilities cannot occur; we conclude that $\beta = 0$, so $(t_{\alpha} - t) q \to 0$ and $A_M(t) q \to 0$ as $t \uparrow t_{\alpha}$.
If the limit $\lim_{t \uparrow t_{\alpha}} q$ was finite, then writing $\log f_{\infty}(t_{\alpha}) = \log f_{\infty}(t_1) + \int_{t_1}^{t_{\alpha}} q(s) \, ds$ would imply that $| \log f_{\infty} (t_{\alpha})|$ is finite, contradicting $f_{\infty}(t_{\alpha}) = 0$.
Therefore, it suffices to consider $|q| \to \infty$ and take $|q|$ sufficiently large due to~\eqref{eqn:q-ODE-f-infty}.
Applying these considerations upon division by $q^2$, we arrive at
\[
(1-t^2) ( q^{-2} q' + 1) + (n-2) (1-t^2) + O(q^{-1}) = 0, \quad \implies \quad (q^{-1})' \to  (n-1) \quad \text{as } \; |q| \to + \infty.
\]
Integrating this relation on $[ t , t_\alpha]$, we find $q^{-1}(t) \geq \frac{n-1}{4}(t_\alpha - t)$, contradicting $\lim_{t \uparrow t_\alpha} (t_\alpha - t) q(t) = 0$. 
We conclude that $c_0 := f_\infty(t_\alpha) > 0$, so $f_a (t_\alpha) \geq c_0 a$ due to~\eqref{eqn:lower-bound-at-ta} as desired.
\end{proof}

We now combine the above ingredients to prove that large initial heights lead to a blow-up.

\begin{proposition}\label{prop:dont-shoot-too-high}
    There exists a constant $C(M)$, depending only on the hypersurface $M \subset \bS^{n-1}$, such that any solution of equation~\eqref{eqn:ode-star} with initial data
    \[
    \{ f(0) = a, \; f'(0) = 0 \} \qquad \text{or} \qquad \Bigl\{ f \bigl( \sqrt{\tfrac{m_1}{m_1+ m_2}} \bigr) = a, \; f' \bigl( \sqrt{\tfrac{m_1}{m_1+ m_2}} \bigr) \leq 0 \Bigr\} ,
    \]
    where $a > C(M)$, has finite-time derivative blowup while positive.
\end{proposition}
\begin{proof}
If $g \in \{ 1, 2\}$, then $M \subset \bS^{n-1}$ is a sphere or a Clifford hypersurface and the property is proved in~\cite{FTW-1}*{Proposition 3.16} (for shots starting from $t=0$) and~\cite{new-minimal-surfaces}*{Lemma 2.4} (for shots starting from $t = \sqrt{\frac{m_1}{m_1 + m_2}}$); these results show that we can take $C(M) = \frac{2}{\sqrt{n}}$ in this case.
We therefore focus on $g \geq 3$ in what follows and we denote $\mu := m_1 + m_2$ as in Proposition~\ref{prop:Psi-one-sided} in what follows.
Recall that $\Psi'( \sqrt{\frac{m_1}{\mu}}) = - (g-2) (ff') ( \sqrt{\frac{m_1}{\mu}})$ by the computation therein, and $f'( \sqrt{\frac{m_1}{\mu}}) \leq 0$ holds in both cases under consideration.
For solutions starting from $\{ f(0) = a, f'(0) = 0\}$, this property is ensured by Lemma~\ref{lemma:general-properties-of-solutions}; thus, $\Psi'( \sqrt{\frac{m_1}{\mu}}) \geq 0$ in all cases.
Moreover, $\Psi \bigl(\sqrt{\frac{m_1}{\mu}} \bigr) = f \bigl( \sqrt{\frac{m_1}{\mu}} \bigr)^2 - \frac{1}{n-2}$, so the blowup Lemma~\ref{lemma:psi-detect-blowup} will yield the desired result upon ensuring that $f \bigl(\sqrt{\frac{m_1}{\mu}} \bigr) > \frac{1}{\sqrt{n-2}}$ for $a$ sufficiently large; this property is automatic in the second case, i.e., for solutions starting from $\sqrt{\frac{m_1}{\mu}}$.
To prove that $f( \sqrt{\frac{m_1}{\mu}}) > \frac{1}{\sqrt{n-2}}$ when $\{ f(0) = a, f'(0) = 0 \}$ and $a$ is large, we proceed in two steps.
First, we argue by contradiction and study this inequality by a blowup argument, as $a \uparrow \infty$; then, we examine the asymptotic behavior of the limit quantity and obtain a contradiction.

\smallskip \noindent \textbf{Step 1:}
To prove the uniform height bound on solutions with $\{ f(0) = a, f'(0) = 0\}$, we argue by contradiction, supposing that there exist $a_j \uparrow \infty$ for which the corresponding solutions $f_j$ with $f_j(0) = a_j$ reach a zero.
Then, $f'_j < 0$ on $[0, \sqrt{\frac{m_1}{\mu}}]$ and $h_j := f_j - \frac{g}{2} A_M f'_j > 0$ by Lemma~\ref{lemma:general-properties-of-solutions}, while $\limsup_j f_j ( \sqrt{\frac{m_1}{\mu}}) \leq \frac{1}{\sqrt{n-2}}$ by the above argument, else $\Psi_j ( \sqrt{\frac{m_1}{\mu}}) \geq 0$ and $\Psi'_j( \sqrt{\frac{m_1}{\mu}}) > 0$ would force $\Psi_j$ to remain positive for all time, prohibiting $f_j$ from reaching a zero.
We may therefore extract a subsequence $f_{j'}$ with $f_{j'} ( \sqrt{\frac{m_1}{\mu}}) \to L \in [0, \frac{1}{\sqrt{n-2}}]$.
Since the functions $f_j$ are concave on $[0, \sqrt{\frac{m_1}{\mu}}]$, the secant line bound implies that $f'_j( \sqrt{\frac{m_1}{\mu}}) \leq - \sqrt{\frac{\mu}{m_1}} (a_j - L)$, so in particular $\epsilon_j := \frac{1}{|f'_j(\sqrt{m_1/\mu})|} \to 0$.
Consequently, the conditions of Proposition~\ref{prop:i-will-take-a-limit} are satisfied, whereby letting $y_j(r) := f_j \bigl( \sqrt{\frac{m_1}{\mu}} + \epsilon_j r \bigr)$ allows us to extract a limit $y_{\infty}$ of the $\{ y_j \}$, which is a solution of equation~\eqref{eqn:phi-limit-equation}.
The argument of $f_j$ is defined for all $\sqrt{\frac{m_1}{\mu}} + \epsilon_j r \in [0 , \sqrt{\frac{m_1}{\mu}}]$, so for all $r \geq - \epsilon_j^{-1} \sqrt{\frac{m_1}{\mu}}$, and thus $y_{\infty}$ is defined for all backwards time $r \in (-\infty,0]$ and $y_j \xrightarrow{C^{\infty}} y_{\infty}$ on compact sub-intervals.
By construction, 
\[
y_j(0) = f_j \bigl( \sqrt{\tfrac{m_1}{\mu}} \bigr) \to L \quad \text{as } \; j \to \infty, \qquad y'_j(0) = \epsilon_j f'_j \bigl( \sqrt{\tfrac{m_1}{\mu}} \bigr) = - \frac{1}{f'_j( \sqrt{m_1/\mu})} f'_j \bigl( \sqrt{\tfrac{m_1}{\mu}} \bigr) = -1,
\]
which implies that $y_{\infty}(0) = L$ and $y'_{\infty}(0) = -1$.
In what follows, we denote $y := y_{\infty}$ for brevity.
Since $h_j \to h := y - gr y'$ by the computation~\eqref{eqn:hi-term-to-h-infty} of Proposition~\ref{prop:i-will-take-a-limit}, we form the quantity 
\begin{equation}\label{eqn:Psi-bar-quantity}
\bar{\Psi} := y(y- gr y') - \tfrac{1}{n-2}.
\end{equation}
Computing similarly to Proposition~\ref{prop:Psi-one-sided}, we find
\begin{align*}
\bar{\Psi}' &= (2-g) y y' - gr \bigl( (y')^2 + yy'' \bigr) = y^{-1} y' \cdot \left[ (2-g) y^2  - gr yy' \Bigl( 1 - (n-2) \frac{y ( y - gry')}{1+y^2} \Bigr) \right]
\end{align*}
by substituting $y''$ from~\eqref{eqn:phi-limit-equation}.
At a zero $r_* \leq 0$ of $\bar{\Psi}$, we have $(n-2) (y - gry') = \frac{1}{y}$, so
\[
\bar{\Psi}'(r_*) = (2-g) y y' - gr_* (y')^2 \tfrac{y^2}{1+y^2}.
\]
By Proposition~\ref{prop:i-will-take-a-limit}, the function $y$ is strictly decreasing due to $y'(0) = -1$, and $y>0$ for $r < 0$ due to $y(0) = L$.
Therefore, $g \geq 3$ together with $\{ y>0, y'<0, r_* \leq 0 \}$ makes both summands above non-negative, so $\bar{\Psi}'(r_*) > 0$ at any zero of $\bar{\Psi}$.
This shows that $\bar{\Psi}$ can have at most one sign change $r_* \in (-\infty,0]$, with $\bar{\Psi} > 0$ on $(r_*, 0]$.

\smallskip \noindent \textbf{Step 2:}
We now prove that $\bar{\Psi}(r) > 0$ as $r \to - \infty$, so $\bar{\Psi} > 0$ on $(-\infty,0]$ will follow from the above argument.
By the second part of Proposition~\ref{prop:i-will-take-a-limit}, the function $y$ is strictly decreasing, and therefore invertible, with inverse $r(y) = C_0 F_0 + C_1 F_1$ for $F_i$ the functions of~\eqref{eqn:r(y)-inverse} in order of appearance, with $F_0$ even and $F_1$ odd.
The conditions for $y_{\infty}$ yield $r(L) = 0$ and $r'(L) = \frac{1}{y'_{\infty}(0)} = -1$, so
\[
C_0 F_0(L) + C_1 F_1(L) = 0 , \quad C_0 F'_0(L) + C_1 F'_1(L) = - 1, \quad \implies C_i = (-1)^i \tfrac{F_i(L)}{(F_0 F'_1 - F_1 F'_0)(L)} ,
\]
by solving the $2 \times 2$ system for $C_0, C_1$.
The functions $F_0, F_1$ solve equation~\eqref{eqn:r-ode-in-terms-of-y}, so their Wronskian satisfies $(F_0 F'_1 - F_1 F'_0)' = \frac{cy}{1+y^2} (F_0 F'_1 - F_1 F'_0)$.
This implies that $F_0 F'_1 - F_1 F'_0 = C (1+y^2)^{\frac{c}{2}}$, where we evaluate $C=1$ by computing at $y=0$.
We therefore obtain
\begin{equation}\label{eqn:r(y)-expression-final-inverse}
    r(y) = (1+L^2)^{- \frac{c}{2}} \bigl[ F_1(L) F_0(y) - F_0(L) F_1(y) \bigr] \, .
\end{equation}
Note that $(n-1)^2 - 4g(n-2) > 0$ in our situation, meaning that the roots $a_{\pm}$ from~\eqref{eqn:r(y)-inverse} are real.
Indeed, if $(n-1)^2 < 4g(n-2)$, then Lemma~\ref{lemma:f-infty-analysis} would imply that $f_j (\sqrt{\frac{m_1}{\mu}}) \to \infty$ as $a_j \to \infty$, contradicting our assumption; thus, $(n-1)^2 \geq 4g(n-2)$.
Examining the list of isoparametric triples $(g,m_1,m_2)$ with $g(m_1 + m_2) = 2(n-2)$ and $g \in \{3,4, 6\}$ shows that $(n-1)^2 - 4g(n-2)$ is never zero, and it cannot be a square unless $(g, m_1 + m_2) \in \{ (3,8) , (4,7) \}$.
In all such cases, $n = \frac{g(m_1+m_2)}{2} + 2$ is even, so the resulting square is odd and $a_+ - a_- = \frac{\sqrt{(n-1)^2 - 4g(n-2)}}{2} \not\in \bZ$, meaning that the growth rates of the hypergeometric functions $F_i$ are non-resonant and correspond to $r(y) \sim C_{\pm} y^{- 2 a_{\pm}}$ as $y \to \infty$.
This behavior comes from the hypergeometric functions~\eqref{eqn:r(y)-inverse} solving equation~\eqref{eqn:r-ode-in-terms-of-y} with $c = n-2$, which in the large-$y$ regime behaves as the ODE
\[
y^2 r'' - (n-2) y r' + g(n-2)  r = 0.
\]
This is an Euler equation with the homogeneous solutions $r(y) \sim C_{\pm} y^{- 2a_\pm}$ obtained above, where $a_{\pm} = \frac{- (n-1) \pm \sqrt{(n-1)^2 - 4 g(n-2)}}{4}$.
The functions $y_j(r)$ considered in the limit procedure attain arbitrarily large values at arbitrarily far backwards times, since $y_j( - \epsilon_j^{-1} \sqrt{\frac{m_1}{\mu}}) = f_j(0) = a_j \to \infty$.
Therefore, $y$ is unbounded as $r \to - \infty$, and the above asymptotic analysis applies to yield 
\[
y(r) \sim C(-r)^{\beta}, \qquad \beta \in \{ \beta_+, \beta_-\}, \quad \beta_{\pm} = - \tfrac{1}{2 a_{\pm}} = \tfrac{(n-1) \pm \sqrt{(n-1)^2 - 4g(n-2)}}{2g (n-2)}, \qquad C \neq 0.
\]
Returning to the quantity $\bar{\Psi}$ considered in~\eqref{eqn:Psi-bar-quantity}, we find that $y'(r) \sim - C \beta (-r)^{\beta-1}$, so
\[
y - gr y' \sim C(1-g\beta) (-r)^{\beta} \qquad \text{and} \qquad \bar{\Psi} \sim C^2(1-g \beta)(-r)^{2 \beta} - \tfrac{1}{n-2},
\]
as $r \to - \infty$.
Crucially, we observe that $1 - g \beta \geq 1 - g \beta_+ > 0$, since
\[
(n-1) + \sqrt{(n-1)^2- 4g(n-2)} < (n-1) + (n-3) = 2(n-2),
\]
which means that $\beta_- < \beta_+ = \frac{(n-1) + \sqrt{(n-1)^2- 4g(n-2)}}{2g(n-2)} < \frac{1}{g}$.
Moreover, $\beta \geq \beta_- > 0$, whereby
\[
\bar{\Psi} \sim C^2 (1 - g \beta) (-r)^{2 \beta} - \tfrac{1}{n-2} \geq C^2(1 - g \beta_+) (-r)^{2 \beta_-} - \tfrac{1}{n-2} \to \infty
\]
as $r \to - \infty$.
Therefore, $\bar{\Psi} > 0$ for $r \to - \infty$, so $\bar{\Psi} > 0$ on $(-\infty,0]$ by the second step of our proof.
In particular, $\bar{\Psi}(0) = y(0)^2 - \frac{1}{n-2} = L^2 - \frac{1}{n-2} > 0$, contradicting the assumption that $L \in [ 0, \frac{1}{\sqrt{n-2}}]$.

The contradiction to the above assumption shows that the solutions $\{ f(0) = a, f'(0)= 0 \}$ have derivative blowup while positive if $a \geq \bar{C}(M)$, for some constant $\bar{C}(M)$.
Finally, we can take $C(M) = \bar{C}(M) + \frac{1}{\sqrt{n-2}}$ to cover shots from both $t=0$ and $t = \sqrt{\frac{m_1}{\mu}}$, completing the proof.
\end{proof}

\subsection{Type I}\label{section:type-I}

We prove the first part of Theorem~\ref{thm:capillary-interpolation}, exhibiting a smooth family of capillary minimal cones $\mathbf{C}_{M, \theta} \subset \bR^{n+1}_+$ for every angle $\theta \in (0 , \frac{\pi}{2}]$, whose links $\Sigma_{M, \theta}$ have the sphere bundle topology described in Section~\ref{section:topology}.
Following the discussion of Section~\ref{sec:minimal-surface-equation}, this result will be equivalent to proving the existence of a smooth family of solutions $\{ f_a \}$ to equation~\eqref{eqn:ode-star} for which the map $a \mapsto (1 - t_a^2) \, f'_a(t_a)^2$ is surjective onto $(0, + \infty]$, where $t_a$ is the zero of $f_a$ and $(1-t_a^2) f'_a(t_a)^2 = \tfrac{4}{g^2} \tan^2 \theta$ encodes the capillary angle $\theta$.
\begin{proof}[Proof of Theorem~\ref{thm:smooth-interpolation}]
    For $a > 0$, let $f_a$ denote the solution of equation~\eqref{eqn:ode-star} with initial data $\{ f_a(0) = a, f'_a(0) = 0 \}$.
    For a solution $f_a$ with a zero at $t_a$, Proposition~\ref{prop:mean-curvature-equation} shows that $\mathbf{C}_{M_i,a} := \textup{graph} \bigl( \rho f_a( \sin \frac{g s(\omega)}{2}) \bigr)$ is a capillary minimal cone in $\bR^{n+1}_+$ with angle $\theta(a) = \arctan ( \frac{g}{2} \sqrt{1-t_a^2} \, |f'_a(t_a)| )$.
    The smooth dependence theorem for ODE solutions implies that the map $a \mapsto \sqrt{1 - t_a^2} \, |f'_a(t_a)|$ is smooth while finite, so $a \mapsto \theta(a)$ is smooth while $\theta(a) \neq \frac{\pi}{2}$.

    For small $a > 0$, the rescaling $\frac{1}{a} f_a$ produces a solution of equation~\eqref{eqn:rescaled-equation-lambda-ODE} with parameter $\lambda = a^2 \downarrow 0$.
    As in the discussion of~\cites{FTW-1 , new-minimal-surfaces }, the rescaled functions converge uniformly in $C^{\infty}_{\textup{loc}}$ to the solution $f_0$ of $\cL_M f = 0$ with initial data $\{ f_0(0) = 1, f'_0(0) = 0 \}$, given by the function $f_M(t)$ of~\eqref{eqn:hypergeometric-isoparametric}; this proves Theorem~\ref{thm:smooth-interpolation}.
    On the other hand, Proposition~\ref{prop:dont-shoot-too-high} shows that if $a > C(M)$, then the solution $f_a$ has finite-time blowup before reaching zero.
    Therefore, we can obtain some
    \[
    a_{M_i}^* := \sup \{ a > 0 : f_a \; \text{ reaches a zero} \} < C(M) < \infty.
    \]
    A standard continuity argument as in the proof of Lemma~\ref{lemma:convergence-of-solutions} shows that the solution $f_{a_{M_i}^*}$ reaches zero at a point $t_{a_{M_i}^*}$ where $f_{a_{M_i}^*}(t_{a_{M_i}^*}) = 0$ and $f'_{a_{M_i}^*}(t_{a_{M_i}^*}) = - \infty$.
    Consequently, each $\mathbf{C}_{M,a}$ for $a \leq a_{M_i}^*$ is a capillary minimal cone, with the property that the rescaled cones $\frac{1}{a} \mathbf{C}_{M,a}$ converge to $\textup{graph} ( \rho f_0(t))$ and their contact angles satisfy $\theta(a) \asymp a$ for small $a$.
    The association $a \mapsto f_a$ is $C^{\infty}_{\textup{loc}}(0,a_{M_i}^*)$, so the map $a \mapsto \mathbf{C}_{M,a}$ produces a smooth family of cones.
    Since $\theta(a_{M_i}^*) = \frac{\pi}{2}$, we conclude that the map $a \mapsto \theta(a)$ is surjective on $(0,\frac{\pi}{2}]$, so the cones $\mathbf{C}_{M,a}$ produce all capillary angles. 
    Moreover, in an interval $\theta \in ( \frac{\pi}{2} - \ve, \frac{\pi}{2}]$, the cones can be labelled $\mathbf{C}_{M,\theta}$ by their contact angle $\theta$, in such a way that $\theta \mapsto \mathbf{C}_{M,\theta}$ is a smooth map and $\mathbf{C}_{M,\theta} \to \mathbf{C}_{M, \frac{\pi}{2}}$ in $C^{\infty}_{\textup{loc}}$, as desired.
\end{proof}

\subsection{Type II}
Finally, we complete the proof of Theorem~\ref{thm:capillary-interpolation} by constructing the capillary cones $\bar{\mathbf{C}}_{M,\theta}$ of type II.
As in~\cite{new-minimal-surfaces}, the key step is to produce a small window near $\sqrt{\frac{m_1}{m_1 + m_2}}$, from which no solutions, \textit{shots}, reach a second zero.

\begin{proposition}\label{prop:small-breather}
There exists an $\bar{\ve}_M>0$ such that for any $t_1 \geq \sqrt{\frac{m_1}{m_1 + m_2}} - \bar{\ve}_M$ and $a>0$, the solution of equation~\eqref{eqn:rescaled-equation-lambda-ODE} with initial data $\{ f(t_1) = 0, f'(t_1) = a \}$ does not reach a second zero.    
\end{proposition}
\begin{proof}
    We know from Lemma~\ref{lemma:general-properties-of-solutions} that for any $\lambda, a \in [0,\infty]$, the solution of equation~\eqref{eqn:rescaled-equation-lambda-ODE} with initial data $\{ f(t_1) =0, f'(t_1) = a \}$ starting from $t_1 = \sqrt{\frac{m_1}{m_1+m_2}}$ remains strictly increasing while defined.
    The reduction strategy of~\cite{new-minimal-surfaces}*{Proposition 2.8} therefore transfers directly, upon applying the compactness results of Lemmas~\ref{lemma:slope-infinity-solution} and~\ref{lemma:convergence-of-solutions}; this shows that we have such a uniform $\bar{\ve}_M > 0$ if and only if there does not exist a sequence of $\epsilon_i \downarrow 0$ such that the solutions $f_i := f_{t_{\alpha} - \epsilon_i}$ with initial data $\{ f(t_{\alpha} - \epsilon_i) = 0, f'(t_{\alpha} - \epsilon_i) = + \infty\}$ reach a second zero, where $t_{\alpha} := \sqrt{\frac{m_1}{m_1 + m_2}}$ for brevity.
    
    Suppose, for contradiction, that such a sequence $\epsilon_i \downarrow 0$ exists, so Lemma~\hyperref[lemma:general-properties-of-solutions]{\ref{lemma:general-properties-of-solutions}$(i)$} implies that $h_i = f_i - \frac{g}{2} A_M f'_i > 0$, while Proposition~\ref{prop:dont-shoot-too-high} forces $\limsup_i f_i(t_{\alpha}) \leq C(M) < + \infty$.
    Note that Proposition~\ref{prop:dont-shoot-too-high} is applicable to this setting because we may assume, without loss of generality, that a subsequence of the $f_i$ has $f'_i( t_{\alpha}) \leq 0$: if we had $f'_i(t_{\alpha}) > 0$ for all but finitely many $i$, we would work with the functions $\hat{f}(t) := f(\sqrt{1-t^2})$ near the point $\hat{t}_{\alpha} := \sqrt{1 -t_{\alpha}^2} = \sqrt{\frac{m_2}{m_1 + m_2}}$, which satisfy equation~\eqref{eqn:ode-star} for the triple $(g,m_2, m_1)$ due to Lemma~\ref{lemma:change-of-variables}.
    Moreover, $\on{sgn} \hat{f}'_i(\hat{t}_{\alpha}) = - \on{sgn} f'_i(t_{\alpha})$, so applying Proposition~\ref{prop:dont-shoot-too-high} to the functions $\hat{f}_i$ proves $\limsup_i \hat{f}_i(\hat{t}_{\alpha}) \leq C(M)$ as needed.

    Therefore, the functions $f_i$ and scales $\epsilon_i$ satisfy the conditions of Proposition~\ref{prop:i-will-take-a-limit}, whereby letting $y_i(r) := f_i(t_{\alpha} + \epsilon_i r)$ allows us to extract a limit $y_{\infty}$ of the $\{ y_i \}$, which is a solution of equation~\eqref{eqn:phi-limit-equation}.
    Moreover, we have uniform bounds $\| y_i\|_{C^{k,\alpha}(-1+\delta,S)} \leq C(k,\alpha,\delta,S)$ and $y_i \to y_{\infty}$ in $C^{\infty}_{\textup{loc}}(-1+\delta,S)$ for any $\delta, S>0$, by the convergence and compactness arguments of~\cite{new-minimal-surfaces}*{Lemmas 2.5 \& 2.6}, which correspond to Lemmas~\ref{lemma:slope-infinity-solution} and~\ref{lemma:convergence-of-solutions}.
    At $r=-1$, the initial conditions for $f_i$ yield $\{ y_i(-1) = 0, y'_i(-1) = \infty \}$; indeed, expressing $A(t_{\alpha} + \epsilon_i r) = 2 \epsilon_i r + O (\epsilon_i^2 r)$ in the computations~\eqref{eqn:f-prime-t,f-t}, we obtain the expansions for $f_i, f'_i$ as
    \[
    f_i(t) = \sqrt{ \tfrac{2}{(n-2) g \epsilon_i }} \, |t - t_i(-1)|^{\frac{1}{2}} + O ( |t - t_i(-1)|), \qquad f'_i(t) = \sqrt{\tfrac{1}{2(n-2) g \epsilon_i}} \, |t - t_i(-1)|^{- \frac{1}{2}} + O(1),
    \]
    for $0 < t - (t_{\alpha} - \epsilon_i) \ll \epsilon_i$.
For $y_i(r) = f_i(t_i(r))$ with $t = t_i(r)$ and $r \in (-1,-1+\delta]$, this transformation therefore produces $y_i(r) = \sqrt{\frac{2(r+1)}{g(n-2)}} + O (r+1)$ and $y'_i(r) = \sqrt{\frac{1}{2(n-2) g (r+1)}}  + O(1)$.
We deduce that the limit solution $y_i \xrightarrow{C^{\infty}_{\textup{loc}}} y_{\infty}$ has vertical initial slope at $r=-1$, and can be obtained via the local series expansion $y_{\infty}(r) = \sqrt{\frac{2(r+1)}{g(n-2)}} + O(r+1)$ near $r = -1$, as in~\cite{new-minimal-surfaces}*{Lemma 2.6}.
By the considerations of Lemmas~\ref{lemma:slope-infinity-solution} and~\ref{lemma:convergence-of-solutions}, the function $y_{\infty}$ agrees with the $C^{\infty}_{\textup{loc}}$ subsequential limit of solutions $y^{(N)}$ to~\eqref{eqn:phi-limit-equation} with initial data $\{ y^{(N)}(-1) = 0, y^{(N)}{}'(-1) = N \}$ as $N \to \infty$.
We conclude that $y_{\infty}$ is the maximally extended solution of the initial value problem for~\eqref{eqn:phi-limit-equation} with data $\{ y(-1) = 0, y'(-1) = \infty\}$, understood in the limit sense of Lemma~\ref{lemma:slope-infinity-solution}.
Moreover, the functions $h_i(r)$ are computed in terms of $y_i(r)$ in~\eqref{eqn:hi-term-to-h-infty} and satisfy $h_i \to h_{\infty} := y_{\infty} - gr y'_{\infty}$ in $C^{\infty}_{\textup{loc}}$; the assumption that $h_i > 0$ implies that $h_{\infty} \geq 0$ is everywhere non-negative.

Finally, we obtain a contradiction by showing that $h_{\infty}$ must become negative at some $r_* > -1$ for the limit solution $y_{\infty}$.
We suppose that this is not the case and argue as in~\cite{new-minimal-surfaces}*{Lemma 2.7}; for notational simplicity, we write $(y,h,c) := (y_{\infty}, h_{\infty},n-2)$ in what follows.
By the second part of Proposition~\ref{prop:i-will-take-a-limit}, the solution $y' > 0$ is strictly increasing while defined, and we see that $h$ satisfies
\[
h'(r) = cg  \,r \, \tfrac{y'(r)^2}{1 + y(r)^2} h(r) - (g-1) y'(r)
\]
by analogy with the relation~\eqref{eqn:2.6NMS}.
At a point $r_*$ where $h(r_*) = 0$, this relation would imply $h'(r_*) = - (g-1) y'(r_*) < 0$, so $h(r_* + \delta) < 0$ for small $\delta > 0$, contradicting $h \geq 0$; therefore, $h > 0$ for all $r \geq -1$.
The equation~\eqref{eqn:phi-limit-equation} becomes $y'' = - c \frac{(y')^2}{1+y^2} h < 0$, so $y'$ is strictly decreasing and positive for all time $r \in (-1,S_b)$, where $S_b \leq \infty$ denotes the maximal existence time of the solution $y$. 
This implies that $y$ cannot have finite-time derivative blowup at any $S_b < \infty$, so $S_b = \infty$ and $y \in C^{\infty}(-1,\infty) \cap C^{\frac{1}{2}} ([-1,\infty))$ is defined for all forward time.
Moreover, $\hat{\ell} := \lim_{r \to \infty} y'(r)$ exists and is finite, with $\hat{\ell} \geq 0$.
In fact, $\hat{\ell} > 0$ is impossible, else $y'' < 0$ would produce some $S_0$ with $\hat{\ell} < y' < \frac{3}{2} \hat{\ell}$ for $r \geq S_0$, whereby
\[
h(r) = y(r) - gr y'(r) \leq \tfrac{3}{2} \hat{\ell}(r - S_0) + y(S_0) - g \hat{\ell} r < y(S_0) - \tfrac{2g-3}{2} \hat{\ell} r \,,
\]
which becomes negative for $g \geq 2$ and $r > S_0 + 4 \hat{\ell}^{-1} y(S_0)$, contradicting $h(r)>0$.
Thus, $\lim_{r \to \infty} y'(r) = 0$.
Next, using $y' >0$ and $\frac{y}{1+y^2} \leq \frac{1}{2}$ in equation~\eqref{eqn:phi-limit-equation}, we find $y'' \geq - \frac{c}{2} (y')^2$, so
\[
\bigl( \tfrac{1}{y'} \bigr)' \leq \tfrac{c}{2} \implies \tfrac{1}{y'(r)} \leq \tfrac{c}{2} (r + B_1) \implies y(r) \geq \tfrac{2}{c} \log(r+B_1) + B_0.
\]
In particular, $y \to \infty$ is unbounded as $r \to \infty$.
Since $y(-1) = 0$, we see that $y$ defines a monotone bijection $[-1,\infty) \to [0,\infty)$, so $\tau(r) := \on{arcsinh} y(r)$ is a strictly increasing map with image $[0,\infty)$ and $\tau'(r) = \frac{y'}{\sqrt{1 +y^2}} = \frac{y'}{\cosh \tau(r)}$.
We define the quantity
\begin{equation}\label{eqn:L(tau)-function}
    L(\tau) := \frac{\sqrt{1 + y(r)^2}}{h(r)} = \frac{\cosh \tau}{y - gr y'} = \frac{1}{\tanh \tau - g r \tau'(r)}, \qquad \text{for } \; r(\tau) = y^{-1}( \sinh \tau).
\end{equation}
Given the properties $h > 0$ and $h(-1) = \infty$, the function $\frac{\sqrt{1 + y(r)^2}}{h(r)}$ is well-defined for all $r \in [-1,\infty)$, so $L$ is defined for all $\tau \in [0,\infty)$ with $L(0) = 0$.
Moreover, $L(\tau)$ satisfies the equation
\begin{equation}\label{eqn:L(tau)-Riccati-equation}
    \cP[L] := L'(\tau) - (g-1) L(\tau)^2 + (c-1) \tanh \tau \, L(\tau) - c = 0 , \qquad L(0) = 0.
\end{equation}
To see this, note that $y' = \tau'(r) \cosh \tau$ and $y'' = \tau'' \cosh \tau + (\tau')^2 \sinh \tau$, so the equation~\eqref{eqn:phi-limit-equation} becomes
\[
\cosh \tau \bigl( \tau'' + (\tau')^2 \tanh \tau \bigr) + c (\tau')^2 \bigl( \sinh \tau - gr \cosh \tau \, \tau' \bigr) = 0.
\]
On the other hand, the definition of $L$ in~\eqref{eqn:L(tau)-function} implies that $g r \tau' = \tanh \tau - L^{-1}$, and differentiating this relation allows us to eliminate $\tau''$ from the above equation and obtain~\eqref{eqn:L(tau)-Riccati-equation}.

Observe that the function $c \tanh \tau$ is a Riccati subsolution for the operator $\cP$, with $\cP [ c \tanh \tau] = - c \bigl((g-2)c+2\bigr)\tanh^2 \tau < 0$, so $L(\tau) > c \tanh \tau$ for $\tau>0$.
We may then write $(c-1) \tanh \tau \leq \frac{c-1}{c} L$, so the differential equation~\eqref{eqn:L(tau)-Riccati-equation} leads to
\begin{align*}
& L'(\tau) = L(\tau) \bigl( (g-1) L(\tau) - (c-1) \tanh \tau \bigr) + c \geq c^{-1} L(\tau)^2 + c,
\end{align*}
due to $g \geq 2$.
The equation $\underline{L}' = c^{-1} \underline{L}^2 + c$ has solution $\underline{L}(\tau) = c \tan \tau$, so $L(\tau) \uparrow \infty$ as $\tau \uparrow \tau_*$ also blows up in finite time $\tau_* \leq \frac{\pi}{2}$, leading to a contradiction. 

Therefore, $h = y - gr y'$ must cross zero, thus it becomes negative at some $r_* > -1$, and so do all the functions $h_i = f_i - \frac{g}{2} A_M f'_i$ for $\epsilon_i$ sufficiently small.
Then, parts \hyperref[lemma:general-properties-of-solutions]{$(i) , (iii)$} of Lemma~\ref{lemma:general-properties-of-solutions} show that the solutions $f_i$ cannot reach a second zero past $t_{\alpha}$, so a sequence of $\epsilon_i \downarrow 0$ as above cannot exist.
We conclude that $f_{t_1}$ cannot reach a second zero if $t_1 > t_{\alpha} - \bar{\ve}_M$, as claimed.
\end{proof}

\begin{proposition}\label{prop:window-near-zero}
    There exists some $\ve_M$ such that for any $t_1 \in (0,\ve_M)$ and $a \in (0,\infty]$, the solution $f$ to equation~\eqref{eqn:rescaled-equation-lambda-ODE} with initial data $\{f(t_1) = 0,f'(t_1) = a\}$ reaches a second zero $t_2$ in its maximal domain of existence, with $(1-t_2^2) f'(t_2)^2 < \frac{1}{2} a^2$.
\end{proposition}
\begin{proof}
We study solutions of equations~\eqref{eqn:ode-star} and~\eqref{eqn:rescaled-equation-lambda-ODE} starting from a point $t_1$ with initial data $(f(t_1), f'(t_1))$, which we denote by $f_{(f(t_1),f'(t_1))}$.
For any $\alpha > 0$, we consider the property
\begin{equation}\label{eqn:original-assumption}\tag{S}
\begin{split}
    & \exists \; \delta = \delta(\alpha) > 0 : t_1 \in (0, \delta), \quad (f(t_1), f'(t_1)) \in (0,\delta) \times (0,\delta), \\
    & \implies f_{(f(t_1),f'(t_1))} \; \text{ has a zero } \; t_2 > \sqrt{\tfrac{m_1}{m_1+m_2}} \; \text{ with } \; (1-t_2^2) f'(t_2)^2 < \alpha.
\end{split}
\end{equation}
This property is obtained by applying a compactness argument as in Step 1 of~\cite{new-minimal-surfaces}*{Proposition 2.9}.
Indeed,~\cite{one-phase-isoparametric}*{Proposition 3.5} shows the linear problem $\cL_M \bar{f} = 0$ has the property
\begin{equation}\label{eqn:linear-problem-property}
    \exists \; \delta > 0 \; : \; (\bar{f}(t_1), \bar{f}'(t_1)) \in (0,2) \times (0,2) \implies \bar{f}(t_2) = 0, \; \sqrt{1-t_2^2} \, |\bar{f}'(t_2)| < C.
\end{equation}
Applying the smooth dependence theorem for solutions of ODE in terms of the parameter $\lambda$ in equation~\eqref{eqn:rescaled-equation-lambda-ODE}, we find some small $\lambda_M > 0$ such that the property~\eqref{eqn:linear-problem-property} holds for solutions $f_{\lambda}$ of equation~\eqref{eqn:rescaled-equation-lambda-ODE} with any $\lambda \in [0,\lambda_M)$.
Finally, $f = \sqrt{\lambda} f_\lambda$ solves equation~\eqref{eqn:ode-star} with 
\[
(f(t_1), f'(t_1)) \in (0, 2 \sqrt{\lambda}) \times (0,2 \sqrt{\lambda}), \qquad f(t_2) = 0, \quad \sqrt{1-t_2^2} \, |f'(t_2)| < C \sqrt{\lambda},
\]
which implies the property~\eqref{eqn:original-assumption} upon taking $\lambda$ and $\delta$ sufficiently small.
Applying a further compactness argument in the parameters $\lambda, a$ as in~\cite{new-minimal-surfaces}*{Proposition 2.8}, using the results of Lemma~\ref{lemma:convergence-of-solutions}, we may reduce the existence of such an $\ve_M > 0$ to the case when $\lambda=1$ and $a = \infty$.
We will prove that any solution of equation~\eqref{eqn:ode-star} with $\{ f(t_1) = 0 , f'(t_1) = a \}$ enters the regime~\eqref{eqn:original-assumption} provided that $t_1 \in (0,\ve_M)$ is sufficiently small, which will imply our claim.
By the above-mentioned compactness argument with $a = \infty$, it suffices to show that the solution $f$ reaches zero with finite derivative; therefore, we can take $\delta = \delta(1)$ in~\eqref{eqn:original-assumption}, and $0 < \ve \ll \frac{\delta^2}{200 C(\delta)^2}$ where $C(\delta) := (1 + 10^3 \delta^{-2})^{\frac{1}{\beta}}$.

Taking $t_1 \in (0,\ve)$, in the region $\{ f \leq \frac{\delta}{10}, f' \leq \frac{\delta}{2}, t \leq 2C(\delta) \ve \}$, we bound 
\[
A_M(t) \leq - \tfrac{3}{4} \tfrac{m_1}{m_1 + m_2} t^{-1}, \qquad f - \tfrac{g}{2} A_M f' \geq \tfrac{3}{8} \tfrac{g m_1}{m_1 + m_2} t^{-1} f', \qquad |f'| < 1 + (f')^2. 
\]
Applying these properties in equation~\eqref{eqn:ode-star} and bounding $t f' \leq 4 C(\delta)^2 \ve^2 f' (1+ (f')^2) t^{-1}$, we obtain
\begin{align*}
& f'' + (n-2) ( 1 + (f')^2 \bigr) \tfrac{2}{7} \tfrac{g m_1}{m_1 + m_2} t^{-1} f' <0 , \quad \implies \quad f'' < - \tfrac{1}{7} g^2 m_1 t^{-1} f' (1 + (f')^2 ).
\end{align*}
Comparing the differential inequality for $p = f'$ with the equation $p' = - \beta t^{-1} p(1 + p^2)$ for $\beta = \frac{1}{7} g^2 m_1$, whose solution with initial data $p(t_1) = f'(t_1)$ is $p(t) = \bigl( (1+f'(t_1)^{-2}) ( \frac{t}{t_1} )^{2 \beta} - 1 \bigr)^{-\frac{1}{2}}$, we find
\begin{equation}\label{eqn:f-prime-bound}
f'(t) \leq \frac{1}{\sqrt{(1+f'(t_1)^{-2}) ( \frac{t}{t_1} )^{2 \beta} - 1}} \leq \frac{1}{\sqrt{( \frac{t}{t_1} )^{2 \beta} - 1}}
\end{equation}
Consequently, taking $\hat{t}_1 = C(\delta) t_1$ where $C(\delta) = (1 + 10^3 \delta^{-2})^{\frac{1}{\beta}}$ makes $f'(\hat{t}_1) < \frac{\delta}{10}$.
For $u \in [1, C(\delta)]$, the mean value theorem makes $u^{2 \beta} - 1 = ( 2 \beta) u_*^{2 \beta - 1}(u-1)$ with $u_* \in [1, u]$ and $2 \beta-1 \geq \frac{1}{7}$, so $u^{2 \beta} - 1 \geq u-1$. Integrating the bound~\eqref{eqn:f-prime-bound} on $[t_1, \hat{t}_1]$, using $f(t_1) = 0$ and $u = \frac{t}{t_1}$, we find
\[
f( \hat{t}_1) \leq t_1 \int_1^{C(\delta)} \frac{du}{\sqrt{u^{2 \beta}-1}} \leq t_1 \int_1^{C(\delta)} \frac{du}{\sqrt{u-1}} < 2 \sqrt{C(\delta)} \, t_1  .
\]
At the last step, we may take $t_1$ sufficiently small to conclude that at least one of 
\[
\{ f \leq \tfrac{\delta}{10},\ f' = 0 \}
\qquad\text{or}\qquad
\{ f \leq \tfrac{\delta}{10},\ f' \leq \tfrac{\delta}{10} \}
\]
must occur at some point $\hat{t}_1 < 2 C(\delta)t_1 < \delta$, where we used $\ve \ll \frac{\delta^2}{200 C(\delta)^2}$ and $\frac{1}{\sqrt{C(\delta)^{\beta}-1}} \leq \frac{\delta}{10}$.
Therefore, either $f'$ vanishes somewhere in $[t_1,C(\delta)t_1]$, in which case at the first such point $\hat{t}_1$ we have $f(\hat{t}_1)\le \frac{\delta}{100}$ and $f'(\hat{t}_1)=0$, or $f'>0$ on $[t_1,C(\delta)t_1]$ and taking $\hat{t}_1=C(\delta)t_1$ yields $f(\hat{t}_1) \leq \frac{\delta}{100}$ and $f'(\hat{t}_1) \leq \frac{\delta}{10}$; in either case $\hat{t}_1<2C(\delta)t_1<\delta$ follows from $\ve \ll \frac{\delta^2}{200 C(\delta)^2}$.

In both cases, this shows that $f$ satisfies the assumption~\eqref{eqn:original-assumption} of Step 1, therefore it reaches a zero $t_2 > \sqrt{\frac{m_1}{m_1 + m_2}}$ with finite derivative and $(1-t_2^2) f'(t_2)^2 < \frac{1}{2} a^2$ as desired.
\end{proof}

\begin{proof}[Proof of Theorem~\ref{thm:capillary-interpolation}, Part II]
We fix an angle $\theta \in (0,\frac{\pi}{2}]$ and let $f_{t_1}$ denote the solution of equation~\eqref{eqn:ode-star} with initial data $\bigl\{ f_{t_1}(t_1) = 0, f'_{t_1}(t_1) = \frac{2}{g} \frac{\tan \theta}{\sqrt{1 - t_1^2}} \bigr\}$.
Using Proposition~\ref{prop:window-near-zero}, we see that for $t_1 \in (0,\ve_M)$ sufficiently small, the solution $f_{t_1}$ reaches a second zero at $\tau(t_1) > \sqrt{\frac{m_1}{m_1+m_2}}$, where
    \begin{align*}
    R[t_1] &:= (1-t_1^2)f'_{t_1}(t_1)^2 - (1 - \tau(t_1)^2) f'_{t_1}(\tau(t_1))^2 \\
    &\;> \tfrac{4}{g^2} \tan^2 \theta - \tfrac{1}{2} \cdot \tfrac{4}{g^2} \cdot \tfrac{1}{1 - t_1^2} \tan^2 \theta > \tfrac{1}{g^2} \tan^2 \theta, \qquad t_1 < \tfrac{1}{2},
    \end{align*}
    so $R[t_1] > 0$ for $t_1 \in (0,\ve_M)$.
    On the other hand, taking $t_1 \uparrow \sqrt{\frac{m_1}{m_1 + m_2}} - \bar{\ve}_M$ makes the solution $f_{t_1}$ remain strictly positive until blowup, by Proposition~\ref{prop:small-breather}.
    Using the compactness result of Lemma~\ref{lemma:convergence-of-solutions}, we can therefore extract some $t_* \in ( \ve_M , \sqrt{\frac{m_1}{m_1 + m_2}} - \bar{\ve}_M)$ such that the solution $f_{t_*}$ of $\{ f_{t_*}(t_*) = 0 , f'_{t_*}(t_*) = \frac{2}{g} \frac{\tan \theta}{\sqrt{1 - t_*^2}} \}$ reaches another zero, at a point $\tau(t_*) > \sqrt{\frac{m_1}{m_1 + m_2}}$, with $f'_{t_*}(\tau(t_*)) = -\infty$.
    If $\theta = \frac{\pi}{2}$, then this discussion shows that $|f'(t)| \to \infty$ as $t \to t_*$ or $t \to \tau(t_*)$, so the solution $f_{t_*}$ is the desired profile curve of a free-boundary minimal cone $\bar{\mathbf{C}}_{M,\frac{\pi}{2}}$ in $\bR^{n+1}_+$. 
    If $\theta < \frac{\pi}{2}$, then the property of the $f_{t_*}$ makes $R[t_1] \to - \infty$ as $t_1 \uparrow t_*$, so there exists a $t_1(\theta) \in (\ve_M, \sqrt{\frac{m_1}{m_1+m_2}} - \bar{\ve}_M)$ such that $R[t_1(\theta)] = 0$ and $ \frac{g}{2} \sqrt{1 - t_1^2} \, |f'_{t_1}(t_1)| = \frac{g}{2} \sqrt{1 - \tau(t_1)^2} \, |f'_{t_1}(\tau(t_1))| = \tan \theta$. 
    Then, $\bar{\mathbf{C}}_{M,\theta} := \text{graph} ( \rho f_{t_1(\theta)})$ is the desired capillary cone with angle $\theta$ and the topology described in Section~\ref{section:topology}.
\end{proof}

\begin{proof}[Proof of Theorem~\ref{thm:smooth-interpolation}, Part II]
    The smooth interpolation result follows similarly to Theorem~\ref{thm:smooth-interpolation}, using an analogous shooting procedure with initial data prescribed along $t = \sqrt{\frac{m_1}{m_1 + m_2}} = \frac{1}{\sqrt{2}}$, which is a line of symmetry for equation~\eqref{eqn:ode-star} under the map $t \leftrightsquigarrow \sqrt{1-t^2}$.
    Namely, given a solution $f_a$ of the problem $\{ f(\frac{1}{\sqrt{2}}) = a, f'(\frac{1}{\sqrt{2}}) = 0 \}$ that reaches zero at a point $t_{2,a}$, we may extend $f_a$ to $[ \sqrt{1- t_{2,a}^2} , \frac{1}{\sqrt{2}}]$ by $\bar{f}_a( t) := f_a(\sqrt{1-t^2})$, where the extension is smooth due to $f'(\frac{1}{\sqrt{2}}) = 0$.
    Moreover, $\bar{f}(t_{1,a}) = 0$ at $t_{1,a} = \sqrt{1 - t_{2,a}^2}$, and Lemma~\ref{lemma:change-of-variables} shows that $\bar{f}_a$ solves equation~\eqref{eqn:ode-star} with $\sqrt{1-t_{1,a}^2} \bar{f}'_a(t_{1,a})^2 = \sqrt{1 - t_{2,a}^2} \bar{f}_a(t_{2,a})^2$, so $\bar{f}'_a$ satisfies the condition~\eqref{eqn:capillary-angle-condition} and produces a capillary profile curve.
    Thus, it suffices to study forward solutions $f_a$.

    Using~\cite{one-phase-isoparametric}*{Proposition 3.5}, we obtain a solution $f_0$ of the linear problem $\cL_M f_0 = 0$ with $\{ f_0( \frac{1}{\sqrt{2}}) = 1, f'_0(\frac{1}{\sqrt{2}}) = 0\}$, for $\cL_M$ the operator in~\eqref{eqn:self-adjoint-legendre}, which reaches zero with finite derivative.
    Therefore, the equation~\eqref{eqn:rescaled-equation-lambda-ODE} also admits a solution of the initial value problem $\{ \tilde{f}_{\lambda}( \frac{1}{\sqrt{2}}) = 1, \tilde{f}'_{\lambda} ( \frac{1}{\sqrt{2}}) = 0 \}$ reaching a zero, for $\lambda \in (0,\ve_0)$ sufficiently small, so $f_a := \sqrt{\lambda} \tilde{f}_{\lambda}$ solves~\eqref{eqn:ode-star} and has a zero, for $a = \sqrt{\lambda}$ sufficiently small.
    On the other hand, Proposition~\ref{prop:dont-shoot-too-high} shows that $f_a$ has derivative blow-up while positive if $a > C(M)$.
    Combining these two facts, the shooting-continuity argument of Section~\ref{section:type-I} produces a smooth family of solutions $\{ f_a \}_{a \in (0,a_M^*)}$ such that $f_{a_M^*}$ reaches zero with contact angle $\frac{\pi}{2}$ and the map $a \mapsto \arctan \bigl( \tfrac{g}{2}\sqrt{1 - t_a^2} \, |f'_a(t_a)| \bigr)$ surjects onto $(0,\frac{\pi}{2}]$.
    Finally, the extensions $\bar{f}_a$ produce the desired smooth capillary family $\{ \check{\mathbf{C}}_{M,a} \}_{a \in (0,a^*_M]}$.
\end{proof}

It would be interesting to apply these capillary methods to produce minimal embeddings into the sphere with other diffeomorphism types.
For example, Carlotto-Schulz~\cite{carlotto-schulz} have constructed minimal tori $\bS^{n-1} \times \bS^{n-1} \times \bS^{n-1} \times \bS^1$ in the sphere $\bS^{3n-1}$ conjectured by Hsiang-Hsiang~\cite{hsiang-hsiang}.

\section{CMC hypersurfaces in the sphere}\label{section:CMC-in-the-sphere}
The results of Theorems~\ref{thm:new-minimal-surfaces}\thru\ref{thm:smooth-interpolation} can more generally be extended to the construction of constant mean curvature hypersurfaces of the sphere $\bS^n$ with the various topologies described in Section~\ref{section:topology}.
For example,~\cites{huang-wei-2022 , lai-wei-tori} constructed CMC embeddings $\bS^{n-1} \times \bS^{n-1} \times \bS^1 \hookrightarrow \bS^{2n}$ by a direct adaptation of the arguments in~\cite{carlotto-schulz}.
These tori have the interesting feature that for mean curvature $H_0 \in (-\ve,0)$, there exist \textit{two} CMC surfaces of Type II, in the sense of Theorem~\ref{thm:new-minimal-surfaces}.
In our general situation, the topology of the resulting CMC surfaces in the sphere follows the characterization of Section~\ref{section:topology}, since that only entails qualitative information about the endpoints and reflection symmetry of the profile curve.

After proving general results about the existence of the CMC surfaces $\mathbf{S}^H_{M_1}  , \mathbf{S}^H_{M_2}, \bar{\mathbf{S}}^H_M$ in the sphere, in Section~\ref{section:axisymmetric} we prove uniqueness and rigidity results for axisymmetric solutions of the capillary CMC problem.
These characterizations include the uniqueness of the sphere and the Clifford torus as rotationally symmetric minimal surfaces in the sphere, as well as the capillary counterpart of this result formulated in Theorem~\ref{thm:uniqueness-of-axisymmetric}.

First, we recall some special solutions of the general CMC equation~\eqref{eqn:CMC-equation}.
For $g=2$, this equation has the CMC Clifford tori as explicit solutions: for $(m_1, m_2) = (k-1,n-k-1)$ and any $a > 0$, we find the family of solutions
\begin{equation}\label{eqn:cmc-clifford-solution}
f_a(t) = \sqrt{a - (a+1) t^2}, \qquad H_0(a) = \tfrac{k}{\sqrt{a}} - (n-k-1) \sqrt{a}.
\end{equation}
The map $a \mapsto H_0(a)$ is a strictly decreasing bijection $(0,\infty) \to (-\infty , \infty)$, so the profiles $f_a$ are free-boundary solutions recovering every value of $H_0$.
The graph of $\rho f_a(t)$ produces the cone $z = \sqrt{a |x|^2 - |y|^2}$ in $(x,y,z)$-coordinates, thus forming the link $\Sigma = \text{graph}( \rho f_a) \cap \bS^n_+$ and doubling across the equator results in the CMC Clifford torus $\bS^{n-k-1}( \sqrt{\frac{1}{a+1}}) \times \bS^{k} ( \sqrt{\frac{a}{a+1}})$.
The Lawson height $a = \frac{k}{n-k-1}$ recovers the minimal Clifford torus.
The analogous CMC solutions for the $g=1$ foliation of $\bS^{n-1}$ are presented in Section~\ref{section:axisymmetric}, including the Clifford torus $\bS^{n-2}( \sqrt{\frac{1}{a+1}}) \times \bS^1( \sqrt{\frac{a}{a+1}})$.

\subsection{The CMC equation}

In the neighborhood of a point where $\{ f(t_0) = 0, \sigma f'(t_0) = + \infty \}$ for $\sigma = \pm 1$, a local analysis of equation~\eqref{eqn:CMC-equation} analogous to the case $H_0 = 0$ produces
\begin{equation}\label{eqn:cmc-f(t0)-signs}
\begin{split}
    f(t)^2 &= \tfrac{4}{g \bigl( \sigma H_0 \sqrt{1-t_0^2} - (n-2) A_M(t_0) \bigr)} |t-t_0| + O (|t - t_0|^{\frac{3}{2}}), \\
    |f'(t)| &= \sqrt{\tfrac{1}{g \, |\sigma H_0 \sqrt{1-t_0^2}- (n-2) A_M(t_0)|}} |t - t_0|^{- \frac{1}{2}} + O(1) ,
\end{split}
\end{equation}
leading to an analogous power series expansion in terms of $\sqrt{t_0 - t}$, as in~\eqref{eqn:f-prime-t,f-t}.
In particular, the solutions of equation~\eqref{eqn:CMC-equation} satisfy the convergence and compactness Lemmas~\ref{lemma:slope-infinity-solution} and~\ref{lemma:convergence-of-solutions}.

We now introduce the Lyapunov quantity
\begin{equation}\label{eqn:lyapunov-quantity}
    \Psi_{H_0}(t) := f \Bigl( f - \frac{g}{2} A_M f' \Bigr) + \frac{H_0}{n-2} f \sqrt{1 + \frac{g^2}{4} (1-t^2) \frac{(f')^2}{1+f^2}} - \frac{1}{n-2}
\end{equation}
for $H_0$ the constant mean curvature appearing in equation~\eqref{eqn:CMC-equation}.
This quantity is constructed to satisfy the following properties, for $g \geq 2$:
\begin{enumerate}[(i)]
    \item When $H_0 = 0$, we have $\Psi_{H_0} = \Psi$ recovering the expression~\eqref{eqn:Psi}.
    \item At a point $t_0$ where $\{ f(t_0) = 0 , f'(t_0) = \pm \infty\}$, we have $\Psi_{H_0}(t_0) = 0$.
    \item For the Clifford solutions $f_a(t)$ of~\eqref{eqn:cmc-clifford-solution}, we have $\Psi_{H_0} \equiv 0$ identically.
    \item At the point $t_{\alpha} = \sqrt{\frac{m_1}{m_1 + m_2}}$ where $A_M(t_{\alpha}) = 0$, we have $\Psi'_{H_0}(t_{\alpha}) \geq 0$ whenever $\Psi_{H_0}(t_{\alpha}) \geq 0$. 
\end{enumerate}
To see the property $(iv)$, we compute that
\[
\Psi'_{H_0}(t_{\alpha}) = - f'(t_{\alpha}) \Bigl[ (g-2) f(t_{\alpha}) + \tfrac{H_0}{1 + f(t_{\alpha})^2} \sqrt{1 + \tfrac{g^2}{4} (1-t_{\alpha}^2) \tfrac{f'(t_{\alpha})^2}{1 + f(t_{\alpha})^2}} \,  \Psi_{H_0}(t_{\alpha}) \Bigr]
\]
where $- f'(t_{\alpha}) \geq 0$.
The quantity $\Psi_{H_0}$ has the same invariant regions as $\Psi$ from Proposition~\ref{prop:Psi-one-sided}.
\begin{proposition}\label{prop:psi-more-general}
    Suppose that $H_0 \geq 0$, $g \geq 3$, and consider a point $t_* \geq \sqrt{\frac{m_1}{m_1 +m_2}}$ and a solution $f$ of equation~\eqref{eqn:CMC-equation} with $\{ f(t_*) > 0, f'(t_*) \leq 0 \}$.
    \begin{enumerate}[(i)]
        \item If $\Psi'_{H_0}(t_*) = 0$, then $A_M(t_*) = f'(t_*) = 0$ or $\Psi_{H_0}(t_*) <0$.
        \item If $\Psi_{H_0}(t_*) = 0$, then $A_M(t_*) = f'(t_*) = 0$ or $\Psi'_{H_0}(t_*) > 0$.
    \end{enumerate}
    Notably, if $\Psi_{H_0}(t_*) , \Psi'_{H_0}(t_*) \geq 0$ then $f$ does not reach zero in its forward interval of definition.
\end{proposition}
\begin{proof}
    The properties $(i), (ii)$ follow from computations identical to those of Proposition~\ref{prop:Psi-one-sided}. 
    The fact that solutions with $\Psi_{H_0}(t_*) , \Psi'_{H_0}(t_*) \geq 0$ blow up while positive then follows from the argument of Lemma~\ref{lemma:psi-detect-blowup} and the properties used in the construction of $\Psi_{H_0}$.
\end{proof}
A direct computation as in Proposition~\ref{prop:mean-curvature-equation} shows that a solution~\eqref{eqn:CMC-equation} with $f'(0) = 0$ satisfies
\begin{equation}\label{eqn:f''(0)}
    f''(0) = -4 \frac{H_0 + (n-1) f(0)}{g^2(m_1 + 1)}.
\end{equation}
In particular, the constant function $f \equiv - \frac{H_0}{n-1}$ is always a solution of equation~\eqref{eqn:CMC-equation}, corresponding to the horizontal latitude sphere $\bS^{n-1} \bigl( \frac{n-1}{\sqrt{H_0^2 + (n-1)^2}} \bigr)$ at constant height $\frac{|H_0|}{\sqrt{H_0^2 + (n-1)^2}}$.
At a critical point $t_* > 0$ of $f$, we compute
\begin{equation}\label{eqn:critical-point-of-CMC}
(1-t_*^2) f''(t_*) = - \tfrac{4}{g^2} \bigl( H_0 + (n-1) f(t_*) \bigr).
\end{equation}
Using these properties, we can analyze the ODE~\eqref{eqn:CMC-equation} similarly to Section~\ref{section:analysis-of-capillary-equation}.
\begin{lemma}\label{lemma:cmc-equation-preliminaries}
    For $H_0 \geq 0$, consider a positive solution $f$ of~\eqref{eqn:CMC-equation} on $(t_1, b) \subset (0,1)$ with $\{f(t_1)\geq 0, f'(t_1) > 0\}$ and let $h_f := f - \frac{g}{2} A_M f'$.
    If $\alpha = \frac{m_1}{m_1 + m_2}$ and $t_1 < \sqrt{\alpha}$, the following properties hold:
    \begin{enumerate}[$(i)$]
        \item Either $f$ is strictly increasing for all time or has a unique critical point $t_c$. 
        Furthermore, $h_f > 0$ on $[t_1,b]$ and $h_f' < 0$ for $t < \min \{t_c, \sqrt{\alpha}\}$, while $h_f' > 0$ for $t > \max \{ t_c ,\sqrt{\alpha}\}$.

        \item If $f(t_2) = 0$, then $t_2 \geq \sqrt{\alpha}$.
        If $f(t_1) = f(t_2) = 0$, then $t_1 < \sqrt{\alpha} < t_2$.

        \item The function $f$ is strictly increasing if and only if $h_f$ has a zero at $t_h$ and $h_f, h_f' < 0$ for $t > t_h$. 
    \end{enumerate}
\end{lemma}
\begin{proof}
If $f$ has a critical point $t_c$ in $\{ f >0\}$, then the expression~\eqref{eqn:critical-point-of-CMC} forces $f''(t_c) < 0$, so $t_c$ is unique if it exists.
The definition of $h_f$ leads to the relation~\eqref{eqn:2.6NMS}, and equation~\eqref{eqn:CMC-equation} implies that
\begin{equation}\label{eqn:h-ODE-general}
    (1-t^2) h'_f = A_M(t) S h_f - Tf' + \tfrac{2H_0}{g} A_M(t) \Bigl( 1 + \tfrac{g^2}{4} (1-t^2) \tfrac{(f')^2}{1+f^2} \Bigr)^{\frac{3}{2}} .
\end{equation}
The quantities $S,T$ are defined as in~\eqref{eqn:ODE-for-h}, with
\[
S = \tfrac{2(n-1)}{g} + \tfrac{g(n-2)}{2} (1-t^2) \tfrac{(f')^2}{1+f^2}, \qquad T = \alpha( \tfrac{g}{2} - \alpha)t^{-2} + \tfrac{1}{2} (g-2) (1-2\alpha) > 0 .
\]
For $H_0 \geq 0$, the identity~\eqref{eqn:h-ODE-general} again shows that at a zero of $h_f$, we have $h'_f < 0$ if $\{ t < \sqrt{\alpha}, f' > 0 \}$ and $h'_f > 0$ if $\{ t > \sqrt{\alpha}, f'< 0 \}$.
Therefore, we may again argue as in Lemma~\ref{lemma:general-properties-of-solutions} and~\cite{new-minimal-surfaces}*{Proposition 2.2} to deduce that the region $R := \{ t \geq \sqrt{\alpha} : h_f(t) > 0, f'(t) < 0 \}$ is forward-invariant in $t$, with $h_f > 0$ while $f>0$, hence the result follows from similar arguments to the ones in~\cite{new-minimal-surfaces}.
\end{proof}

\begin{lemma}\label{lemma:shooting-endpoint-H0geq0}
    Suppose that $H_0 \geq 0$ and $g \geq 2$, then the following hold:
    \begin{enumerate}[$(i)$]
        \item There exists a $C(M, H_0)$ such that for any $a > C(M,H_0)$, the solution of equation~\eqref{eqn:CMC-equation} with initial data $\{ f(0) = a, f'(0) = 0 \}$ blows up while non-negative.  
        \item Let $t_{\alpha,H_0} \geq \sqrt{\frac{m_1}{m_1+m_2}}$ be the point where $H_0 \sqrt{1-t_{\alpha,H_0}^2} = (n-2) A_M(t_{\alpha, H_0})$.
        There exists an $\bar{\ve}_{M,H_0}>0$ such that for any $t_1 \geq t_{\alpha, H_0} - \bar{\ve}_{M,H_0}$ and $a>0$, the solution of equation~\eqref{eqn:CMC-equation} with initial data $\{ f(t_1) = 0, f'(t_1) = a \}$ does not reach a second zero.  
    \end{enumerate}
\end{lemma}
\begin{proof}
For $H_0 = 0$, these results are proved in Propositions~\ref{prop:dont-shoot-too-high} and~\ref{prop:small-breather}, respectively, so we denote $t_{\alpha} := \sqrt{\frac{m_1}{m_1+m_2}}$ and focus on $H_0 > 0$ in what follows.
For the first result, the expression~\eqref{eqn:cmc-clifford-solution} gives an explicit solution with derivative blowup when $g=2$, so we focus on $g \geq 3$ in what follows.
Using Proposition~\ref{prop:psi-more-general} in place of Proposition~\ref{prop:Psi-one-sided} and Lemma~\ref{lemma:psi-detect-blowup}, it suffices to show that $\Psi_{H_0} (t_{\alpha}) > 0$, which in turn reduces to $f_a( t_{\alpha}) > \frac{1}{\sqrt{n-2}}$ for $a$ sufficiently large, due to Proposition~\ref{prop:psi-more-general}.

In fact, we prove the stronger statement $f_a(t_{\alpha}) \to \infty$ as $a \to \infty$.
Suppose, for contradiction, that there exists a sequence $a_i \to \infty$ with $\limsup_i f_i(t_{\alpha}) < \infty$, and fix some $N > \limsup_i f_i(t_{\alpha}) + 2$.
On the interval where $f_i>0$, Proposition~\ref{prop:psi-more-general} implies $h_i := f_i - \frac{g}{2} A_M f_i' > 0$.
Equation~\eqref{eqn:CMC-equation} shows
\begin{equation}\label{eqn:pi'bound}
(1-t^2) f_i'' = -\frac{4}{g^2} f_i + t f_i' - (n-2)\left(\frac{4}{g^2} - \frac{(1-t^2)(f_i')^2}{1+f_i^2}\right)  h_i - \frac{4H_0}{g^2}
\left( 1 + \frac{g^2}{4}(1-t^2)\frac{(f_i')^2}{1+f_i^2} \right)^{\frac{3}{2}},
\end{equation}
so $f_i'(t) < 0$ on $(0,t_{\alpha})$ and the solutions are concave.

Since $A_M(t) < 0$ on $(0,t_{\alpha})$ and $h_i>0$, we have $f_i + \frac{g}{2} |A_M(t)| f_i' > 0$, so $|f_i'| < \frac{2 f_i}{g |A_M(t)|}$ with $|A_M(t)| = \frac{t_{\alpha}^2- t^2}{t}$, and integrating from $0$ to $t < t_{\alpha}$ implies that $f_i(t) \geq a_i \bigl( \frac{t_{\alpha}^2 - t^2}{t_{\alpha}^2} \bigr)^{\frac{1}{g}}$, so $f_i(t) \uparrow \infty$ as $i\to\infty$ for every fixed $t< t_{\alpha}$.
For all sufficiently large $i$, we therefore find a unique $\tau_i \in (0,t_{\alpha})$ where $f_i(\tau_i) = N$, so $L_i \leq f_i \leq N$ and $0 < h_i \leq f_i \leq N$ in $[\tau_i, t_{\alpha}]$.

Each term on the right-hand side of~\eqref{eqn:pi'bound} is non-positive, so keeping only the $H_0$ term in the slope equation, on $[\tau_i,t_{\alpha}]$, we can bound $f_i'' \leq \frac{gH_0}{2} \frac{\sqrt{1-t^2}}{(1+f_i^2)^{3/2}} \, (f_i')^3$.
Since $t\leq t_{\alpha}$ and $f_i\leq N$ on this interval, there exists a $c_N := \frac{gH_0}{2} \frac{\sqrt{1-t_{\alpha}^2}}{(1+N^2)^{3/2}} > 0$ such that $f_i'' \leq c_N (f_i')^3$ on $[\tau_i,t_{\alpha}]$, and integrating on $[ t, t_{\alpha}]$ shows that $\bigl( (f'_i)^{-2} \bigr)' = - 2 (f'_i)^{-3} f'' \leq - 2c_N$, so
\[
f'_i(t)^{-2} \geq f'_i(t_{\alpha})^{-2} + 2 c_N(t_{\alpha}-t) \implies |f'(t)| \leq \tfrac{1}{\sqrt{f'_i(t_{\alpha})^{-2} + 2 c_N(t_{\alpha}-t)}}.
\]
Integrating once more from $\tau_i$ to $t_{\alpha}$, we obtain
\[
N - L_i = -\int_{\tau_i}^{t_{\alpha}} f_i'(s)\, ds \leq  \int_{\tau_i}^{t_{\alpha}} \frac{ds}{ \sqrt{f'_i(t_{\alpha})^{-2} + 2 c_N (t_{\alpha}-s)} } \leq \int_0^{t_{\alpha}-\tau_i} \frac{d\sigma}{\sqrt{2 c_N \sigma}} =
\sqrt{\frac{2(t_{\alpha}-\tau_i)}{c_N}}.
\]
Since $N-L_i \geq 1$, it follows that $t_{\alpha} - \tau_i \geq \frac{c_N}{2}$ for all sufficiently large $i$.
This contradicts the bound $f_i(t) \geq a_i \bigl( \frac{t_{\alpha}^2 - t^2}{t_{\alpha}^2} \bigr)^{\frac{1}{g}}$ for $a_i \to \infty$, so $(i)$ is proved.

For the second result, we consider the point $t_{\alpha, H_0}$ where $H_0 \sqrt{1-t_{\alpha,H_0}^2} = (n-2) A_M(t_{\alpha, H_0})$.
Since $A_M(t)$ is increasing in $t$ with $A_M(1) = \frac{m_2}{m_1+m_2}$, this point is uniquely defined and satisfies $t_{\alpha, H_0} \in \bigl( \sqrt{\frac{m_1}{m_1+m_2}}, 1 \bigr)$ for any $H_0 \in (0,+\infty)$.
Therefore, $H_0 \sqrt{1-t_1^2} - (n-2) A_M(t_1) > 0$ for any $t_1 < t_{\alpha, H_0}$, so the expression~\eqref{eqn:cmc-f(t0)-signs} shows that the forward solution of equation~\eqref{eqn:CMC-equation} with initial data $\{ f(t_1) = 0, f'(t_1) = a \}$ with $a > 0$ is well-defined and initially increasing.
This computation also shows that solutions with this initial data remain strictly increasing for $t_1 > t_{\alpha, H_0}$, so we can restrict ourselves to $t_1 \in (t_{\alpha, H_0} - \bar{\ve}, t_{\alpha, H_0})$ in what follows. 

Using the expressions~\eqref{eqn:cmc-f(t0)-signs}, we have convergence and compactness as in Lemmas~\ref{lemma:slope-infinity-solution} and~\ref{lemma:convergence-of-solutions} for such solutions in terms of the shooting point $t_1$ and initial velocity $a$.
On the other hand, $H_0 > 0$ implies that $A_M(t_{\alpha,H_0}) > 0$, so we can find some $\bar{\ve}_{M,H_0} > 0$ with $A_M(t_{\alpha,H_0} - \bar{\ve}_{M,H_0}) > 0$.
For any $t_1$ with $t_{\alpha, H_0} - t_1 \in (0,\bar{\ve}_{M, H_0})$, we therefore find
\[
\lim_{t \downarrow t_1} h(t) = f(t_1) - \tfrac{g}{2} A_M(t_1) f'(t_1) = - \tfrac{1}{2} g a A_M(t_1) < 0
\]
for the solution of equation~\eqref{eqn:CMC-equation} with initial data $\{ f(t_1) = 0, f'(t_1) = a \}$.
Then, Lemma~\ref{lemma:cmc-equation-preliminaries} shows that this solution remains increasing for all $t > t_1$ and cannot reach a second zero, as desired.
\end{proof}

\begin{lemma}\label{lemma:some-M-breather}
    For any $H_0 \geq 0$ and $g \geq 2$, there exists some $\ve(M,H_0)$ such that for any $t_1 \in (0,\ve)$ and $a \in (0,\infty]$, the solution $f$ of equation~\eqref{eqn:CMC-equation} with initial data $\{ f(t_1) = 0 , f'(t_1) = a \}$ reaches a second zero $t_2$ in its maximal domain of existence, with $(1-t_2^2) f'(t_2)^2 < \frac{1}{2} a^2$.
\end{lemma}
\begin{proof}
    This result follows from the same arguments as in Proposition~\ref{prop:window-near-zero}.
\end{proof}

With these ingredients, we can now prove Theorem~\ref{thm:cmc-construction}.

\begin{figure}
    \centering
    \includegraphics[width=0.3\linewidth]{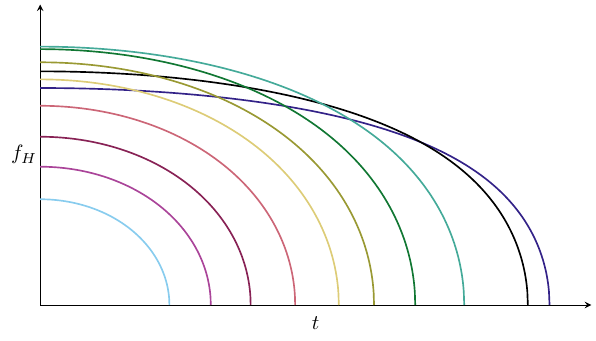}
    \includegraphics[width=0.3\linewidth]{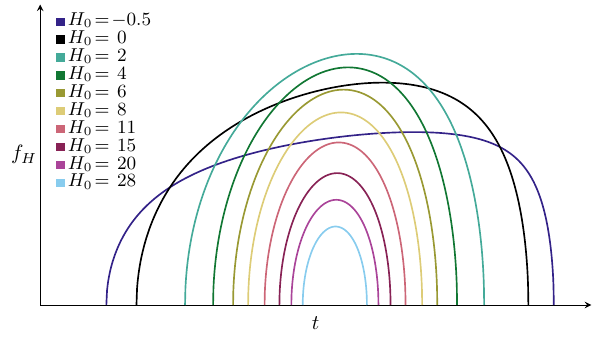}
    \includegraphics[width=0.3\linewidth]{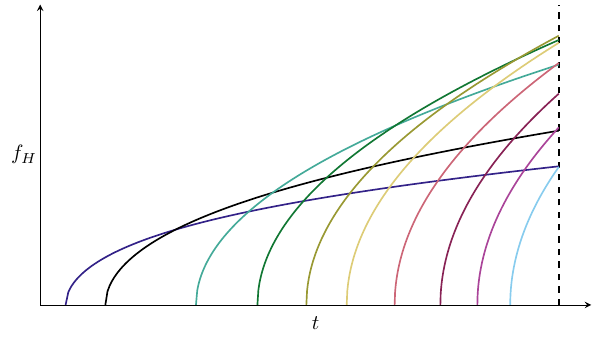}
    \caption{We display numerically computed profile curves for varying mean curvature $H$ for both Type I and II cones for $(g,m_1, m_2) = (4,2,5)$ corresponding to $\Sigma^H_{M_1, \frac{\pi}{2}}$. We can see some intervals of $H$ where the surfaces do not intersect, thereby laminating a region.
    However, this property fails for the minimal leaf $\Sigma_{M_1, \frac{\pi}{2}}$, which intersects the CMC profiles.}
    \label{fig:typeI&IICMCProfiles}
\end{figure}

\begin{proof}[Proof of Theorem~\ref{thm:cmc-construction}]
    The proof closely mirrors the shooting-continuity arguments in Theorems~\ref{thm:new-minimal-surfaces} - \ref{thm:smooth-interpolation} for the CMC surfaces $\mathbf{S}^{H}_{M_i}$ and $\bar{\mathbf{S}}^H_M$, so we highlight the modifications needed.
    We now restrict ourselves to $H_0>0$, which is in fact easier.

\smallskip \noindent \textbf{Type I:}
Considering solutions $f_a$ of equation~\eqref{eqn:CMC-equation} with $\{ f(0) = a, f'(0) =0 \} $, the expression~\eqref{eqn:f''(0)} shows that $f', f'' < 0$ for small $t>0$, so $f_a$ is strictly decreasing by Lemma~\ref{lemma:cmc-equation-preliminaries}.
The property~\eqref{eqn:f''(0)} together with the smooth dependence of ODE solutions shows that there exists some $t_0 \in (0,1)$ and a sufficiently small $\ve_0$ such that any solution $f_a$ with $a \in (0,\ve_0)$ satisfies $f_a''(t) \leq - \frac{2 H_0}{g^2(1+m_1)}$ for $t \in [0,t_0]$.
Since $f'_a(0) = 0$, integrating twice yields $f_a(t) \leq a - \frac{H_0}{g^2(1+m_1)} t^2$, so $f_a(t_0) < 0$ for $a < \ve^*_0 := \frac{H_0 t_0^2}{g^2(1+m_1)}$ and the solutions $f_a$ reach a zero $t_a$ with finite derivative.

In fact, we obtain $\sqrt{1-t_a^2} |f'_a(t_a)| \to 0$ as $a \downarrow 0$, with the rate
\begin{equation}\label{eqn:rate-of-derivative}
\sqrt{1-t_a^2} |f'_a(t_a)| = \frac{2 \sqrt{2H_0}}{g \sqrt{m_1 + 1}} a^{\frac{1}{2}} (1 + o(1)), \qquad \text{as } \; a \downarrow 0.
\end{equation}
Indeed, the above arguments give Taylor expansions $f_a(t) = a + \frac{1}{2} f''_a(0) t^2 + o(t^2)$ and $f'_a(t) = f''_a(0) t + o(t)$ uniformly as $t< \frac{t_0}{10},a < \frac{H_0t_0^2}{10 g^2(1+m_1)}$.
Using $f''_a(t) = -\frac{4 H_0}{g^2(m_1+1)} + o(1)$ for $t < \frac{t_0}{10}$, we find
\[
0 = f_a(t_a) = a + \tfrac{1}{2} f''_a(0) t_a^2 + o(t_a^2) \implies t_a^2 = 2\bigl( \tfrac{4 H_0}{g^2(m_1+1)} \bigr)^{-1} a(1 + o(1)) = \tfrac{g^2(m_1+1)}{2H_0} a(1 + o(1)),
\]
so $t_a = g \sqrt{\frac{(m_1+1) a}{2 H_0}}(1 + o(1))$ and $|f'_a(t_a)| = \frac{4 H_0}{g^2(m_1+1)} t_a(1+o(1)) = \frac{2 \sqrt{2H_0}}{g \sqrt{m_1+1}} a^{\frac{1}{2}}$.
Combining these properties proves the claimed behavior~\eqref{eqn:rate-of-derivative} as $a \downarrow 0$.

For $g = 2$ principal curvatures, we have an explicit solution~\eqref{eqn:cmc-clifford-solution} of equation~\eqref{eqn:CMC-equation} with derivative blowup while non-negative; for $g \geq 3$, Lemma~\hyperref[lemma:shooting-endpoint-H0geq0]{\ref{lemma:shooting-endpoint-H0geq0} $(i)$} shows that there exists an $a_{H_0}$ for which $f_{a_{H_0}}$ has derivative blowup while positive.
Combined with the property $\sqrt{1-t_a^2} |f'_a(t_a)| \to 0$ as $a \downarrow 0$ from~\eqref{eqn:rate-of-derivative}, this shows that a subset of the interval $(0 , a_{H_0}]$ produces solutions $f_a$ with a zero at $t_a$ such that the map $a \mapsto \arctan \bigl(\frac{g}{2} \sqrt{1-t_a^2} |f'_a(t_a)| \bigr)$ is surjective onto $(0, \frac{\pi}{2}]$.
This proves the claim.

\smallskip \noindent \textbf{Type II:}
The proof of this result proceeds identically to part II of Theorem~\ref{thm:new-minimal-surfaces}.
We use Lemma~\ref{lemma:some-M-breather} to show that the shooting procedure with $\{ f_{t_1}(t_1) = 0 , f'_{t_1}(t_1)= \frac{2}{g} \frac{\tan \theta}{\sqrt{1-t_1^2}} \}$ can be initiated for $t_1$ sufficiently close to $0$, producing by Lemma~\ref{lemma:cmc-equation-preliminaries} a second zero $\tau(t_1) > \sqrt{\frac{m_1}{m_1+m_2}}$ where
\[
R[t_1] := (1-t_1^2) f'_{t_1}(t_1)^2 - (1 - \tau(t_1)^2) f'_{t_1}(\tau(t_1))^2 > \tfrac{1}{g^2} \tan^2 \theta.
\]
Then, Lemma~\hyperref[lemma:shooting-endpoint-H0geq0]{\ref{lemma:shooting-endpoint-H0geq0} $(ii)$} yields $R[t_1] \to + \infty$ as $t_1 \uparrow (t_{\alpha, H_0} - \bar{\ve}_{M, H_0})$, so the same arguments apply.

\smallskip \noindent \textbf{Type II, $m_1 = m_2$:}
Arguing as in the Type I case shows that for $a \in (0,\ve_0)$ sufficiently small, the solution $f_a$ of equation~\eqref{eqn:CMC-equation} with initial data $\{ f( \frac{1}{\sqrt{2}}) = a, f'(\frac{1}{\sqrt{2}}) = 0\}$ reaches zero with finite derivative.
Moreover, Proposition~\ref{prop:psi-more-general} shows that $f_a$ blows up while positive for $a > C(M, H_0)$ sufficiently large, therefore, we may argue as in the second part of Theorem~\ref{thm:smooth-interpolation} to conclude.
\end{proof}

Finally, we comment on the modifications needed in the construction of Theorem~\ref{thm:cmc-construction} to produce closed embedded CMC surfaces in $\bS^n$ with negative mean curvature.
Unlike the above result, which allows any $H_0 \geq 0$, the negative CMC construction is only possible for $H_0 \in (-\ve_n,0)$ for $\ve_n>0$ a small dimensional constant.
Geometrically, the resulting hypersurfaces are small normal perturbations of the corresponding minimal surfaces $\mathbf{S}_{M_i}, \bar{\mathbf{S}}_M$ produced in Theorem~\ref{thm:new-minimal-surfaces}.

For equation~\eqref{eqn:CMC-equation} with $H_0 \leq 0$, the quantity $h = f - \frac{g}{2} A_M f'$ does not satisfy the full monotonicity of Lemma~\ref{lemma:cmc-equation-preliminaries}.
We have the following replacement:
\begin{lemma}\label{lemma:max-followed-by-min}
Consider a solution $f$ of equation~\eqref{eqn:CMC-equation} with $H_0 \leq 0$ and let $t_{\alpha} := \sqrt{\frac{m_1}{m_1 +m_2}}$.
Then,
\begin{enumerate}[(i)]
    \item A local maximum of $f$ in $(0, t_{\alpha})$ cannot be followed by a local minimum still in $(0, t_{\alpha})$.
    \item A local minimum in $(t_{\alpha},1)$ cannot be followed by a local maximum still in $(t_{\alpha},1)$.
\end{enumerate}
\end{lemma}
\begin{proof}
Similarly to Lemma~\ref{lemma:cmc-equation-preliminaries}, we consider the quantity $\tilde{h} := h_f + \frac{H_0}{n-1}$, which satisfies $\tilde{h} \equiv 0$ for the constant solution $f_0 = - \frac{H_0}{n-1}$.
We introduce the integrating factor $\mu(t) := \exp \bigl( - \int_{t_0}^t \frac{A(s) S(s)}{1-s^2} \, ds \bigr)$, so the equation~\eqref{eqn:h-ODE-general} for $h_f$ becomes
\begin{equation}\label{eqn:mu-h-tilde-equation}
    (1-t^2) (\mu \tilde{h})' = - \mu \left[ Tf' -  \tfrac{2 H_0}{g(n-1)} A_M(t) \Phi \Bigl( \tfrac{g^2}{4} (1-t^2) \tfrac{(f')^2}{1+f^2} \Bigr) \right] , 
\end{equation}
where $\Phi(x) := (n-1) \bigl( (1+x)^{\frac{3}{2}} - 1 \bigr) - (n-2) x$.
We observe that $\Phi(0) = 0$ and $\Phi'(x) = \frac{3}{2} (n-1) \sqrt{1+x} - (n-2) > \frac{n+1}{2} > 0$, so $\Phi(x) > 0$ for $x>0$.
The equation~\eqref{eqn:mu-h-tilde-equation} implies that
\begin{equation}\label{eqn:t-sqrt-alpha-sufficient-property}
\{ t \leq t_{\alpha}, \; f' \leq 0 \} \implies ( \mu \tilde{h})' \geq 0, \qquad \{ t \geq t_{\alpha}, \; f' \geq 0 \} \implies ( \mu \tilde{h})' \leq 0,
\end{equation}
with strict inequality whenever one of the above inequalities is strict.
Moreover, at a critical point $t_c$ of $f$, the identity~\eqref{eqn:critical-point-of-CMC} shows that $(1-t_c^2) f''(t_c) = - \frac{4(n-1)}{g^2} \tilde{h}(t_c)$, so $\tilde{h}(t_c) > 0$ if and only if $t_c$ is a local maximum, and $\tilde{h}(t_c) < 0$ if and only if it is a local minimum.

These properties imply the Lemma as follows.
Taking $s_1 < s_2$ to be the two successive critical points in case $(i)$ or $(ii)$, we would have $(\mu \tilde{h})(s_2) < 0 < (\mu \tilde{h}) (s_1)$ and $f' < 0$ in $(s_1,s_2)$  (respectively, $( \mu \tilde{h}) (s_1) < 0 < ( \mu \tilde{h})(s_2)$ and $f' > 0$), contradicting the sign of $( \mu \tilde{h})'$ from~\eqref{eqn:t-sqrt-alpha-sufficient-property}. 
\end{proof}

Moreover, Lemma~\hyperref[lemma:shooting-endpoint-H0geq0]{\ref{lemma:shooting-endpoint-H0geq0}$(i)$} also holds for $H_0 < 0$, in the sense that solutions with $\{ f_a(0)= a, f'_a(0) =0 \}$ satisfy $f_a(\sqrt{\frac{m_1}{m_1+m_2}}) \to + \infty$ as $a \to \infty$.
However, the quantity $\Psi_{H_0}$ introduced in~\eqref{eqn:lyapunov-quantity} to detect the blowup of Type I solutions no longer enjoys the monotonicity properties of Proposition~\ref{prop:psi-more-general}.
Instead, one argues by continuity of the shooting problem, showing that the profile curve of a CMC hypersurface cannot degenerate to the focal manifold endpoints $\{ 0,1\}$.

\begin{lemma}\label{lemma:corner-of-the-box}
    For $g \geq 2$ and $H_0 \in \bR$, there do not exist solutions of equation~\eqref{eqn:CMC-equation} on $[t_1,t_2]$ with
    \begin{align*}
    & \{ (f(t_1)= 0 \; \textup{ and } \; f'(t_1) = + \infty) \quad \textup{or} \quad t_1 = 0 \} \qquad \textup{and} \qquad \{ t_2 = 1, \; f(1) = 0 \}, \quad \qquad\textup{or} \\
    & \{ t_1 =0, \; f(t_1) = 0 \} \qquad \textup{and} \qquad \{ ( f(t_2) = 0 \; \text{ and } \; f'(t_2) = - \infty) \quad \textup{or} \quad t_2 = 1 \}.
    \end{align*}
    Equivalently, the resulting CMC surface $\mathbf{S} \subset \bS^n$ cannot be tangent to the equatorial $\bS^{n-1}$.
\end{lemma}
For $H_0 \geq 0$, this result is clear by the endpoint computation of Proposition~\ref{prop:mean-curvature-equation}.
For $H_0 < 0$, one may argue using Brendle's rigidity theorem for CMC hypersurfaces~\cite{brendle-ihes}.

Finally, one needs to replace the shooting-continuity argument of Theorem~\ref{thm:cmc-construction} by a more delicate argument.
The necessary modification lies in the first step of the construction of Type I solutions, namely in proving that solutions $f_a$ of the problem with initial data $\{ f(0) = a, f'(0) = 0 \}$ reach a zero with finite derivative when $a$ is in an appropriate range.
Then, letting $a \to \infty$ produces a critical initial height $a_*$ such that the solution $f_{a_*}$ is the profile curve of a free-boundary minimal surface, which can be doubled to a complete embedded CMC surface $\mathbf{S}^H_{M_i}$ as in Theorem~\ref{thm:cmc-construction}.
The following existence argument is only valid for $H_0 \in (-\ve_n,0)$ sufficiently small, indicating why a general construction is not possible for $H_0<0$: the expression~\eqref{eqn:f''(0)} forces solutions with $f'(0) = 0$ to satisfy $f(0) > \frac{|H_0|}{n-1}$ if $f'<0$ initially, but such profiles blow up before reaching zero if $|H_0| \to \infty$.
\begin{lemma}\label{lemma:reach-zero-H0-negative}
    There exist sufficiently small $\ve_0, \lambda_0$ depending only on $(g, m_1, m_2)$ such that for any $a \in (0,\ve_0)$ and $H_0 \in ( - \lambda_0 a, 0)$, the solution of equation~\eqref{eqn:CMC-equation} with data $\{ f(0) = a, f'(0) = 0 \}$ reaches zero with finite derivative, at a point $t > \sqrt{\frac{m_1}{m_1 + m_2}}$.
\end{lemma}
\begin{proof}
    Let $f(t) = a \, f_{a,\lambda}(t)$ and $\lambda := a^{-1} |H_0|$, so $f_{a,\lambda}$ solves the equation
    \begin{equation}\label{eqn:inhomogeneous-equation}
    \cL_M f_{a,\lambda} = \tfrac{4 \lambda}{g^2} + a\, E(a,f_{a,\lambda},f_{a,\lambda}',f_{a,\lambda}''), \qquad \{ f_{a,\lambda}(0) = 1, f_{a,\lambda}'(0) = 0 \}
    \end{equation}
    where $\cL_M$ is the operator from~\eqref{eqn:self-adjoint-legendre} and $E(a,f_{a,\lambda},f_{a,\lambda}',f_{a,\lambda}'')$ is a remainder term obtained by rescaling as in equation~\eqref{eqn:rescaled-equation-lambda-ODE} and Proposition~\ref{prop:i-will-take-a-limit}, which satisfies $\sup |E(a,u,u',u'')| \leq M( |u|_{C^{2,\alpha}})$.
    For fixed $\lambda$, the solution $f_{a,\lambda}$ of equation~\eqref{eqn:inhomogeneous-equation} converges uniformly in $C^{\infty}$ to the solution $f_{\lambda}$ of the linear problem $\cL_M f_{\lambda} = \frac{4 \lambda}{g^2}$, which admits the constant solution $ \frac{\lambda}{n-1}$.
    We therefore obtain
    \[
    \cL_M f_{\lambda} = \tfrac{4 \lambda}{g^2}, \quad \{ f_{\lambda}(0) = 1, f_{\lambda}'(0) = 0 \} \quad \implies \quad f_{\lambda}(t) = \tfrac{\lambda}{n-1} + \bigl( 1 - \tfrac{\lambda}{n-1} \bigr) f_M(t),
    \]
    where $f_M(t)$ is the hypergeometric function~\eqref{eqn:hypergeometric-isoparametric}.
    For $M = \bS^{n-2}$, we compute this function as $f_M(t) = 1 - 2t^2$, while $g \geq 2$ produces $f_M(t) \to - \infty$ as $t \uparrow 1$.
    We therefore obtain $f_{\lambda}(1) < 0$ in either case, so $f_{\lambda}$ has a zero $t_{\lambda}$.
    Moreover, using the fact that $f_M > 0$ on $[0,\sqrt{\frac{m_1}{m_1+m_2}}]$ from~\cite{one-phase-isoparametric}*{Lemma 3.2}, we find $f_{\lambda} > 0$ there, so $t_{\lambda} > \sqrt{\frac{m_1}{m_1 + m_2}}$.
    Finally, the convergence $f_{a,\lambda} \to f_{\lambda}$ shows that $f_{a,\lambda}$ also has a zero $t_{a,\lambda} > \sqrt{\frac{m_1}{m_1 + m_2}}$ for $a \in (0,\ve_0)$ sufficiently small.
    We then compute
    \[
    \lim_{a \downarrow 0} (1 - t_{a,\lambda}^2) f'_{a,\lambda}(t_{a,\lambda})^2 = (1 - t^2_{\lambda}) f'_{\lambda}(t_{\lambda})^2 = (1 - t_{\lambda}^2) \bigl( 1 - \tfrac{\lambda}{n-1} \bigr)^2 f'_M(t_{\lambda})^2 < + \infty \, ,
    \]
    hence $f_{a,\lambda}$ reaches zero with finite derivative for $a \in (0,\ve_0)$ sufficiently small.
\end{proof}

The existence of Type II surfaces $\bar{\mathbf{S}}^H_M$ for $H \in (-\ve_n,0)$ can be obtained by similar methods to Lemma~\ref{lemma:shooting-endpoint-H0geq0} and Theorem~\ref{thm:cmc-construction}.
Indeed, Proposition~\ref{prop:small-breather} shows that the solutions of equation~\eqref{eqn:ode-star} with initial data $\{ f(t_1) = 0, f'(t_1) = a\}$ remain increasing if $t_1 > \sqrt{\frac{m_1}{m_1+m_2}} - \bar{\ve}_M$, and equation~\eqref{eqn:CMC-equation} has the same property for $H_0 \in (-\ve_n,0)$.

\subsection{Uniqueness of the axisymmetric capillary cone}\label{section:axisymmetric}

We now study the minimal and CMC surfaces coming from the foliation of $\bS^{n-1}$ with $g=1$, given by CMC spheres $M_s = \{ \omega \in \bS^{n-1} : \la \omega, e_n \rg = \cos s \}$ with focal submanifolds the north and south poles.
The equation~\eqref{eqn:ode-star} becomes
\[
(1-t^2) f'' + (f - tf') + 3f + (n-2) \left( 4 + (1-t^2) \tfrac{(f')^2}{1+f^2} \right) \left( f - \tfrac{1}{2} \bigl( t - \tfrac{1}{2} t^{-1} \bigr) f' \right) = 0.
\]
For any $a$, the function $f(t) = a(1-2t^2)$ satisfies
\[
f - \tfrac{1}{2} \bigl( t - \tfrac{1}{2} t^{-1} \bigr) f' = 0, \qquad (1-t^2) f'' + (f - tf') + 3f = 0
\]
and $f'(0) = 0$, so it is the unique solution of the above equation with data $\{ f(0) = a, f'(0) = 0 \}$.
Recalling that $t = \sin \frac{s}{2}$, we find $\phi(s) = f ( \sin \frac{s}{2}) = a \cos s$, so $\cos s = \la \omega, e_n \rg$ produces $\phi(\omega) = a \omega_n$ for $\omega = \frac{x}{|x|}$ and $\la \omega, e_n \rg = \frac{x_n}{|x|}$.
The resulting homogeneous degree-one function on $\bR^n$ is $u(x) = \rho \phi(\omega) = a x_n$, whose graph is the capillary half-plane solution $u(x) = (x_n)_+$.
Notably, there exist no non-trivial Type I solutions in this case.
On the other hand, the Clifford torus $\bS^{n-2} ( \sqrt{\frac{n-2}{n-1}}) \times \bS^1 ( \sqrt{\frac{1}{n-1}}) \subset \bS^n$ is obtained as the unique Type II minimal surface in these coordinates, with $\bar{f}_{\bS^{n-2}}(t)$ given in~\eqref{eqn:clifford-solution}.
To classify the $F$-invariant capillary surfaces in this case, we first recall the equation for axisymmetric profile curves obtained in~\cite{FTW-1} in terms of the ambient coordinate $t = \frac{x_n}{|x|}$, given by
\begin{equation}\label{eqn:axisymmetric-equation}\tag{A}
     f'' + (f-tf')\left(\frac{n-1}{1-t^2} + (n-2)\frac{(f')^2}{1+f^2}\right) = -\frac{H_0}{1-t^2}\left(1 + (1-t^2)\frac{(f')^2}{1+f^2}\right)^{\frac{3}{2}}.
\end{equation}
The capillary boundary condition becomes $(1-t_*^2) f'(t_*)^2 = \tan^2 \theta$ at every zero $t_*$ of $f$.

The main result of this section is the following rigidity property, leading to a uniqueness result of rotationally symmetric capillary minimal and CMC surfaces.
\begin{theorem}\label{thm:even-function-axisymmetric}
    Any solution of the axisymmetric capillary CMC problem for~\eqref{eqn:CMC-equation} is even.
\end{theorem}

We first reduce equation~\eqref{eqn:CMC-equation} to the ODE~\eqref{eqn:axisymmetric-equation} and derive the capillary boundary condition.

\begin{lemma}\label{lemma:axisymmetric-equivalence}
    The equation~\eqref{eqn:CMC-equation} for $\tau = \sin \frac{s}{2}$ is equivalent to equation~\eqref{eqn:axisymmetric-equation} for $\xi = \frac{x_n}{|x|}$ under the change of variables $\xi = 1 - 2 \tau^2$ and $\tau = \sqrt{\frac{1-\xi}{2}}$.
    The profile curve $\bar{f}_{\bS^{n-2}}(t)$ with expression~\eqref{eqn:clifford-solution} produces the Clifford torus $\bS^{n-2} (\sqrt{\frac{n-2}{n-1}}) \times \bS^1 ( \sqrt{\frac{1}{n-1}})$.
\end{lemma}
\begin{proof}
    This result follows by direct computation, using the chain rule.
    Note that $\omega = ( \sin s \eta, \cos s) \in \bS^{n-1}$ for $\eta \in \bS^{n-2}$, where $s$ is the isoparametric coordinate, so $\xi = \la \omega, e_n \rg = \cos s$ and $\tau = \sin \frac{s}{2}$ are related by $\xi = 1-2 \tau^2$ and $\tau = \sqrt{\frac{1-\xi}{2}}$ as claimed.
    The computation of~\cite{FTW-1}*{Lemma 3.3} shows that the Clifford torus $\bS^{n-2} ( \sqrt{\frac{n-2}{n-1}}) \times \bS^1( \sqrt{\frac{1}{n-1}})$ in the axisymmetric variable $\xi$ is the solution $\hat{f}_{n,1}(\xi) = \sqrt{\frac{1 - (n-1) \xi^2}{n-2}}$ of equation~\eqref{eqn:axisymmetric-equation}, and substituting $\xi = 1 - 2 \tau^2$ produces~\eqref{eqn:clifford-solution}.
\end{proof}
For arbitrary $H_0$, the family of axisymmetric solutions~\eqref{eqn:clifford-solution} generalizes to
\begin{equation}\label{eqn:clifford-solution-a=1}
    f_a (\tau) = \sqrt{a - (a+1) (1-2\tau^2)^2}, \qquad H_0(a) = \tfrac{1}{\sqrt{a}} - (n-2) \sqrt{a}
\end{equation}
with the properties discussed above, and the resulting hypersurface is the CMC torus $\bS^{n-2}( \sqrt{\frac{1}{a+1}}) \times \bS^1(\sqrt{\frac{a}{a+1}})$.
Unlike the case $H_0 = 0$, for every $H_0 > 0$ we also obtain axisymmetric solutions
\begin{equation}\label{eqn:axisymmetric-solutions}
    f_{\lambda}(\tau) = \sqrt{\lambda- 4 (\lambda+1) \tau^2(1-\tau^2)}, \qquad H_0(\lambda) = \tfrac{n-1}{\sqrt{\lambda}}
\end{equation}
to equation~\eqref{eqn:CMC-equation}.
Under the change of variables $t = 1 -2 \tau^2$, the functions $f_{\lambda}(\tau)$ become $f_{\lambda}(t) = \sqrt{(\lambda+1) t^2 - 1}$ for $t = \frac{x_n}{|x|}$, with resulting surfaces the vertical latitude spheres $\bS^{n-1} ( \sqrt{\frac{\lambda}{1+\lambda}}) \subset \bS^n$ along the hyperplane $\{ x_n = \frac{1}{\sqrt{1+\lambda}} \}$.
We may therefore study the axisymmetric equation~\eqref{eqn:axisymmetric-equation} in the variable $t = \frac{x_n}{|x|}$ for $x = (x', x_n) \in \bR^{n-1} \times \bR$, where $t$ is allowed to be positive as well as negative.

To prove the uniqueness of solutions to equation~\eqref{eqn:axisymmetric-equation}, we introduce some auxiliary quantities.
\begin{definition}\label{def:auxiliary-quantities}
Given a solution $f$ of equation~\eqref{eqn:axisymmetric-equation}, we denote
\[
h:=f-tf', \qquad u:=\frac{t^2+f^2}{1+f^2},\qquad \theta:=\arg(t+if), \qquad
W:=\sqrt{1+f^2+(1-t^2)(f')^2}.
\]
We also define
\begin{align}
    \cW &:=(1-t^2)^{\frac{n-1}{2}} (1+f^2)^{- \frac{n-1}{2}} \left( \frac{\sqrt{1+f^2} ( f - tf')}{\sqrt{1 + f^2 + (1-t^2) (f')^2}} + \frac{H_0}{n-1} \right), \label{eqn:conserved-quantity} \\
    D(u)&:=\mathcal W(1-u)^{-\frac{n-1}{2}}-\frac{H_0}{n-1} = \frac{\sqrt{1+f^2} \, h}{W},  \label{eqn:D(u)-quantity} \\
    x &:= \sqrt{\frac{u}{1-u}}, \qquad z := \frac{D(u)}{x} = D(u) \sqrt{\frac{1-u}{u}}. \label{eqn:introduce-x,z}
\end{align}
Applying equation~\eqref{eqn:axisymmetric-equation} shows that the above quantities satisfy the identities 
\begin{align}
    \frac{d\theta}{dt}&=-\frac{h}{t^2+f^2},
\qquad
u' = 2\frac{t(1+f^2)+(1-t^2)ff'}{(1+f^2)^2}, \label{eqn:shared-identities-axisymmetric} \\
    h' &= t\left(\frac{n-1}{1-t^2}+(n-2)\frac{(f')^2}{1+f^2}\right)h + \frac{tH_0}{1-t^2}\left(1+(1-t^2)\frac{(f')^2}{1+f^2}\right)^{\frac{3}{2}}, \label{eqn:h-ODE-axisymmetric} \\
\bigl( t(1+f^2)&+(1-t^2)ff' \bigr)^2 =
(1+f^2)W^2\bigl(u-(1-u)D(u)^2\bigr), \label{eqn:shared-energy-identity-axisymmetric} \\
D(u) &= \cW (1+x^2)^{\frac{n-1}{2}} - \tfrac{H_0}{n-1}, \qquad z = \cW (1+x^2)^{\frac{n-1}{2}} x^{-1} - \tfrac{H_0}{n-1} x^{-1}, \label{eqn:D(u)-from-x} \\
\frac{du}{d\theta} &= \pm \frac{2u\sqrt{u-(1-u)D(u)^2}}{D(u)}. \label{eqn:shared-autonomous-u-theta}
\end{align}
The last property is valid on any interval where $h$ has a strict sign, so $\theta$ defines a global coordinate.
\end{definition}

\begin{lemma}\label{lemma:conserved-quantity}
For any $H_0  \in \bR$, the quantity $\cW$ is constant along solutions of equation~\eqref{eqn:axisymmetric-equation}.
\end{lemma}
\begin{proof}
The quantity $W := \sqrt{1 + f^2 + (1-t^2)(f')^2}$ introduced in Definition~\ref{def:auxiliary-quantities} satisfies
\begin{align*}
\frac{f - tf' + (1-t^2) f''}{W^2} &= - \left[ \frac{(n-2)(f-tf')}{1+f^2} + \frac{H_0 W}{(1+f^2)^{\frac{3}{2}}} \right], \qquad 2WW' = 2f'(f - tf' + (1-t^2)f''),
\end{align*}
whereby $W' = -(n-2) \frac{f'(f-tf')}{1+f^2}W - H_0 \frac{f'}{(1+f^2)^{\frac{3}{2}}}W^2$.
Applying~\eqref{eqn:axisymmetric-equation} again yields
\[
\frac{(f-tf')'}{f-tf'} = \frac{(n-1)t}{1-t^2} + (n-2)\frac{t(f')^2}{1+f^2} + \frac{t H_0}{(1-t^2)(f-tf')} \frac{W^3}{(1+f^2)^{\frac{3}{2}}}.
\]
We define $\cW_1 := (1-t^2)^{\frac{n-1}{2}} (1+f^2)^{-\frac{n-2}{2}} \frac{f-tf'}{W}$ and $\cW_2:= (1-t^2)^{\frac{n-1}{2}} (1+f^2)^{-\frac{n-1}{2}}$, which satisfy
\begin{align*}
    \frac{d}{dt} \log \cW_1 &= -\frac{(n-1)t}{1-t^2} - (n-2)\frac{ff'}{1+f^2} + \frac{(f-tf')'}{f-tf'} - \frac{W'}{W}, \\
   \frac{d}{dt} \cW_1 &= \frac{H_0 W}{(1+f^2)^{\frac{3}{2}}} \left( \frac{tW^2}{(1-t^2)(f-tf')} + f' \right) \\
    &= H_0 (1-t^2)^{\frac{n-3}{2}} (1+f^2)^{-\frac{n+1}{2}} \bigl( t(1+f^2) + (1-t^2)ff' \bigr) \\
    &= - \frac{H_0}{n-1} \frac{d}{dt} \cW_2.
\end{align*}
In the third step, we used the identity $tW^2 + (1-t^2)(f-tf')f' = t(1+f^2) + (1-t^2)ff'$.
We conclude that $\frac{d}{dt} \cW = \frac{d}{dt} ( \cW_1 + \frac{H_0}{n-1} \cW_2) = 0$, so $\cW$ is conserved as claimed.
\end{proof}

\begin{lemma}\label{lemma:axisymmetric-same-side-angle-inequality}
There exists no positive solution $f$ of equation~\eqref{eqn:axisymmetric-equation} on an interval $0 \not\in (t_1,t_2)$ with $f(t_1)=f(t_2)=0$ satisfying the capillary boundary condition $(1-t_1^2)f'(t_1)^2 = (1-t_2^2)f'(t_2)^2$.
\end{lemma}
\begin{proof}
    It suffices to prove the claim when $0<t_1<t_2$, otherwise we replace $f(t)$ with $f(-t)$.
    The differential equation~\eqref{eqn:h-ODE-axisymmetric} for $h = f - tf'$ has the form $h' = a(t) h + b(t)$ with $\on{sgn} b(t) = \on{sgn} H_0$ while $f'(t_2) < 0 < f'(t_1)$ makes $h(t_1) < 0 < h(t_2)$, so $h$ has a first zero $t_* > t_1$ with $h'(t_*) \geq 0$.
    For $H_0 \leq 0$, computing at $t_h$ forces $h'(t_*) = h(t_*) = H_0 = 0$, and ODE uniqueness implies that $f \equiv 0$.

    For $H_0 > 0$, let $ \tilde{H}_0 := \frac{H_0}{n-1}$, so $h'(t_h) = b(t_h) > 0$ at a zero $t_h \in (t_1,t_2)$ of $h$ shows that this zero is unique.
    Setting $x_0 := x(t_h)$ and using the relation~\eqref{eqn:D(u)-from-x} together with $D(u) (t_h)= 0$, we find
    \begin{equation}\label{eqn:D(u)-z-properties}
    \begin{split}
        D(u) &= \tilde{H}_0 \left(\frac{(1+x^2)^{\frac{n-1}{2}}}{(1+x_0^2)^{\frac{n-1}{2}}}-1\right), \qquad z(x)=\frac{\tilde{H}_0}{x}\left(\frac{(1+x^2)^{\frac{n-1}{2}}}{(1+x_0^2)^{\frac{n-1}{2}}}-1\right), \\
        \frac{\partial z}{\partial x} &= \frac{\tilde{H}_0}{x^2(1+x_0^2)^{\frac{n-1}{2}}}
\left[ (1+x^2)^{\frac{n-3}{2}}\bigl((n-2)x^2-1\bigr)+(1+x_0^2)^{\frac{n-1}{2}}\right].
    \end{split}
    \end{equation}
The function $s \mapsto s^{\alpha} \bigl( s - \tfrac{2(\alpha+1)}{2\alpha+1} \bigr)$ has derivative $\frac{1+\alpha}{1+2\alpha} \bigl( s + 2 \alpha(s-1) \bigr) s^{\alpha-1} > 0$ for $s>1$ and $\alpha>- \frac{1}{2}$, so it is strictly increasing; using this fact for $s = 1+x^2$ and $\alpha = \frac{n-3}{2}$, we find
\[
\tfrac{1}{\tilde{H}_0} x^2 (1 + x_0^2)^{\frac{n-1}{2}} z_x \geq (1+x_0^2)^{\frac{n-1}{2}} - 1 > 0,
\]
so $z_x > 0$ on $(0,\infty)$ due to $n \geq 3$.
At $t_h$, we have $f'(t_h) = t_h^{-1} f(t_h) > 0$, so $u'(t_h) > 0$ and $x'(t_h)>0$, while identity~\eqref{eqn:shared-identities-axisymmetric} shows that $\theta = \textup{arg}(t + i f)$ increases in $(t_1, t_h)$ and decreases on $(t_h, t_2)$.
Therefore, $\frac{d \theta}{dx} > 0$ on the left branch and $\frac{d \theta}{dx} < 0$ on the right branch near $t_h$, whereby 
\[
\left( \frac{d \theta}{dx} \right)^2 = \frac{z^2}{x^2 (1+x^2) (1 - z^2)} \quad \implies \quad \frac{d \theta}{dx} = - \frac{z(x)}{x \sqrt{1+x^2} \sqrt{1-z(x)^2}}
\]
upon expressing $z = z(x)$ via~\eqref{eqn:D(u)-from-x}.
The right-hand side of this equation depends only on $x$, so the first-order ODE for $\theta$ with initial value $\theta(x_0)= \theta(t_h)$ determines $\theta$ as a single valued function of $x$; the same property holds on the increasing branch.
Note that $u = \frac{t^2 + f^2}{1+f^2}$ maps $[t_1, t_2]$ into a compact subset of $(0,1)$ because $[t_1,t_2] \subset (0,1)$ imply that $u$ is bounded away from $\{ 0,1\}$, hence $x = \sqrt{\frac{u}{1-u}}$ is bounded away from $\{ 0,\infty\}$.
Moreover, the identity~\eqref{eqn:shared-energy-identity-axisymmetric} immediately shows that $|z| \leq 1$, where equality cannot occur in this situation, so $|z| < 1$ and $x$ cannot assume the same value twice on a branch.
We have then that $x$ is monotone and $\frac{d \theta}{dx}$ is well-defined.

Geometrically, the function $\theta = \textup{arg}(t + if)$ represents the polar angle, so $\theta(t_1) = \theta(t_2) =0$.
For $\lambda \in (0,1)$, we define $x_-(\lambda) < x_0 = x(t_h) < x_+ ( \lambda)$ be the unique points with $z( x_{\pm}(\lambda)) = \pm \lambda$ and let
\[
\Theta_{\pm}(\lambda) := \int_{x_0}^{x_{\pm}(\lambda)} \frac{z(x)}{x \sqrt{1+x^2} \sqrt{1 - z(x)^2}} \, dx, \qquad \Theta'_{\pm}(\lambda) = \frac{\lambda}{\sqrt{1-\lambda^2}} \frac{1}{x_{\pm}(\lambda) \sqrt{1 + x_{\pm}(\lambda)^2} \, z_x(x_{\pm}(\lambda))},
\]
by using $z$ as a coordinate due to $z_x > 0$.

We now claim that $\Theta_-(\lambda) > \Theta_+(\lambda)$ for all $\lambda \in (0,1]$.
Since $\Theta_{\pm}(\lambda) \to 0$ as $\lambda \downarrow 0$, this property will follow from proving $\Theta'_-(\lambda) > \Theta'_+(\lambda)$ for every $\lambda \in (0,1)$.
Let $v_{\pm} := \sqrt{1 + x_{\pm}(\lambda)^2}$ for brevity, so using $z(x_{\pm}) = \pm \lambda$ in the explicit $z$-expression~\eqref{eqn:D(u)-z-properties} produces
\begin{align*}
\frac{v_{\pm}^{n-1}}{(1 + x_0^2)^{\frac{n-1}{2}}} &= 1 \pm \frac{\lambda}{\tilde{H}_0} x_{\pm} \implies \frac{\lambda}{\tilde{H}_0} = \frac{v_+^{n-1} - v_-^{n-1}}{x_+ v_-^{n-1} + x_- v_+^{n-1}}, \\
\sqrt{\lambda^{-2}-1} \Theta_{\pm}'(\lambda) &= \frac{1}{x_{\pm} v_{\pm} z_x(x_{\pm})} = \frac{1}{\tilde{H}_0} \frac{v_{\pm}}{(n-1) x_{\pm} \mp \frac{\lambda}{\tilde{H}_0} \pm (n-2) \frac{\lambda}{\tilde{H}_0} x_{\pm}^2}  ,
\end{align*}
where the latter property comes from evaluating the expression $x(1+x^2) z_x = \bigl( (n-2) x^2 - 1 \bigr) z + H_0 x$ at $x = x_{\pm}$.
Substituting the expression for $\frac{\lambda}{\tilde{H}_0}$ computed above, we obtain
\begin{align*}
     &\tilde{H}_0 \sqrt{\lambda^{-2}-1} \bigl( \Theta'_-(\lambda) - \Theta'_+(\lambda) \bigr) \\
    &= \frac{(x_+ v_+^{n-1} + x_- v_-^{n-1}) \bigl[ (n-1) (x_+ + x_-) (x_+ v_+^{n-2} - x_- v_-^{n-2}) - (v_+ + v_-)(v_+^{n-1} - v_-^{n-1}) \bigr]}{ \big[ (n-1) (x_+ + x_+) x_- v_-^{n-2} + v_- ( v_+^{n-1} - v_-^{n-1}) \bigr] \cdot \big[(n-1) (x_+ + x_+) x_+ v_+^{n-2} - v_+ ( v_+^{n-1} - v_-^{n-1}) \bigr]}.
\end{align*}
The claimed positivity is therefore equivalent to the positivity of the numerator in brackets.
Since $v_{\pm} = \sqrt{1 + x_{\pm}^2}$, we have $v_+ - v_- = \frac{x_+^2 - x_-^2}{v_+ + v_-} = \frac{(x_+ - x_-)(x_+ + x_-)}{v_+ + v_-}$, so the above positivity becomes
\[
(x_+ v_+^{n-2} - x_- v_-^{n-2}) > (x_+ - x_-) v_+^{n-2} >  (x_+ - x_-) \cdot \tfrac{1}{n-1} \tfrac{v_+^{n-1} - v_-^{n-1}}{v_+ - v_-} = (x_+ - x_-) \xi^{n-2}
\]
by the mean value theorem, where $\xi \in (v_- , v_+)$ due to $x_- < x_+$ and $v_- < v_+$.
Therefore, $\Theta'_- (\lambda)  > \Theta'_+(\lambda)$ and $\Theta_- > \Theta_+$ for $\lambda \in (0,1)$.
To complete the proof as $\lambda \uparrow 1$, note that $\lambda_{\pm} = 1$ makes $z(x_{\pm}) = \pm 1$, hence $(1-z^2)'(x_{\pm})= - 2 z(x_{\pm})z_x(x_{\pm}) \neq 0$ means that $1 - z(x_{\pm})^2$ has a simple zero at $x_{\pm}$, due to $z_x > 0$.
The $\Theta_{\pm}$ integrands have an integrable singularity $|x - x_{\pm}|^{-\frac{1}{2}}$ near those points, hence they remain finite as $\lambda \uparrow 1$ and $\Theta_-(1) > \Theta_+(1)$ as well.

Returning to the integral identity for the polar angle $\theta(t_h)$, we have
\[
\theta(t_h) = \int_{x_0}^{x(t_1)} \frac{z(x)}{x \sqrt{1+x^2} \sqrt{1-z(x)^2}} \, dx = \int_{x_0}^{x(t_2)} \frac{z(x)}{x \sqrt{1+x^2} \sqrt{1-z(x)^2}} \, dx
\]
due to $\theta(t_1) = \theta(t_2) = 0$.
Moreover, the level $\lambda_* := |z(x(t_i))| = \frac{(1-t_i^2) |f'(t_i)|}{\sqrt{1 + (1-t_i^2) f'(t_i)^2}}$ is well-defined for $i=1,2$ due to $(1-t_1^2) f'(t_1)^2 = (1 - t_2^2) f'(t_2)^2$, with $\lambda_* \in (0,1]$.
Consequently,
\[
\theta(t_h) = \Theta_-(\lambda_*) > \Theta_+(\lambda_*) = \theta(t_h)
\]
for arbitrary $\lambda_* \in (0,1]$, producing a contradiction.
We conclude that no such solutions exist.
\end{proof}

This rules out capillary Type II solutions, so now we show that any ODE solution to~\eqref{eqn:axisymmetric-equation} on an interval containing $0$ is even.

\begin{proposition}\label{prop:axisymmetric-is-symmetric}
    Let $f(t)$ be a positive solution of equation~\eqref{eqn:axisymmetric-equation} on an interval $(t_1, t_2)$ with $f(t_1) = f(t_2) = 0$.
    If $(1 - t_1^2) f'(t_1)^2 = (1 - t_2^2) f'(t_2)^2$, then $f'(0) = 0$ and $f$ is an even function.
\end{proposition}
\begin{proof}
Applying Lemma~\ref{lemma:axisymmetric-same-side-angle-inequality}, we must have $t_1 \in (-1,0)$ and $t_2 \in (0,1)$.
We study the ODE~\eqref{eqn:shared-autonomous-u-theta}, which is an autonomous equation for $u(\theta)$. 
Computing $u'= 2\frac{t(1+f^2)+(1-t^2)ff'}{(1+f^2)^2}$ shows $u'(t_i) = 2t_i$, so $u$ must have an interior minimum at some $t_0\in(t_1,t_2)$, which we may assume satisfies $t_0\geq 0$, otherwise we replace $f(t)$ by $f(-t)$. 
Let $u_0:=u(t_0)$ and $\theta_0:=\on{arg}(t_0+if(t_0)) = \arctan \frac{f(t_0)}{t_0}$, so $\theta_0\leq \frac{\pi}{2}$ with equality if and only if $t_0 = 0$.
We will prove that $u$, and therefore $f$, is symmetric about $t=0$ by forcing $\theta_0 = \frac{\pi}{2}$.

Using $f(t_0) - t_0 f'(t_0) = \frac{t_0^2+f(t_0)^2}{(1-t_0^2)f(t_0)}>0$ and $u'(t_0) = 0$, we find $D(u_0) > 0$ and $u_0 = (1-u_0) D(u_0)^2$.
If $u$ is constant, then $f(t)=\sqrt{\frac{u_0-t^2}{1-u_0}}$, so $f'(0)=0$ and $f$ is even.
Otherwise,~\eqref{eqn:shared-identities-axisymmetric} and~\eqref{eqn:D(u)-quantity} show that $D(u_0) > 0$ and $\frac{d \theta}{d t}(t_0) = - \frac{h(t_0)}{t_0^2 + f(t_0)^2} < 0$, so $\theta$ is a local parameter near $t_0$ and we define $u_\pm(\tau):=u(\theta_0\pm\tau)$ for $\tau \geq 0$.
Since $u$ is minimized at $\theta_0$, both branches are increasing for $\tau>0$ and satisfy \eqref{eqn:shared-autonomous-u-theta} with the positive sign. 
Moreover, the function $G:u \mapsto \int_{u_0}^u \frac{D(s)}{2s \sqrt{s - (1-s) D(s)^2}} ds$ is strictly increasing on the relevant branch, and $G(u_{\pm}(\tau)) = \tau$ for both $u_{\pm}$, hence $u_+=u_-$ on their common domain.
Let $U(\tau)$ denote the common increasing branch, so $D(U(\tau)) > 0$ and $f-tf'>0$ along the solution, thus~\eqref{eqn:D(u)-quantity}, \eqref{eqn:h-ODE-axisymmetric} show that $\theta$ is a global coordinate along $f$ that decreases from $\pi$ to $0$, with endpoints corresponding to the values $\theta = \pi - \theta_0$ and $\theta = \theta_0$.
In particular, we must have $\tau_{\max} \geq \max \{ \theta_0, \pi - \theta_0 \} \geq \frac{\pi}{2}$.
For $x = \sqrt{\frac{U}{1-U}}$ and $w := \sqrt{1-z^2} = \sqrt{1 - \frac{D(U)^2}{x^2}}$ as in~\eqref{eqn:introduce-x,z}, we have $x(0) = \sqrt{\frac{u_0}{1 - u_0}} = D(u_0) > 0$ and $w(0) = 0$, while $\tau(t_1) = \pi - \theta_0$ and $\tau(t_2) = \theta_0$, so computing as in Lemma~\ref{lemma:axisymmetric-same-side-angle-inequality} makes the contact angle condition for $(1-t_i^2) f'(t_i)^2$ equivalent to $w(\pi - \theta_0) = w(\theta_0)$.
Denoting derivatives in $\tau$ by primes, we now compute
\begin{align*}
& x' = \frac{x \sqrt{1+x^2}}{\sqrt{1-w^2}} w, \qquad x(1+x^2) \Bigl( \frac{\partial }{\partial x} \sqrt{1-w^2} \Bigr) = \bigl( (n-2) x^2 - 1 \bigr) \sqrt{1-w^2} + H_0 x, \\
& \quad \implies w'=-\frac{((n-2)x^2-1) \sqrt{1-w^2} +H_0x}{\sqrt{1+x^2}}.
\end{align*}
Differentiating again, we conclude that
\begin{equation}\label{eqn:w-double-prime-ODE}
    w''+w=-(n-2)(n-1) \, \cW \, \frac{x^3(1+x^2)^{\frac{n-3}{2}}}{\sqrt{1-w^2}}\,w.
\end{equation}
We first rule out the case $\cW \leq 0$.
Since $D(u_0)=x(0)$, we have
\[
\cW = \left(x(0)+\frac{H_0}{n-1}\right)(1+x(0)^2)^{-\frac{n-1}{2}} \implies \on{sgn} \cW = \on{sgn} \Bigl( x(0) + \frac{H_0}{n-1} \Bigr).
\]
Therefore, $\cW \leq 0$ forces $x(0) + \frac{H_0}{n-1} \leq 0$, so $x \geq x(0)$ implies that
\[
D(U) = \cW (1 +x^2)^{\frac{n-1}{2}} - \tfrac{H_0}{n-1} \leq \cW (1 + x(0)^2)^{\frac{n-1}{2}} - \tfrac{H_0}{n-1} = D(u_0) = x(0).
\]
As a result, we obtain $w \geq \sqrt{1 - \frac{x(0)^2}{x^2}}$ and $\frac{\sqrt{1-w^2}}{w} \leq \frac{x(0)}{\sqrt{x^2 - x(0)^2}}$ for all $\tau$.
We conclude that
\[
\frac{d \tau}{dx} = \frac{\sqrt{1-w^2}}{x \sqrt{1+x^2} \, w} \leq \frac{x(0)}{x \sqrt{1+x^2} \sqrt{x^2 - x(0)^2}}.
\]
Integrating this inequality from $x(0)$ to $\infty$ shows
\[
\tau_{\mr{max}} \leq \int_{x(0)}^\infty \frac{x(0)}{x\sqrt{1 + x^2}\sqrt{x^2 - x(0)^2}}\,dx \leq \mr{arctan}\frac{1}{x(0)} < \frac{\pi}{2},
\]
producing a contradiction, thus proving $\cW>0$.

Returning to equation~\eqref{eqn:w-double-prime-ODE}, we use $\cW,w>0$ to see that $w'' + w < 0$.
In particular, this implies
\[
\bigl( w'(\tau) \sin \tau - w(\tau) \cos \tau \bigr)' = (w'' + w) \sin \tau < 0, \qquad \bigl[ w'(\tau) \sin \tau - w(\tau) \cos \tau \bigr]_{\tau=0} = 0,
\]
so $\bigl( \frac{w(\tau)}{\sin \tau} \bigr)' = \frac{w'(\tau) \sin \tau - w(\tau) \cos \tau}{\sin^2 \tau} < 0$ for $\tau>0$.
Since this function is strictly decreasing on $(0, \tau_{\max})$, we obtain $\frac{w(\pi - \theta)}{\sin (\pi - \theta)} < \frac{w(\theta)}{\sin \theta}$ strictly for $\theta \in (0,\frac{\pi}{2})$, so $w( \pi - \theta_0) = w(\theta_0)$ forces $\theta_0 = \frac{\pi}{2}$ and $t_0 = 0$.
The identity $t_0(1+f(t_0)^2)+(1-t_0^2)f(t_0)f'(t_0)=0$ and the fact that $t_0 = 0$ shows $f'(0)=0$, so $f$ is even by ODE uniqueness, proving the claim.
\end{proof}

\begin{proof}[Proof of Theorem~\ref{thm:uniqueness-of-axisymmetric}]
Applying the reductions considered above, the study of $\textup{O}(n-1)$-invariant cones is equivalent to finding solutions of the axisymmetric equation~\eqref{eqn:axisymmetric-equation} with $f(t_1) = f(t_2) = 0$ and $(1 - t_1^2) f'(t_1)^2 = (1 - t_2^2) f'(t_2)^2$.
Proposition~\ref{prop:axisymmetric-is-symmetric} implies that any such function $f$ must be even, so $f'(0) = 0$.
Therefore, $f$ is determined uniquely by the value $a =f(0)$.
Denoting the solution with data $\{ f(0) = a, f'(0) = 0 \}$ by $f_a$ with a zero at $t_a$,~\cite{FTW-1}*{Theorem 1.2} shows that $f_a$ reaches a zero if and only if $a \leq \sqrt{\frac{1}{n-2}}$ and the map $a \mapsto \sqrt{1-t_a^2} \, |f'_a(t_a)|$ is strictly increasing on $[0, \frac{1}{\sqrt{n-2}}]$, producing a unique profile curve for each value of $\theta_a := \arctan ( \sqrt{1-t_a^2} |f'_a(t_a)|) \in (0, \frac{\pi}{2}]$.
Thus, the cone $\mathbf{C}_{n,1,\theta}$ of~\cite{FTW-1} is the unique axisymmetric capillary minimal cone of angle $\theta$.
\end{proof}

\bibliography{ref}
\end{document}